\documentclass[a4paper]{amsart}
\usepackage{amssymb,latexsym}

\usepackage[active]{srcltx}
\usepackage{amsmath}
\usepackage{amsfonts}

\usepackage{amsthm}
\usepackage{amscd}
\usepackage{mathrsfs}
\usepackage{mathtools}
\usepackage{fancyhdr}
\usepackage{graphicx}
\usepackage{calc}
\usepackage{url}

\usepackage{times}
\usepackage{verbatim}
\usepackage[cmtip,arrow,matrix,curve,tips,frame]{xy}
\usepackage{hyperref}
\xyoption{matrix}

\usepackage{ifthen}
\usepackage{epsfig}
\usepackage{color}
\usepackage{amsxtra}
\usepackage{amstext}
\usepackage{amssymb}
\usepackage{stmaryrd}
\usepackage{mathrsfs}
\usepackage{pifont}
\usepackage{charter}
\usepackage{avant}
\usepackage{courier}
\usepackage[T1]{fontenc}
\usepackage{textcomp}
\usepackage{bbm}

\usepackage{tikz} 
\usetikzlibrary{matrix,arrows,calc,decorations.pathreplacing,fit}
\usepackage{tikz-cd}

\numberwithin{equation}{subsection}

\theoremstyle{plain}
\newtheorem{thm}[subsection]{Theorem}
\newtheorem*{thm*}{Theorem}
\newtheorem{nothm}{Theorem}
\newtheorem{thmsub}[equation]{Theorem}

\newtheorem*{exthm*}{Expected Theorem}

\newtheorem{lem}[subsection]{Lemma}
\newtheorem{lemsub}[equation]{Lemma}
\newtheorem*{lem*}{Lemma}

\newtheorem{prop}[subsection]{Proposition}

\newtheorem{propsub}[equation]{Proposition}
\newtheorem*{prop*}{Proposition}
\newtheorem{noprop}[nothm]{Proposition}

\newtheorem{cor}[subsection]{Corollary}
\newtheorem{corsub}[equation]{Corollary}
\newtheorem*{cor*}{Corollary}

\newtheorem*{claim*}{Claim}

\newtheorem*{conj*}{Conjecture}

\theoremstyle{definition}

\newtheorem*{defn*}{Definition}
\newtheorem{defnsub}[equation]{Definition}

\newtheorem{deflemsub}[equation]{Definition/Lemma}
\newtheorem{deflem}[subsection]{Definition/Lemma}

\newtheorem{axiom}{Axiom}

\newtheorem{exasub}[equation]{Example}
\newtheorem*{exa*}{Example}

\theoremstyle{remark}
\newtheorem{rmk}[subsection]{Remark}
\newtheorem{rmksub}[equation]{Remark}
\newtheorem*{rmk*}{Remark}

\numberwithin{figure}{subsection}
\numberwithin{table}{subsection}

\newcounter{listnum}
\newenvironment{numlist}
   {
      \begin{list}
            {{(\thelistnum)}}{
            \usecounter{listnum}
            \setlength{\itemsep}{.5ex}
            \setlength{\parsep}{0ex}
            \setlength{\parskip}{0ex}
            \setlength{\topsep}{.5ex}
         }
   }
   {
      \end{list}
   }

\newcounter{asslistcounter}
\newenvironment{assertionlist}{
 \begin{list}
  {\upshape (\alph{asslistcounter})}
  {\setlength{\leftmargin}{18pt}
   \setlength{\rightmargin}{0pt}
   \setlength{\itemindent}{0pt}
   \setlength{\labelsep}{5pt}
   \setlength{\labelwidth}{13pt}
   \setlength{\listparindent}{\parindent}
   \setlength{\parsep}{0pt}
   \setlength{\itemsep}{0pt}
   \setlength{\topsep}{-.5\parskip}
   \usecounter{asslistcounter}}}
  {\end{list}}

\newcounter{subenvcounter}
\newenvironment{subenv}{%
 \begin{list}
  {\em (\arabic{subenvcounter})}
  {\setlength{\leftmargin}{20pt}
   \setlength{\rightmargin}{0pt}
   \setlength{\itemindent}{0pt}
   \setlength{\labelsep}{5pt}
   \setlength{\labelwidth}{13pt}
   \setlength{\listparindent}{\parindent}
   \setlength{\parsep}{0pt}
   \setlength{\itemsep}{0pt}
   \setlength{\topsep}{-\parskip}
   \usecounter{subenvcounter}}}
  {\end{list}}

\DeclareMathOperator{\Adm}{Adm}

\DeclareMathOperator{\cInd}{c-Ind}

\DeclareMathOperator{\End}{End}
\DeclareMathOperator{\Fil}{Fil}

\DeclareMathOperator{\Frac}{Frac}

\DeclareMathOperator{\id}{id}
\DeclareMathOperator{\image}{im}

\DeclareMathOperator{\Gal}{Gal}
\DeclareMathOperator{\GL}{GL}
\DeclareMathOperator{\GSp}{GSp}

\DeclareMathOperator{\Hom}{Hom}
\DeclareMathOperator{\Ind}{Ind}

\DeclareMathOperator{\Isom}{Isom}
\DeclareMathOperator{\Lie}{Lie}
\DeclareMathOperator{\Ig}{Ig}
\DeclareMathOperator{\Map}{Map}

\DeclareMathOperator{\M}{M}

\DeclareMathOperator{\pr}{pr}
\DeclareMathOperator{\QIsog}{QIsog}
\DeclareMathOperator{\Rep}{Rep}

\newcommand{\Sh}{\mathsf{Sh}}

\DeclareMathOperator{\Spec}{Spec}
\DeclareMathOperator{\Spf}{Spf}

\DeclareMathOperator{\Tor}{Tor}
\DeclareMathOperator{\trace}{Tr}
\DeclareMathOperator{\Res}{Res}

\def \AA {\mathbb{A}}
\def \BB {\mathbb{B}}
\def \CC {\mathbb{C}}
\def \DD {\mathbb{D}}

\def \FF {\mathbb{F}}
\def \GG {\mathbb{G}}

\def \II {\mathbb{I}}

\def \NN {\mathbb{N}}

\def \PP {\mathbb{P}}
\def \QQ {\mathbb{Q}}

\def \WW {\mathbb{W}}
\def \XX {\mathbb{X}}

\def \ZZ {\mathbb{Z}}

\def \Acal {\mathcal{A}}
\def \Bcal {\mathcal{B}}
\def \Ccal {\mathcal{C}}
\def \Dcal {\mathcal{D}}
\def \Ecal {\mathcal{E}}
\def \Fcal {\mathcal{F}}
\def \Gcal {\mathcal{G}}
\def \Hcal {\mathcal{H}}

\def \Lcal {\mathcal{L}}
\def \Mcal {\mathcal{M}}

\def \Ocal {\mathcal{O}}
\def \Pcal {\mathcal{P}}

\def \Scal {\mathcal{S}}

\def \Ucal {\mathcal{U}}

\def \Xcal {\mathcal{X}}
\def \Ycal {\mathcal{Y}}
\def \Zcal {\mathcal{Z}}

\def \Gfr {\mathfrak{G}}

\def \Mfr {\mathfrak{M}}

\def \Sfr {\mathfrak{S}}

\def \Ufr {\mathfrak{U}}
\def \Vfr {\mathfrak{V}}

\def \Xfr {\mathfrak{X}}
\def \Yfr {\mathfrak{Y}}

\def \mfr {\mathfrak{m}}

\def \pfr {\mathfrak{p}}

\def \Ascr {\mathscr{A}}

\def \Escr {\mathscr{E}}
\def \Fscr {\mathscr{F}}
\def \Gscr {\mathscr{G}}

\def \Lscr {\mathscr{L}}
\def \Mscr {\mathscr{M}}

\def \Oscr {\mathscr{O}}

\def \Sscr {\mathscr{S}}

\def \Xscr {\mathscr{X}}

\def \Cbar {\bar{C}}

\def \hbar {\bar{h}}

\def \xbar {\bar{x}}
\def \ybar {\bar{y}}
\def \zbar {\bar{z}}

\def \Wtilde {\tilde{W}}
\def \Xtilde {\tilde{X}}

\def \wtilde {\tilde{w}}

\def \ztilde {\tilde{z}}

\def \Dbf {\mathbf{D}}

\def \Gsf {\mathbf{G}}

\def \Asf  {\mathsf{A}}

\def \Esf  {\mathsf{E}}

\def \Gsf  {\mathsf{G}}

\def \Ksf  {\mathsf{K}}

\def \Ssf  {\mathsf{S}}

\def \Xsf  {\mathsf{X}}

\def \FFbar {\overline{\mathbb{F}}}

\def \QQbar {\overline{\mathbb{Q}}}

\def \BT    {{Barsotti-Tate group}}

\def \lto   {\longrightarrow}
\def \mono  {\hookrightarrow}

\def \epi   {\twoheadrightarrow}

\def \isom  {\stackrel{\sim}{\rightarrow}}

\newcommand{\pot}[1]{ [\hspace{-0,17em}[ {#1} ]\hspace{-0,17em}] }

\newcommand{\restr}[2]{{#1}\raise-.5ex\hbox{\ensuremath|}_{#2}}

\newcommand{\bigslant}[2]{{\raisebox{.2em}{$#1$}\hspace{-.3em}\left/ \hspace{-.2em}\raisebox{-.2em}{$#2$}\right.}}

\def \univ {{\rm univ}}
\def \Flag {{\mathscr{F}\ell}}
\def \Grass {{\rm Gr}}
\def \perf {{\rm perf}}
\def \Def {{\mathfrak{Def}}}

\DeclareMathOperator{\RHom}{RHom}

\def \ad {\mathrm{ad}}

\newcommand{\loc}{\mathrm{loc}}

\newcommand{\red}{\mathrm{red}}

\newcommand{\coh}[1]{\mathrm{H}^{#1}}

\newcommand{\Lrm}{\mathrm{L}}
\newcommand{\Rrm}{\mathrm{R}}

\newcommand{\ol}[1]{\overline{#1}}
\newcommand{\wt}[1]{\widetilde{#1}}
\newcommand{\wh}[1]{\widehat{#1}}
\newcommand{\iv}{^{-1}}

\newcommand{\ra}{\rightarrow}

\newcommand{\hra}{\hookrightarrow}
\newcommand{\thra}{\twoheadrightarrow}
\newcommand{\riso}{\xrightarrow{\sim}}

\newcommand{\et}{\text{\rm\'et}}

\newcommand{\qcet}{\text{\rm qc\'et}}

\newcommand{\dR}{\text{\rm dR}}

\newcommand{\Dr}{\mathrm{Dr}}
\DeclareMathOperator{\Ext}{Ext}
\DeclareMathOperator{\Mant}{Mant}
\DeclareMathOperator{\Mod}{Mod}
\DeclareMathOperator{\Spa}{Spa}

\newcommand{\spbf}{\mathbf{sp}}

\newcommand{\Fpbar}{\overline{\FF}_p}

\newcommand{\Zp}{\ZZ_p}
\newcommand{\Zpbr}{\breve\ZZ_p}

\newcommand{\Qp}{\QQ_p}

\newcommand{\Qpbar}{\overline{\QQ}_p}

\newcommand{\Acris}{A_{\rm cris}}
\newcommand{\Bcris}{B_{\rm cris}}
\newcommand{\cris}{\mathrm{cris}}
\usepackage[british]{babel}
\author[P.~Hamacher, W.~Kim]{Paul Hamacher and Wansu Kim}

\def \k {{\bar{\mathbb{F}}}_p}
\def \F {{\mathbb{Q}_p}}
\def \O {{\mathbb{Z}_p}}
\def \W {{\breve{\ZZ}_p}}
\renewcommand{\L}{\breve{\QQ}_p}
\def \Igfr {{\mathfrak{Ig}}}
\def \Ssb {{\bar{\mathscr{S}}}}

\def \Ebreve {{\breve{E}}}

\def \hyph {\text{-}}
\def \disc {\mathrm{disc}}
\def \ft {\mathrm{ft}}

\address{Paul Hamacher\\
Technische Universit\"at M\"unchen\\
Zentrum Mathematik - M 11\
Boltzmannstra{\ss}e 3\\
85748 Garching\\
Deutschland\\ \newline
\indent Wansu Kim\\%
Department of Mathematical Sciences\\
Korea Advanced Institute of Science and Technology (KAIST)\\
291 Daehak-ro\\
Yuseong-gu\\ 
Daejeon 34141\\
South Korea}
\email{hamacher@ma.tum.de, wansu.math@gmail.com}

\title[]{$l$-adic \'etale cohomology of Shimura varieties of Hodge type with non-trivial coefficients}

\begin{document}

\begin{abstract}
  Let $(\Gsf,\Xsf)$ be a Shimura datum of Hodge type. Let $p$ be an odd prime such that $\Gsf_{\QQ_p}$ splits after a tamely ramified extension and $p\nmid |\pi_1(\Gsf^{\rm der})|$. Under some mild additional assumptions that are satisfied if the associated Shimura variety is proper and $\Gsf_{\QQ_p}$ is either unramified or residually split, we prove the generalisation of Mantovan's formula for the $l$-adic cohomology of the associated Shimura variety. On the way we derive some new results about the geometry of the Newton stratification of the reduction modulo $p$ of the Kisin-Pappas integral model.
\end{abstract}

\maketitle

\tableofcontents

\section*{Introduction}

The aim of this paper is to obtain the generalisation of Mantovan's formula for the cohomology of Hodge-type Shimura varieties, under some mild assumptions. %
To explain the motivation, it is a folklore belief that the cohomology of Shimura varieties should decompose in terms of certain automorphic representations and their Langlands parameters. Similarly, we expect that there should exist a local analogue of Shimura varieties whose cohomology decomposes according to the local Langlands and Jacquet-Langlands correspondences; \emph{cf.} \cite{RapoportViehmann:LocShVar}. It is natural to expect that there should be a relationship between the cohomology of Shimura varieties and ``local Shimura varieties'' that encodes the local-global compatibility of the Langlands correspondence.
For compact Shimura varieties associated to PEL Shimura datum unramified at a chosen prime $p$, such a cohomological formula was obtained by Mantovan \cite{Mantovan:UnitaryShVar,Mantovan:Foliation,Mantovan:PELnon-trivCoeff} built upon the work of Harris and Taylor on the special case of Shimura varieties of Harris-Taylor type \cite[IV, Theorem~2.7]{HarrisTaylor:TheBook}.\footnote{Variants  of such a formula were also studied in \cite{Rapoport:GuideRednShVar},\cite{Fargues:AsterisqueLLC}. Readers are referred to the introduction of \cite{HarrisTaylor:TheBook} for the works prior to \emph{loc.~cit.}} %

Let us now briefly explain the aforementioned result of Harris-Taylor and Man\-to\-van.
Let $(\Gsf,\Xsf)$ be a Shimura datum associated to a PEL datum, and assume that the associated Shimura variety is compact. We choose a prime $\pfr$ of the Shimura field $\Esf$ over a prime $p$ where $\Gsf$ is unramified. Then Mantovan's formula (for the constant coefficients) expresses the alternating sum the ``infinite-level'' Shimura variety  $\sum_i (-1)^i\coh i(\Sh(\Gsf,\Xsf),\overline\QQ_l) $, as a virtual representation of $\Gsf(\AA_f)\times W_{\Esf_\pfr}$, in terms of the compactly supported cohomology of the following spaces
\begin{enumerate}
\item  rigid analytic towers associated to PEL Rapoport-Zink spaces, playing the role of ``local Shimura varieties'';
\item  Igusa towers, which is a tower of smooth varieties in characteristic~$p$.
\end{enumerate}

The cohomology of (the rigid analytic towers of) PEL Rapoport-Zink spaces carries purely local information, while the cohomology of Igusa towers carries global information. Mantovan's formula, which will be precisely stated later on, expresses the cohomology of Shimura varieties in a way that nicely separates the purely local contribution from the remaining global information.

One of the key ingredients of the proof is the study of the geometry of the Newton stratification of the mod~$p$ reduction of Shimura varieties with hyperspecial\footnote{This is where we use the assumption that $\Gsf$ is unramified at $p$.} level at $p$; namely, the product of Igusa tower and the underlying reduced subscheme of PEL Rapoport-Zink spaces forms a nice ``pro-\'etale cover up to perfection''. This property is often called the \emph{almost product structure} of Newton strata.

Note that in the PEL case Mantovan's formula can be generalised for non-constant automorphic \'etale sheaves, instead of constant sheaf $\overline{\QQ}_l$; \emph{cf.} \cite{Mantovan:PELnon-trivCoeff}. Also we may drop the compactness assumption and allow non-compact PEL Shimura varieties; \emph{cf.} \cite[\S6]{LanStroh:NearbyCycles}. Indeed,  compactness was assumed so that  the nearby cycles spectral sequence should hold for the integral models of Shimura varieties, but it was proved in \emph{loc.~cit.} that the nearby cycles spectral sequence holds if the integral models have ``good compactifications''.

Recently, basic definitions and results on mixed characteristic PEL Shimura varieties have been generalised for Hodge-type Shimura varieties; for example, integral canonical models at hyperspecial level at $p$ were constructed \cite{Kisin:IntModelAbType,KimMadapusiPera:2adicICM}, and they were shown to have good compactifications \cite{MadapusiPera:Compactification}. Furthermore, Kottwitz' description of mod~$p$ points of PEL Shimura varieties with hyperspecial level at $p$ can be generalised to Hodge-type Shimura varieties \cite{Kisin:LanglandsRapoport}. Also the local analogue of Hodge-type Shimura varieties (or Hodge-type generalisation of PEL Rapoport-Zink spaces) was obtained in \cite{Kim:RZ} in the unramified case (or alternatively, \cite{HowardPappas:GSpin}), and we also have the almost product structure of the Newton strata in Hodge-type Shimura varieties with hyperspecial level at $p$ \cite{Hamacher:ShVarProdStr}.

Kisin and Pappas have constructed integral models of Hodge-type Shimura varieties with parahoric\footnote{Instead of parahoric level structure, we will work with the ``Bruhat-Tits level structure'' for technical reasons; in other words, the level at $p$ is given by the \emph{full} stabiliser in $\Gsf(\QQ_p)$ of some facet in the Bruhat-Tits building, not necessarily by the connected stabilsier (which is a parahoric group). Note that integral models with the ``Bruhat-Tits level structure'' is obtained as the normalisation of the closure of the generic fibre in some Siegel modular variety.} level structure  at $p>2$, if $\Gsf_{\QQ_p}$ splits after some tamely ramified extension of $\QQ_p$. It is natural, therefore, to ask whether one can prove Mantovan's formula whenever the Shimura variety admits a Kisin-Pappas integral model.

In this paper, we use the following setup. We fix a Shimura datum of Hodge type $(\Gsf,\Xsf)$. We denote by $\{\mu\}$ the conjugacy class of cocharacters defined by $\Xsf$, and let $\Esf$ be its field of definition. Let $p$ be a rational prime and $\pfr | p$ a prime of $\Esf$. Assume that $\Gsf_{\QQ_p}$ splits over a tamely ramified extension and denote by $\Sscr$ a Kisin-Pappas integral model of $\Sh_\Ksf(\Gsf,\Xsf)$ over the ring of integers $O_E$ of $p$-adic completion $E$ of $\Esf$. In particular we assume $\Ksf=\Ksf^p\Ksf_p$ and $\Ksf_p\subset\Gsf(\QQ_p)$  is assumed to be a stabiliser of some facet in the Bruhat-Tits building. We denote by $\Gscr_x$ the associated Bruhat-Tits integral model of $\Gsf$. By construction, $\Sscr$ comes with a finite morphism $\rho\colon \Sscr \to \Sscr'$ into some Siegel moduli space and thus with a principally polarised abelian variety $\Ascr_\Gsf \to \Sscr$. In order to prove Mantovan's formula for $\Sh(\Gsf,\Xsf)$, we assume that the following assertions hold (see \S~\ref{sect ADLV} or \S~\ref{sect kuenneth} for the precise formulations.)

\begin{axiom} \label{axiom RZ intro}
 Let $b \in B(\Gsf_{\QQ_p},\mu)$ such that the Newton stratum in the special fibre $\Ssb$ of $\Sscr$ is nonempty. The Rapoport-Zink uniformisation map $\Theta'\colon \Mfr^{\rho(b)} \to \Sscr'$ has a unique lift on $\FFbar_p$-points
 \[
  f\colon X_\mu(b)_{\Ksf_p}(\FFbar_p) \to \Sscr(\FFbar_p)
 \]
 such that for any $P \in X_\mu(b)_{\Ksf_p}(\FFbar_p)$ the isomorphism of $\Ascr_{\Gsf,f(P)}[p^\infty]$ with the principally polarised {\BT} associated to $P$ respects the canonical crystalline Tate tensors on both sides.
\end{axiom}

\begin{axiom} \label{axiom:LanStroh intro}
Let $\Ksf=\Ksf^p\Ksf_p$ be as above, and let $\Sh_m$ denote the Shimura variety with the same prime-to-$p$ level $\Ksf^p$ and ``full level $p^m$ structure'' (in the sense of  subsection~\ref{ssect Drinfeld}).
 Denote by $\Sscr_m$ the normalisation of $\Sscr$ in $\Sh_m$, and by $\Ssb_m$ its special fibre. Then for any automorphic \'etale sheaf $\Lscr_\xi$, the canonical morphism
 \[
  \coh i_c(\Ssb_{m},\Rrm\Psi\Lscr_\xi) \to \coh i_c(\Sh_{m},\Lscr_\xi)
 \]
 is an isomorphism.
\end{axiom}

While not yet proven in full generality, it is reasonable to expect that these assertions will be proven in the near future. 

Axiom A was proven in the case of hyperspecial level structure by Kisin in \cite[Thm.~1.4.4]{Kisin:LanglandsRapoport}, and in the case when $\Gsf_{\QQ_p}$ is residually split (for any parahoric level structure at $p$) by R.~Zhou \cite[Prop.~6.7]{Zhou:ModpPts}. 
Axiom B was proven by Lan and Stroh under the assumption that a ``good compactification'' for $\Sscr_m$ exist in \cite[Cor.~5.20]{LanStroh:NearbyCycles}. Following the reasoning of the proof of Mantovan's formula for PEL Shimura varieties in their paper, we see that it suffices to prove the isomorphy (respectively, the existence of good compatifications) for the more standard choice of the integral model given by the relative normalisation of the integral model of the Shimura variety with parahoric level structure (instead of the integral models constructed via Drinfeld level structure as in Mantovan's papers \cite{Mantovan:UnitaryShVar,Mantovan:Foliation}). This was done in the PEL case in \cite{Lan:CompactificationRamifiedPEL} (see also \cite[Prop.~2.2]{LanStroh:NearbyCycles}). %

To obtain Mantovan's formula in the tamely ramified Hodge-type case, we need to define Rapoport-Zink spaces and Igusa varieties in this generality, and show that the Newton strata, defined by Kisin and Pappas on $\FFbar_p$-points, are locally closed. For any of these constructions some sort of global ``crystalline'' tensors on $\Ascr_{\Gsf,\k}$ are needed which coincide with the tensors constructed by Kisin and Pappas when restricted to formal neighbourhoods of closed points. While we do not constructed tensors on the crystal of $\Ascr_\Gsf$, we prove the following result

 \begin{noprop}[{\emph{cf.}\ Prop.~\ref{prop global tensors}}]\label{prop global tensors intro}

Assume that $p>2$.
Let $\wh\Sscr$ denote the $p$-adic completion of $\Sscr\otimes_{O_E}O_{\breve{E}}$. Then  there exists a family of $\Psi$-invariant tensors $(t_{\alpha})$ in $P(\Ascr_{\Gsf}[p^\infty]_{\wh\Sscr})^\otimes$, where $P(\cdot)$ denotes the ``display''\footnote{``Displays'' over (not necessarily affine) $p$-adic formal schemes are described right above the statement of Prop.~\ref{prop global tensors}.} associated to a {\BT}, such that for every $z \in \Sscr(\k)$ the restriction of $t_{\alpha}$ to the formal neighbourhood of $z$ coincides with the tensors constructed by Kisin and Pappas.
 \end{noprop}

 Using this proposition, we deduce that for any $b \in B(G,\{\mu\})$ the Newton stratum $\Ssb^b \subset \Ssb$ is locally closed and thus can be considered as subvariety. Moreover, using Proposition~\ref{prop global tensors intro} and Axiom~\ref{axiom RZ intro}, we can construct (perfected) Igusa varieties and Rapoport-Zink spaces of Hodge type as closed subspaces of the corresponding spaces in the Siegel case. To explain, let $b \in B(G,\{\mu\})$ and $b' \coloneqq \rho(b) \in B(\GSp_n,\rho(\mu))$. Choosing a (nice enough) base point in $\Ssb^b$, we obtain a principally polarised {\BT} $\XX$ equipped with crystalline Tate-tensors. Then we have formal schemes $\Igfr^{b'}$ and $\Mfr^{b'}$ over the Siegel modular variety $\Sscr'$, where $\Igfr^{b'}$ parametrises polarised isomorphisms of $\Ascr_{\Gsf,\k}[p^\infty]$  with $\XX$, and $\Mfr^{b'}$ parametrises polarised quasi-isogenies of $\Ascr_{\Gsf,\k}[p^\infty]$  with $\XX$.
 
As key geometric input for our proof of Mantovan's formula, we prove the following version of the almost product structure of Newton strata $\Ssb^b\subset \Ssb$ of Hodge-type Shimura varieties.

 \begin{nothm}[{\emph{cf.} \S\S~\ref{ssect construction of RZ}, \S\S~\ref{ssect Igusa Var}, Prop.~\ref{thm product structure}}]
  We continue to assume that $p > 2$ and that Axiom~\ref{axiom RZ intro} holds. Then for any $b\in B(G,\{\mu\})$ and a suitable base point in $\Ssb^b$,  there exist closed formal subschemes $\Igfr^b\subset \Igfr^{b'}$ and $\Mfr^b\subset\Mfr^{b'}$, defined by some natural conditions involving the tensors obtained in Proposition~\ref{prop global tensors intro}. %
  
Furthermore, the almost product structure $\Mfr^{b'} \times \Igfr^{b'} \to \Sscr'$ lifts uniquely to a morphism $\Mfr^b \times \Igfr^b \to \Sscr$, which defines a pro\'etale $J_b(\QQ_p)$-torsor ``up to perfection'' over the Newton stratum $\Ssb^b$. Due to the uniqueness of the lifting, it is moreover invariant under the Hecke correspondences and Weil group action.
 \end{nothm}
Unsurprisingly, the construction of $\Igfr^b$ and $\Mfr^b$, as well as the proof of  almost product structure of Newton strata, is analogous to the unramified Hodge-type case using Axiom~\ref{axiom RZ intro} as an input (\emph{cf.} \cite{Hamacher:ShVarProdStr}, \cite{HowardPappas:GSpin}). Another input for almost product structure in the tamely ramified case is the deformation-theoretic analogue of almost product structure obtained in \cite{Kim:ProductStructure}.

 In order to formulate the main result, we need one further ingredient. Let $\Mcal^b \coloneqq (\Mfr^b)^{\rm ad}_{\breve{E}}$ the adic generic fibre. As in the case of Rapoport-Zink spaces with hyperspecial level strucure we get a Galois tower $(\Mcal^b_K)_{K \subset \Ksf_p}$ over $\Mcal^b$. For a certain final system $(\Mcal^b_m)$ we construct auxilliary integral models $\Mfr^b_m$ and define Mantovan's functor
\[
 \Mant_{b,\mu}(-) \coloneqq \sum_{i=0}^{2d}(-1)^i\varinjlim_m \Ext_{J_b(\QQ_p)}^{-2d+i}(\Rrm\Gamma_c(\Mcal_m^b,\QQbar_l),-)(-d),
\]
from the Grothendieck group of smooth $J_b(\QQ_p)$-representations over $\QQbar_l$ to the Grothendieck group of smooth representations of $\Gsf(\AA_f)\times W_E$ over $\QQbar_l$.

 \begin{nothm}[\emph{cf.}\ Cor.~\ref{cor:MantShVar}]\label{thm Mantovan}
  Assume there exists at least one Kisin-Pappas integral model for $(\Gsf,\Xsf)$ such that Axiom~\ref{axiom RZ} and Axiom~\ref{axiom:LanStroh intro} hold. Then for any $l$-adic automorphic sheaf $\Lscr_\xi$ we have the following equality of virtual smooth representations of $\Gsf(\AA_f)\times W_E$ over $\QQbar_l$:
 \begin{multline*}
 \sum_i (-1)^i \varinjlim_{\Ksf\subset\Gsf(\AA_f)}\coh{i}_c(\Sh_\Ksf(\Gsf,\Xsf)_{\overline{\breve{E}}},\Lscr_\xi) \\
 = \sum_{b\in B(G,\{\mu\})}\sum_{j}(-1)^j\Mant_{b,\mu}(\varinjlim_{\Ksf^p\subset\Gsf(\AA_f^p)}\coh j_c(\bar\Ig^b_{\Ksf^p},\Lscr_\xi)),
\end{multline*}
where $\bar\Ig^b_{\Ksf^p}$ denotes the underlying reduced subscheme of $\Igfr^b$.
 \end{nothm}

Let us outline the structure of the paper. In the first two chapters we give background on Mieda's duality theorem and the Kisin-Pappas integral model, respectively. Then we discuss the display tensors over the $p$-adic completion of $\Ascr_G$; as a result we obtain the Newton stratification of the special fibre in the third chapter. In the fourth and fifth chapter we lay the geometric foundation of Mantovan's formula. First, we construct Rapoport-Zink spaces of Hodge type with Bruhat-Tits level structure by generalising the construction of Howard and Pappas. Then we define Igusa varieties and prove the almost product structure of Newton strata in the special fibre of Shimura varieties. Finally, in the sixth section we prove Mantovan's formula in the Hodge-type case, which is our main result.

\subsection*{Acknowledgement}
 We thank Rong Zhou and Stephan Neupert for helpful conversations and in particular for sharing a preliminary version of their respective preprints with us. We are grateful to Kai-Wen Lan for his explanations about compactifications of Shimura varieties. %
We would also like to thank the anonymous referee, whose comments have been extremely helpful.
 The first named author was partially supported by the ERC starting grant 277889 ``Moduli of local $G$-shtukas''.
The second named author was supported by the EPSRC (Engineering and Physical Sciences Research Council) in the form of EP/L025302/1.

  \section{Review of compactly supported cohomology and Mieda's duality theorem} \label{sect mieda}
 In this section, we review some basic definitions and recall Mieda's duality theorem (Theorem~\ref{thm:Poincare}), following \cite{Mieda:Zelevinsky}. 
 
 Throughout this paper, we always work with torsion \'etale sheaves with torsion order invertible on the base scheme, and $\ZZ_l$- and $\QQ_l$- sheaves with $l$ invertible on the base scheme. For \'etale sheaves on adic spaces, we additionally assume that the torsion order and $l$ are not divisible by the residue characteristic.

\subsection{Review of the compactly supported cohomology for schemes}\label{subsec:CompSuppCohScheme}
The theory of compactly supported cohomology for partially proper \emph{analytic} adic spaces was developed in \cite[\S5.3]{Huber:EtCohBook}.
Mieda \cite[\S3.1]{Mieda:Zelevinsky} observed that the definition of partially proper rigid varieties (via Huber's theory of adic spaces; \emph{cf.} \cite[Definition~1.3.3]{Huber:EtCohBook}) can be adapted to schemes, and one can analogously develop the theory of compactly supported cohomology. We recall the definition and basic properties.

\begin{defnsub}[{\cite[Definition~3.1]{Mieda:Zelevinsky}}]
A morphism $f\colon X \to Y$ of schemes is called \emph{partially proper} if it is locally of finite type, separated, and universally specialising; or equivalently, $f$ is locally of finite type and satisfies the ``valuative criterion for properness'' (but not necessarily quasi-compact).
\end{defnsub}

In the case when $Y$ is the spectrum of a field $\kappa$, $X$ is partially proper over $\kappa$ if and only if $X$ is a locally finite union of its irreducible components, which are proper schemes over $\kappa$; \emph{cf.} \cite[Corollary~3.3]{Mieda:Zelevinsky}. A natural example of partially proper and non-proper $\kappa$-schemes is the underlying reduced scheme of a Rapoport-Zink space.

For a scheme $X$, we let $X_\et$ denote the \'etale site, and $\widetilde{X}_\et$ the \'etale topos.  For a torsion coefficient ring $\Oscr$, we let $\Oscr\hyph\widetilde{X}_\et$ denote the category of \'etale sheaves of $\Oscr$-modules on $X$. We will always choose $\Oscr$ to be a $\ZZ/N\ZZ$-algebra, where  $N$ is invertible in $\Ocal_X$.


\begin{defnsub}[{\emph{cf.} \cite[Definition~3.4]{Mieda:Zelevinsky}, \cite[\S~1]{Huber:ComparisonEllAdic}}]\label{def:CompSuppCoh}
Assume that $X$ is a partially proper scheme over a field $\kappa$. 
For $\Fscr\in\Oscr\hyph\widetilde{X}_\et$, let $\Gamma_c(X,\Fscr)$ denote the $\Oscr$-module of global sections with proper support.

For an \'etale $\ZZ_l$-sheaf $\Fscr$, we let $\Gamma_c(X,\Fscr)$ denote the $\ZZ_l$-module of global sections with proper support.\footnote{Assume that the \'etale $\ZZ_l$-sheaf $\Fscr$ is given by an inverse system $\{\Fscr_n\}$ for $\Fscr_n\in(\ZZ/l^n)\hyph\widetilde X_\et$.  Then a section in $\Gamma(X,\Fscr)$ has \emph{proper support} if its image in $\Gamma(X,\Fscr_n)$ is supported in a common quasi-compact closed subset independent of $n$. Therefore, $\Gamma_c(X,\Fscr)$ may \emph{not} coincide with the inverse limit of $\Gamma_c(X,\Fscr_n)$ in general, as it was pointed out in \cite[Introduction]{Huber:ComparisonEllAdic}.}	We similarly define $\Rrm\Gamma_c(X,\Fscr)$ when $\Fscr$ is an \'etale $\QQ_l$-sheaf or $\overline\QQ_l$-sheaf.
\end{defnsub}

For an $\Oscr$-module $\Fscr$, let $\Rrm\Gamma_c(X,\Fscr)\in D^+(\Mod(\Oscr))$ and $\coh i_c(X,\Fscr)$ denote the right derived functors.  If $\Fscr$ is an $l$-adic sheaf, we define $\Rrm\Gamma_c(X,\Fscr)\in D^+(\Mod(\ZZ_l))$ and $\coh i_c(X,\Fscr)$ analogous to \cite[\S~1, pp.~219--220]{Huber:ComparisonEllAdic}, where $\Rrm\Gamma_c(X,\Fscr)$ was defined when $X$ is a partially proper analytic adic space. We sketch the construction for the reader's convenience. First, we embed the category of $\ZZ_l$-sheaves, considered as projective systems $(\Fscr_n)_{n \in \NN}$ of $\ZZ/l^n$-modules, into the category of projective systems of sheaves on the \'etale site of $X$ with values in $\ZZ_l$-modules indexed by $\NN$. Since the latter category has enough injectives, we can define derived functors on this categories. We extend the $\Gamma_c$ to the above category of projective systems by defining $\Gamma_c((\Fscr_n)) \subseteq \varprojlim \Gamma(\Fscr_n)$ as the submodule of sections $(s_n)_{n\in \NN}$ whose collective support is contained in a quasi-compact closed subset of $X$ and define  $\Rrm\Gamma_c$ and $\coh i_c(X,-)$ to be its right derived functor.


\begin{rmksub}
\begin{subenv}
\item
In the case that $X$ is  proper over $\kappa$, then we have $\Rrm\Gamma_c(X,\Fscr) = \Rrm\Gamma(X,\Fscr)$. But if $X$ is not quasi-compact, then $\Rrm\Gamma_c(X,\Fscr)$ and $\Rrm\Gamma(X,\Fscr)$ could be different, and $\Rrm\Gamma_c(X,\Fscr)$ behaves more nicely.
\item
Note that the ``more standard'' definition of $\Rrm\Gamma_c(X,\Fscr)$ (via compactification) was made for a separated finite-type scheme $X$ over $\kappa$ (in particular, $X$ has to be quasi-compact), so it does not cover partially proper $\kappa$-schemes that are not quasi-compact. Indeed, the two definitions are compatible in the intersecting case (namely, when $X$ is proper over $\kappa$).
\end{subenv}
\end{rmksub}

\begin{propsub}\label{prop:CompSupp} 
Let $X$ be partially proper over a field $ \kappa$, and $\Fscr\in\Oscr\hyph\widetilde{X}_\et$.
\begin{enumerate}
	\item\label{prop:CompSupp:qcClosed}   We have $\coh i_c(X,\Fscr) \cong \varinjlim_Z\coh i_Z(X,\Fscr)$, where $Z$ runs through quasi-compact closed subsets of $X$.
	\item\label{prop:CompSupp:qcOpen}  
	We have 
	\[\coh i_c(X,\Fscr) \cong \varinjlim_U \coh i_c(U,\Fscr_U)\cong\varinjlim_U \coh i(\overline U,\Fscr_{\overline U}),\] 
	where $U$ runs through quasi-compact open subschemes of $X$ and $\overline U$ is the Zariski closure of $U$ in $X$.
	\item For any filtered direct system $\{\Fscr_\lambda\}$ with $\Fscr\coloneqq \varinjlim_\lambda\Fscr_\lambda$, we have $\coh i_c(X,\Fscr) \cong \varinjlim_\lambda\coh i_c(X,\Fscr_\lambda)$.
\end{enumerate}
The properties (\ref{prop:CompSupp:qcClosed}) and (\ref{prop:CompSupp:qcOpen})  also hold if $\Fscr$ is an \'etale $\ZZ_l$- sheaf (respectively, $\QQ_l$-sheaf; respectively, $\overline\QQ_l$-sheaf).
\end{propsub}
\begin{proof}
When $\Fscr$ is a torsion \'etale sheaf, the proposition was verified in \cite[Proposition~3.5, Corollary~3.6]{Mieda:Zelevinsky} except   $\varinjlim_U \coh i_c(U,\Fscr_U)\cong\varinjlim_U \coh i(\overline U,\Fscr_{\overline U})$. Note first that since $U$ is quasi-compact, its Zariski closure $\overline U$ is quasi-compact by  \cite[Proposition~3.2,~iii)]{Mieda:Zelevinsky}, so we can find another quasi-compact open $U'$ containing $\overline U$. Therefore, we have pushforward maps
\[
\coh i_c(U,\Fscr_U) \to \coh i(\overline U,\Fscr_{\overline U}) \to \coh i_{\overline U}(X,\Fscr)\to \coh i_c(U',\Fscr_{U'}),
\]
which proves the isomorphism $\varinjlim_U \coh i_c(U,\Fscr_U)\cong\varinjlim_U \coh i(\overline U,\Fscr_{\overline U})$.
The properties (\ref{prop:CompSupp:qcClosed}) and (\ref{prop:CompSupp:qcOpen})  for $l$-adic \'etale sheaves can be verified by repeating the proof of \cite[Proposition~3.5]{Mieda:Zelevinsky}, with a small technical caveat: Since Mieda's proof uses that the category $\Oscr$-$\Xtilde_\et$ has enough injectives, we have to work in the category of projective system of sheaves on the \'etale site of $X$ with values in $\ZZ_l$-modules.
\end{proof}



\subsection{Cohomology of partially proper adic spaces with support condition}
Recall that a morphism of adic spaces $f:\Xcal\ra\Ycal$  is \emph{partially proper} if it is locally of $^+$weakly finite type, separated, and universally specialising; \emph{cf.} \cite[Definition~1.3.3]{Huber:EtCohBook}. Furthermore, there is a valuative criterion for partial properness; \emph{cf.} \cite[Proposition~1.3.7]{Huber:EtCohBook}.

Now, let $\Xcal$ be a partially proper adic space over $(C,O_C)$, where $C$ is an algebraically closed non-archimedean field and $O_C$ its valuation ring. We define a \emph{$\ZZ_l$-sheaf} on $\Xcal$ to be a projective system of $\ZZ/l^n\ZZ$-sheaves as in \cite[\S1]{Huber:ComparisonEllAdic}, and define a $\QQ_l$-sheaf in the same was as the scheme case.

\begin{defnsub}[{\cite[Definition~3.13]{Mieda:Zelevinsky}}]
A \emph{support set} $\Ccal$ of $\Xcal$  is a set consisting of closed subsets with the following properties:
\begin{enumerate}
\item For $\Zcal,\Zcal'\in\Ccal$, we have $\Zcal\cup \Zcal'\in\Ccal$.
\item For $\Zcal\in\Ccal$, any closed subset $\Zcal'\subset \Zcal$ also belongs to $\Ccal$.
\end{enumerate}

Let $\Ccal$ be a support set of $\Xcal$, and let $\Fscr$ be a torsion or $l$-adic \'etale sheaf on $\Xcal$. Then we define $\Gamma_\Ccal(\Xcal,\Fscr)\coloneqq \varinjlim_{\Zcal\in\Xcal}\Gamma_\Zcal(\Xcal,\Fscr)$. We let $\Rrm\Gamma_\Ccal(\Xcal,\Fscr)$ denote the derived functor of $\Gamma_\Ccal(\Xcal,\Fscr)$ as an object in a suitable derived category, and write $\coh i_\Ccal(\Xcal,\Fscr)\coloneqq \Rrm^i\Gamma_\Ccal(\Xcal,\Fscr)$. Just as the case of partially proper schemes, we have
\[
\coh i_\Ccal(\Xcal,\Fscr) \cong \varinjlim_{\Zcal\in\Ccal}\coh i_\Zcal(\Xcal,\Fscr).
\]
\end{defnsub}

\begin{defnsub}[{\cite[Definition~3.14]{Mieda:Zelevinsky}}]
\begin{enumerate}
\item  For a set $\Ccal_0$ of closed subsets of $\Xcal$, the \emph{support set generated by $\Ccal_0$} is the support set $\Ccal$ consisting of closed subsets of finite unions of elements in $\Ccal_0$.
\item For a morphism $f:\Xcal'\ra\Xcal$ of  partially proper adic spaces over $(C,O_C)$ and for a support set $\Ccal$ of $\Xcal$, we define $f\iv(\Ccal)$ to be the support set for $\Xcal'$  generated by $\{f\iv \Zcal;\ \Zcal\in\Ccal\}$.
\item  Let $(\Xcal,\Ccal)$ and $(\Xcal',\Ccal')$ be partially proper adic spaces over $(C,O_C)$ equipped with support sets. Then a morphism of pairs $f:(\Xcal',\Ccal')\ra(\Xcal,\Ccal)$ is a morphism of adic spaces $f:\Xcal'\ra\Xcal$ such that $f\iv(\Ccal)\subseteq\Ccal'$. If  $\Fscr$ is torsion \'etale sheaf on $\Xcal$, we denote by $f^\ast$ the natural morphism
\[
f^\ast :\Rrm\Gamma_{\Ccal}(\Xcal,\Fscr) \ra \Rrm\Gamma_{\Ccal'} (\Xcal',f^\ast \Fscr).
\]
\end{enumerate}
\end{defnsub}

\begin{exasub}
Let $c$ denote the set of all quasi-compact closed subsets of $\Xcal$ (which is a support set). Then we define $\Gamma_c$, $\Rrm\Gamma_c$ and $\coh i_c$ using the support set $c$ in place of  $\Ccal$. As $\Xcal$ is assumed to be partially proper over $(C,O_C)$, $\Rrm\Gamma_c$  defines the correct notion of compactly supported cohomology; \emph{cf.} \cite[\S5.3]{Huber:EtCohBook}. Furthermore,  any proper morphism of partially proper adic spaces $f:\Xcal'\ra\Xcal$ induces a morphism of pairs $f:(\Xcal',c)\ra(\Xcal,c)$, so one can define $f^\ast :\Rrm\Gamma_c(\Xcal,\Fscr) \ra \Rrm\Gamma_c(\Xcal',f^\ast \Fscr)$ for any $\Fscr$.
\end{exasub}

Let $ O $ be a complete discrete valuation ring with separably closed residue field $\kappa$. Let $E$ be the fraction field of $O$.
For any formal scheme $\Xfr$ locally formally of finite type over $\Spf O$, we let $\Xfr^\ad$ denote the associated adic space over $\Spa(O,O)$.\footnote{In \cite{Huber:EtCohBook}, $\Xfr^\ad$ is denoted by $t(\Xfr)$.}  For any non-archimedean extension $K$ of $E$, we write 
\begin{equation}\label{eq rig gen fib}
\Xfr^\ad_K\coloneqq \Xfr^\ad\times_{\Spa(O,O)}\Spa(K,O_K).
\end{equation}
Note that $\Xfr^\ad_K$ coincides with the adic space associated to the rigid generic fibre of $\Xfr_{ O _K}\coloneqq \Xfr\times_{\Spf O }\Spf O _K$ (constructed by Berthelot).

The following proposition shows the relation between partially proper $\kappa$-varieties and partially proper adic spaces over $E$ via formal models.

\begin{propsub}[{\cite[Proposition~3.9]{Mieda:Zelevinsky}, \cite[Proposition~4.23]{Mieda:FormalNearbyCycles}}]
Let $\Xfr$ be a formal scheme which is locally formally of finite type over $\Spf O$. If $\Xfr^\red$ is partially proper over $\kappa$, then the rigid generic fibre  $\Xfr^\ad_{E}$ of $\Xfr$ is partially proper over $\Spa(E, O)$
\end{propsub}
In particular, the rigid generic fibre of a Rapoport-Zink formal moduli scheme is partially proper, as well as its natural \'etale covering defined via level structures.

Let $C$ denote an algebraically closed complete extension of $E=\Frac( O )$, with ring of integers $ O _C$;  for example, we may take $C = \widehat{\overline E}$. 
Recall that we have a natural morphism of locally topologically ringed spaces
\begin{equation}\label{eqn:specialisation}
	\spbf(=\spbf_\Xfr): (\Xcal,\Ocal_{\Xcal}^+) \ra (\Xfr^\ad,\Ocal_{\Xfr^\ad}^+) \ra (\Xfr,\Ocal_\Xfr),
\end{equation}
where the last arrow is the morphism final among morphisms from adic spaces over $\Spa(O ,O)$ to $\Xfr$; \cite[Proposition~4.1]{Huber:GenFormalRigid}. Therefore, we obtain the following morphism of \'etale sites
\[
\spbf (=\spbf_\Xfr):\Xcal_\et \ra \Xfr_\et \cong \Xfr^\red_\et.
\]
More explicitly, $\spbf^\ast $ sends $Y\in\Xfr^\red_\et$ to $\Yfr^\ad_C\in\Xcal_\et$, where $\Yfr\in\Xfr_\et$ is the unique lift of $Y$.

\begin{defnsub}\label{def:FormalModelSupportSet}
In the above setting, we further assume that $\Xfr^\red$ is partially proper over $\kappa$.
Following \cite[Definition~3.15]{Mieda:Zelevinsky} we let $\Ccal_\Xfr$ denote the support set of $\Xcal$ generated by $\{\spbf\iv(Z)\}$ where $Z\subset\Xfr^\red$ are quasi-compact closed subsets. 
\end{defnsub}

\begin{rmksub}
For a closed subset $Z\subset \Xfr^\red$, $\spbf^{-1}(Z)$ is in general larger than its interior $\spbf\iv(Z)^\circ$. For example, if $\Xfr\coloneqq \PP^1_O$ and $Z$ is the closed point $0$ in the special fibre, then $\spbf^{-1}(0) \setminus \spbf\iv(0)^\circ$ consists of a single rank-$2$ specialisation of the Gauss point of the closed unit disk around $0$.
\end{rmksub}

\begin{exasub}\label{exa:SuppSet}
If $\Xfr$ is locally of finite type over $O$ (i.e., a uniformiser of $O$ generates an ideal of definition of $\Xfr$), then we  have $c=\Ccal_\Xfr$. 

In general, we only have $c\subset \Ccal_\Xfr$, as the following example shows. Let $\Xfr\coloneqq \Spf O[[t]]$. Then $\Xcal\coloneqq \Xfr^\ad_C$ is an open unit disk and $\Xfr^\red = \Spec \kappa$  is proper over $\kappa$. Clearly, the support set $\Ccal_\Xfr$ contains $\Xcal = \spbf\iv(\Xfr^\red)$,  so $\Ccal_\Xfr$ consists of all closed subsets of $\Xcal$ and we have $\Rrm\Gamma_{\Ccal_\Xfr}(\Xcal,\bullet) = \Rrm\Gamma(\Xcal,\bullet)$. Note that $\Xcal$ is not quasi-compact so $c\ne\Ccal_\Xfr$. The same applies to $\Xfr\coloneqq \Spf R$ where $R$ is a complete local noetherian flat $O$-algebra.
\end{exasub}

Unless $c=\Ccal_\Xfr$, the support set $\Ccal_\Xfr$ depends on the choice of the formal model $\Xfr$. In particular, $\Rrm\Gamma_{\Ccal_\Xfr}(\Xcal,-)$ may not necessarily have pull back functoriality for any morphism $f:\Xcal'\ra \Xcal$ unless $f$ satisfies some good property with respect to the formal model $\Xfr$. Indeed, we have the following result:

\begin{lemsub}[{\cite[Lemma~3.17]{Mieda:Zelevinsky}}]\label{lem:Funct}
In the setting of Definition~\ref{def:FormalModelSupportSet},
let $\Xfr$ and $\Xfr'$ are formal schemes locally formally of finite type over $O $ such that $\Xfr^\red$ and $\Xfr^{\prime\red}$ are partially proper over $\kappa$. We write $\Xcal\coloneqq \Xfr^\ad_C$ and $\Xcal'\coloneqq \Xfr^{\prime\ad}_C$.
Let $f:\Xfr'\ra\Xfr$ be a morphism of finite type over $O$, and we let $f_C:\Xcal'\ra\Xcal$  denote the induced map on the generic fibre. 

Then we have $f_C\iv(\Ccal_{\Xfr}) = \Ccal_{\Xfr'}$. In particular, if $f_C$ is an isomorphism (e.g., if $f$ is a finite iteration of admissible blow-ups), then we have an isomorphism of pairs $f_C:(\Xcal',\Ccal_{\Xfr'})\riso(\Xcal,\Ccal_{\Xfr})$.
\end{lemsub}

It was observed by Mieda that the cohomology $\Rrm\Gamma_{\Ccal_\Xfr}(\Xcal,\Fscr)$ is a suitable ``nearby cycle cohomology'' (\emph{cf.} \cite[Prop.~3.16]{Mieda:Zelevinsky}), which we review now.

\begin{defnsub}\label{def:FormalNearbyCycles}
Let $\Xcal=\Xfr^\ad_C$, and let $\Oscr$ be a torsion ring.
For  $\Fscr\in\Oscr\hyph\widetilde\Xcal_\et$, we define 
\[
\Rrm\Psi_\Xfr\Fscr\coloneqq  \Rrm\spbf_*\Fscr \in D^+(\Xfr^\red,\Oscr),
\]
and call it a \emph{formal nearby cycle sheaf}. We naturally extend the definition of $\Rrm\Psi_\Xfr\Fscr$ when $\Fscr$ is an $l$-adic \'etale sheaf (i.e., $\ZZ_l$-, $\QQ_l$-, or $\overline\QQ_l$- sheaf).

As observed in \cite[\S5.4.2]{Fargues:AsterisqueLLC}, $\Rrm\Psi_\Xfr\Fscr$ is compatible with the formal nearby cycle sheaves defined via Berkovich spaces \cite{Berkovich:VanishingCyclesFormal2}.
\end{defnsub}

To interpret the compactly supported cohomology of formal nearby cycle sheaves, we introduce  the  notion of \emph{locally algebraisable formal schemes}.
\begin{defnsub}[{\cite[Definition~3.19]{Mieda:FormalNearbyCycles}}]
Let $\Xfr$ be a formal scheme locally formally of finite type over $\Spf O$. 
We say that  $\Xfr$ is \emph{locally algebraisable} if there exists a Zariski open covering $\{\Ufr\}$ of $\Xfr$ such that each $\Ufr$ can be obtained as a completion of some  finite-type $O$-scheme.
\end{defnsub}

\begin{propsub}\label{prop:VanCycSpecSeq}
Let $\Xfr$ be a locally algebraisable formal scheme over $O$ such that $\Xfr^\red$ is partially proper over $\kappa$, and let $\Xcal$ be the generic fibre of $\Xfr$.
Let $\Fscr$ be either a \emph{constructable} torsion sheaf with torsion order invertible in $O$, or a \emph{constructable} $l$-adic sheaf on $\Xcal$ with $l$ invertible in $O$. 

Then we have a natural isomorphism 
\[\Rrm\Gamma_c(\Xfr^\red,\Rrm\Psi_\Xfr\Fscr)\cong \Rrm\Gamma_{\Ccal_\Xfr}(\Xcal,\Fscr).\]
\end{propsub}
\begin{proof}
It suffices to handle the case when $\Fscr\coloneqq \{\Fscr_n\}$ is a constructible \'etale $\ZZ_l$-sheaf with $l$ invertible in $O$ (so $\Fscr_n$ is a constructible \'etale $\ZZ/l^n$-sheaf).
From \cite[Theorem~8.3.5]{Huber:EtCohBook} and \cite[Corollary~3.6]{Berkovich:VanishingCyclesFormal2}, it follows that $R\Psi_\Xfr \Fscr_n$ is constructible and we have the following isomorphism for any finite union $Z$ of irreducible  components of $\Xfr^{\red}$:
\[
\Rrm\Gamma(Z, (R\Psi_\Xfr \Fscr_n)|_{Z}) \cong \Rrm\Gamma_{\spbf^{-1}(Z)}(\Xcal, \Fscr_n)\quad\forall n\geqslant 1.
\]
By Grothendieck's spectral sequence, we have
\begin{align*}
 \Rrm\Gamma(Z, (R\Psi_\Xfr \Fscr)|_{Z}) &\cong \Rrm\hspace{-.2em}\varprojlim \Rrm\Gamma(Z, (R\Psi_\Xfr \Fscr_n)|_{Z}) \\
 \Rrm\Gamma_{\spbf^{-1}(Z)}(\Xcal, \Fscr) &\cong \Rrm\hspace{-.2em}\varprojlim\Rrm\Gamma_{\spbf^{-1}(Z)}(\Xcal, \Fscr_n)
\end{align*}
and thus taking the derived limit of both sides of the above isomorphism yields
\[
\Rrm\Gamma(Z, (R\Psi_\Xfr \Fscr)|_{Z}) \cong \Rrm\Gamma_{\spbf^{-1}(Z)}(\Xcal, \Fscr).
\]
Now, by taking $\varinjlim_Z$ on both sides we obtain the following
\begin{align*}
 \Rrm\Gamma_c(\Xfr^\red,\Rrm\Psi_\Xfr\Fscr) &\cong 
\varinjlim_Z\Rrm\Gamma(Z, (R\Psi_\Xfr \Fscr)|_{Z})\\
& \cong 
\varinjlim_Z\Rrm\Gamma_{\spbf^{-1}(Z)}(\Xcal, \Fscr)\cong
\Rrm\Gamma_{\Ccal_\Xfr}(\Xcal,\Fscr),
\end{align*}
which concludes the proof.
\end{proof}
\subsection{Smooth equivariant sheaves on analytic \'etale sites}

Let $J$ be a locally pro-$p$ topological group; in other words, a locally profinite topological group whose topology is generated by pro-$p$ open subgroups. In practice, $J$ will be the group of $\Qp$-points of some reductive group over $\Qp$, or quotients thereof. In the intended application where $\Xfr\coloneqq \Mfr^b$ is a Rapoport-Zink formal scheme associated to $(G,b)$, we will work with the natural action of $J\coloneqq J_b(\Qp)$ on $\Xfr$.

Let us recall basic definitions and key properties of \emph{smooth} $J$-equivariant torsion \'etale sheaves, following \cite[\S{IV}.8\emph{ff.}]{Fargues:FaltingsIsom}. Here, "smoothness" is about the $J$-action on an \'etale sheaf, and does \emph{not} refer to the `` lissit\'e'' of  the underlying \'etale sheaf.

Let $\Oscr$ be an artinian quotient of some finite extension of $\ZZ_l$ with $l\ne p$. %
Note that such $\Oscr$ is Gorenstein, which will be relevant in Theorem~\ref{thm:Poincare}.

Let $\Xcal$ denote an adic space locally of finite type over $(C, O _C)$ with a continuous $J$-action on it; that is for any affinoid open $\Ucal\coloneqq \Spa(A,A^+)$ of $\Xcal$, the stabiliser $J_\Ucal$ of $\Ucal$ is an open subgroup of $J$ and the $J_{\Ucal}$-action on $A$ is continuous with respect to the natural topology on $A$.

In the intended application, $\Xcal$ admits a formal model $\Xfr$ over $\Spf O$ with $\Xfr^\red$ partially proper (as in the setting of Definition~\ref{def:FormalModelSupportSet}), and the $J$-action on $\Xcal$ extends to a continuous $J$-action on $\Xfr$.  (Here, continuity of the $J$-action on a formal scheme is defined similarly; \emph{cf.} \cite[D\'efinition~2.3.10]{Fargues:AsterisqueLLC}.)

For an adic space $\Xcal$ locally of finite type over $(C, O _C)$, we let $\Xcal_\et$ denote the \'etale site and $\Xcal_{\qcet}$ the site where the underlying category is the category of \emph{quasi-compact} adic spaces \'etale over $\Xcal$. 
Since any $\Ycal\in \Xcal_{\et}$ admits an \'etale cover $\{\Ucal_\alpha\}$ where  each $\Ucal_\alpha$ is quasi-compact, it follows that these the natural map defines an isomorphism of topoi $\tilde\Xcal_{\qcet}\riso\tilde\Xcal_{\et}$.
Let us recall the following  result that is fundamental in working with $J$-equivariant \'etale sites and $J$-equivariant \'etale sheaves:
\begin{lemsub}[{\cite[Key~Lemma~7.2]{Berkovich:VanishingCyclesFormal}}]\label{lem:EquivKeyLem}
For any $\Ucal\in\Xcal_\qcet$, there exists a compact open subgroup $J_\Ucal\subset J$ such that the $J_\Ucal$-action on $\Xcal$ lifts canonically and continuously to $\Ucal$.
\end{lemsub}

\begin{defnsub}[{\cite[D\'efinition~IV.8.1]{Fargues:FaltingsIsom}}]
A $J$-equivariant \'etale sheaf $\Fscr$ on $\Xcal$ is defined to be \emph{smooth} if for any $\Ucal\in\Xcal_\qcet$, the $J_\Ucal$-action on $\Gamma(\Ucal,\Fscr)$ is smooth, where $J_\Ucal$ is an open compact subgroup of $J$ stabilising $\Ucal$. (Note that such $J_\Ucal$ exists by Lemma~\ref{lem:EquivKeyLem}, and the notion of smoothness is independent of the choice of $J_\Ucal$.) We write  $\Oscr\hyph\widetilde\Xcal_\et/J$ for the category of smooth $J$-equivariant \'etale sheaves of $\Oscr$-modules on $\Xcal$.
\end{defnsub}
For any $J$-equivariant \'etale sheaf $\Fscr$ on $\Xcal$, one can naturally associate a smooth $J$-equivariant sheaf $\Fscr^{(\infty)}$, called the \emph{smoothening} of $\Fscr$ (in French, \emph{lissification}), defined as follows. For any $\Ucal\in\Xcal_\qcet$,
we define $\Gamma(\Ucal,\Fscr^{(\infty)})$ to be the set of smooth sections in $\Gamma(\Ucal,\Fscr)$ with respect to the $J_\Ucal$-action (where $J_{\Ucal}$ is as defined in Lemma~\ref{lem:EquivKeyLem}). This presheaf turns out to be a sheaf $\Fscr^{(\infty)}$ on $\Xcal_\qcet$, and we view it as a $J$-equivariant sheaf on $\Xcal_\et$ using the equivalence of topoi $\widetilde\Xcal_\qcet\cong\widetilde\Xcal_\et$. Note that the functor $\Fscr\rightsquigarrow\Fscr^{(\infty)}$ is the right adjoint of the natural inclusion of the category of smooth $J$-equivariant sheaves into the category of $J$-equivariant sheaves; \emph{cf.} \cite[\S{IV}.8.3.3]{Fargues:FaltingsIsom}.

Let $f:\Ucal\ra\Xcal$ be a $J$-equivariant \'etale morphism. Then we have the following operations (\emph{cf.}  \cite[Definition~3.25]{Mieda:Zelevinsky}):
\begin{enumerate}
\item pull back $\Oscr\hyph\widetilde\Xcal_\et/J\ra\Oscr\hyph\widetilde\Ucal_\et/J$ defined by $\Fscr\rightsquigarrow f^\ast \Fscr$.
\item extension by zero $\Oscr\hyph\widetilde\Ucal_\et/J\ra\Oscr\hyph\widetilde\Xcal_\et/J$ defined by $\Gcal\rightsquigarrow f_!\Gcal$; \emph{cf.} \cite[Lemma~3.24]{Mieda:Zelevinsky}. 
\item push forward $\Oscr\hyph\widetilde\Ucal_\et/J\ra\Oscr\hyph\widetilde\Xcal_\et/J$ defined by $\Gcal\rightsquigarrow (f_*\Gcal)^{(\infty)}$.
\end{enumerate}

For an open subgroup $K\subset J$, it is possible to ``sheafify'' the restriction, the compact induction, and the induction to obtain the following functors (\emph{cf.}  \cite[Definition~3.22]{Mieda:Zelevinsky}):
\begin{align*}
\Res^J_K&:\Oscr\hyph\widetilde\Xcal_\et/J\ra\Oscr\hyph\widetilde\Xcal_\et/K\\
\cInd^J_K,\ \Ind^J_K&:\Oscr\hyph\widetilde\Xcal_\et/K\ra\Oscr\hyph\widetilde\Xcal_\et/J.
\end{align*}

Finally, if $f:\Ucal\ra\Xcal$ is a $K$-equivariant \'etale morphism for an open subgroup $K\subset J$, then we have the functors 
\begin{align*}
f^\ast \circ\Res^J_K&:\Oscr\hyph\widetilde\Xcal_\et/J\ra\Oscr\hyph\widetilde\Ucal_\et/K\\
\cInd^J_K\circ f_!,\ \Ind^J_K\circ(f_*)^{(\infty)}&:\Oscr\hyph\widetilde\Ucal_\et/K\ra\Oscr\hyph\widetilde\Xcal_\et/J,
\end{align*}
 satisfying the adjunction relations generalising the Frobenius reciprocity; \emph{cf.}  \cite[Proposition~3.26]{Mieda:Zelevinsky}. 

Note that smoothness of $J$-action on $\Fscr$ does not necessarily imply that the natural $J$-action on $\Gamma(\Ucal,\Fscr)$ is smooth, unless $\Xcal$ is already quasi-compact. (If we assume that $J$ acts transitively on the set of connected components of $\Xcal$, then stabiliser of any locally constant function taking different values at each connected component is the intersection of the stabilisers of all connected components, which may not be open if there are infinitely many connected components.
Recall that by Berkovich's ``Key Lemma'' (Lemma~\ref{lem:EquivKeyLem}) the stabiliser of a single connected component is open.) 
Instead, $\Gamma(\Xcal,\Fscr)$ admits a natural action of the convolution algebra of compactly supported $\Oscr$-valued distributions\footnote{We will recall the definition of compactly supported distributions in \S\ref{subsec:Zelevinsky}.}  on $J$; \emph{cf.} \cite[Lemme~2.3]{Fargues:Zelevinsky}, \cite[Lemma~3.34]{Mieda:Zelevinsky}.%

On the other hand, compactly supported sections are always smooth in the following sense:
\begin{lemsub}[{\cite[Lemma~3.40]{Mieda:Zelevinsky}}]
Assume that $\Xcal$ is partially proper over $(C, O _C)$, equipped with a continuous $J$-action. Then  the $J$-action on $\Gamma_c(\Xcal,\Fscr)$ is smooth for   any $\Fscr\in\Oscr\hyph\widetilde\Xcal_\et/J$. 
\end{lemsub}

\begin{defnsub}\label{def:EquivCoh}
Let $\Rep_\Oscr(J)$ be the category of $\Oscr$-modules with smooth $J$-action. Then we define
\[
\Rrm\Gamma_c(\Xcal,\bullet):\Oscr\hyph\widetilde\Xcal_\et/J \ra D^+(\Rep_\Oscr(J))
\]
to be the derived functor of $\Gamma_c(\Xcal,\bullet):\Oscr\hyph\widetilde\Xcal_\et/J \ra \Rep_\Oscr(J)$. We write $\coh i_c(\Xcal,\Fscr)\coloneqq \Rrm^i\Gamma_c(\Xcal,\Fscr)$ for any $\Fscr\in\Oscr\hyph\widetilde\Xcal_\et/J$, and call it the $i$th \emph{compactly supported $J$-equivariant cohomology}.
\end{defnsub}

Note that a smooth $J$-equivariant sheaf $\Fscr$ can also be viewed as just an \'etale sheaf by forgetting the $J$-action (i.e., by applying the forgetful functor $\Oscr\hyph\widetilde\Xcal_\et/J\ra\Oscr\hyph\widetilde\Xcal_\et$). We intentionally use the same notation $\Rrm\Gamma_c(\Xcal,\Fscr)$ without specifying whether we view $\Fscr$ as a smooth $J$-equivariant sheaf or just as an \'etale sheaf\footnote{In \cite[\S3.3]{Mieda:Zelevinsky}, the compactly supported $J$-equivariant cohomology is denoted by $\Rrm\Gamma_c(\Xcal/J,\bullet)$, but we find this notation could lead to confusion with the compactly supported cohomology of the ``quotient'' $\Xcal/J$ in some suitable sense.}, thanks to the following proposition.

\begin{propsub}[{\cite[Corollary~3.42]{Mieda:Zelevinsky}}]\label{prop:EquivCoh}
Assume that $\Xcal$ is partially proper over $(C, O _C)$. Then the following diagram is 2-commutative
\[\xymatrix{
D^+(\Oscr\hyph\widetilde\Xcal_\et/J) \ar[d] \ar[rr]^-{\Rrm\Gamma_c(\Xcal,\bullet)} & &
D^+(\Rep_\Oscr(J)) \ar[d] \\
D^+(\Oscr\hyph\widetilde\Xcal_\et) \ar[rr]^-{\Rrm\Gamma_c(\Xcal,\bullet)} & &
D^+(\Mod(\Oscr)) 
},\]
where the vertical arrows are forgetful functors. 
\end{propsub}
The main content of this proposition is that an injective sheaf in $\Oscr\hyph\widetilde\Xcal_\et/J$ is acyclic for $\Gamma_c(\Xcal,\bullet)$ on $\Oscr\hyph\widetilde\Xcal_\et$; \emph{cf.} \cite[Proposition~3.39]{Mieda:Zelevinsky}.

This proposition implies that for $\Fscr\in \Oscr\hyph\widetilde\Xcal_\et/J$, the $J$-equivariant compactly supported cohomology  $\coh i_c(\Xcal,\Fscr)\in\Rep_\Oscr(J)$ coincides with the ``usual non-equivariant'' compactly supported cohomology $\coh i_c(\Xcal,\Fscr)\in\Mod(\Oscr)$ equipped with the $J$-action induced by functoriality. (In particular, we obtain smoothness of this $J$-action on the non-equivariant compactly supported cohomology.)

\subsection{Poincar\'e duality and derived duality \emph{a la} Zelevinsky, Dat, \dots}\label{subsec:Zelevinsky}
Let $J$ be a locally pro-$p$ topological group.
Let $\Oscr$ be a $0$-dimensional Gorenstein ring where $p$ is invertible (but not necessarily of characteristic~$0$); in practice, $\Oscr$ is either a field or an artinian quotient of a discrete valuation ring where $p$ is invertible. Note that flatness, projectivity and injectivity are equivalent for modules over Gorenstein artinian ring, which makes the functor $\Hom_\Oscr(\bullet,\Oscr)$ \emph{exact}.

Before introducing the convolution algebra of compactly supported distributions on $J$, let us start with some motivation. Let $\Rep_\Oscr(J)$ be the category of $\Oscr$-modules with \emph{smooth} $J$-action, and $\Rep_\Oscr(J^\disc)$ denote the category of $\Oscr[J^\disc]$-module. (We let $J^\disc$ denote the underlying group of $J$ with the discrete topology.) Then the natural functor $i_J:\Rep_\Oscr(J) \ra \Rep_\Oscr(J^\disc)$ admits a right adjoint
\[
\infty_J:\Rep_\Oscr(J^\disc)\ra \Rep_\Oscr(J); \quad V\rightsquigarrow \varinjlim_K V^K;
\] 
where $K\subset J$ runs over compact open subgroups. For any ``operation'' on $\Rep_\Oscr(J^\disc)$ that may not necessarily preserve smoothness, one may consider ``smoothening'' it by applying $\infty_J$. However, $\infty_J$ is \emph{not} right exact in general (as the $K$-invariance functor may not be right exact), so it is not  straightforward to apply $\infty_J$ to constructions involving homological algebra. Such homological algebra issues can be circumvented if we can work with modules over ``compactly supported distributions'' instead of $\Rep_\Oscr(J^\disc)$. Let us recall the basic definitions.

Let $\Hcal_\Oscr(J)\coloneqq \Ccal^\infty_c(J;\Oscr)$ be the $\Oscr$-module  of compactly supported locally constant functions $J\ra\Oscr$. We recall the following definition from \cite[Appendix~B.2]{Dat:LanglandsDrinfeld}, which also works when $\Oscr$ is a torsion ring as far as $p$ is invertible in $\Oscr$.
\begin{defnsub}[\emph{cf.} {\cite[\S1.8]{Dat:Finitude}}]
An \emph{$\Oscr$-valued distribution} (or a \emph{distribution}) on $J$ means a linear functional $\mu\in \Hom_\Oscr(\Hcal_\Oscr(J),\Oscr)$. Traditionally, we write $\int_J f d\mu\coloneqq  \mu(f)$ for any $f\in\Hcal_\Oscr(J)$. 

A distribution $\mu$ is \emph{compactly supported} if there exists a compact subset $K\subset J$ such that $\int_J f d\mu=0$ for any locally constant function $f$ supported away from $K$. Let $\Dcal_\Oscr(J)$ denote the $\Oscr$-module of compactly supported distributions.

We define the convolution action
\[
\Dcal_\Oscr(J)\times\Hcal_\Oscr(J)\ra\Hcal_\Oscr(J);\ (\mu, f)\mapsto \mu*f,
\]
where $(\mu*f)(\gamma)\coloneqq  \int_J (\gamma\cdot f)d\mu$ for any $\gamma\in J$. Here, $\gamma\cdot f$ is the right regular action of $J$ on $\Hcal_\Oscr(J)$.

Note that the map $\Dcal_\Oscr(J)\ra\End_\Oscr(\Hcal_\Oscr(J))$, sending $\mu$ to $(f\mapsto\mu*f)$, is injective. For $\mu,\mu'\in\Dcal_\Oscr(J)$, we define the convolution product $\mu*\mu'\in\Dcal_\Oscr(J)$ to be the distribution corresponding to the endomorphism $f\mapsto \mu*(\mu'*f)$ of $\Hcal_\Oscr(J)$, which uniquely exists. Note that the convolution product makes $\Dcal_\Oscr(J)$ a $\Oscr$-algebra with unity, where the multiplicative unity is the ``delta function at identity'' $\delta_e:f\mapsto f(e)$.
\end{defnsub}

Note that $\Dcal_\Oscr(J)$ contains the following $\Oscr$-subalgebras:
\begin{enumerate}
\item We have $\Oscr[J^\disc] \hra \Dcal_\Oscr(J)$ by sending $\gamma\in J$ to the ``delta function'' $\delta_\gamma:f\mapsto f(\gamma)$.
\item If there is a $\Oscr$-valued Haar measure $dh$ on $J$ in the sense of Vign\'eras \cite[\S{I}.3]{Vigneras:ModellRep}, then we can embed $\Hcal_\Oscr(J)\hra\Dcal_\Oscr(J)$ by $f\mapsto f dh$. Although the embedding depends on the choice of Haar measure, the image  precisely consists of compactly supported distributions that are invariant under some compact open subgroup of $J$. 
\end{enumerate}

\begin{defnsub}[\emph{cf.} {\cite[\S1.8]{Dat:Finitude}}]
For an $\Oscr$-module $V$, let $\Ccal^\infty(J;V)$ denote the $\Oscr$-module of locally constant morphisms from $J$ to $V$.  We define a natural $\Oscr$-linear morphism
\begin{align*}
\Dcal_\Oscr(J) &\to  \Hom_\Oscr(\Ccal^\infty(J;V),V);\\
\mu &\rightsquigarrow  (\rho\mapsto \int_J \rho d\mu),
\end{align*}
where for any given $\rho \in \Ccal^\infty(J;V)$, we define $\int_J \rho d\mu\in V$ as follows. 
We choose an open compact subset $K\subset J$ where  $\mu\in\Dcal_\Oscr(J)$ is supported; that is, $\int_J f d\mu = 0$ if $f$ is supported in $J\setminus K$. Then $\int_J \rho d\mu\in V$ is the image of $\rho$ under the following map
\[
\xymatrix@1{
\Ccal^\infty(J;V)\ar[r]^-{(\bullet)|_K}&
\Ccal^\infty(K;V) \cong \Ccal^\infty (K;\Oscr)\otimes_\Oscr V \ar[r]^-{\mu\otimes V}&
\Oscr\otimes_\Oscr V = V
}.
\]
 (The isomorphism in the middle can be deduced from the case when $K$ is a finite set, since  $\Ccal^\infty(K;V) $ is the directed union of maps from some finite discrete quotients of $K$.)
Note that  $\int_J \rho d\mu$ does not depend on the choice of support $K$.
\end{defnsub}

Let $\Mod(\Dcal_\Oscr(J))$ denote the category of $\Dcal_\Oscr(J)$-modules.
Then we have a natural functor
\begin{equation}
i_{\Dcal_\Oscr(J)}:\Rep_\Oscr(J)\ra\Mod(\Dcal_\Oscr(J)),
\end{equation}
defined as follows: the underlying $\Oscr$-module of $i_{\Dcal_\Oscr(J)}(V)$ is $V$, and $\mu\in\Dcal_\Oscr(J)$ acts on $v\in V$ as $\mu*v\coloneqq \int_J\rho_v d\mu$ where $\rho_v\in\Ccal^\infty(J;V)$  sends $\gamma\in J$ to $ \gamma\cdot v$. Since $\Dcal_\Oscr(J)$ contains $\Oscr[J^\disc]$,  it follows that $i_{\Dcal_\Oscr(J)}$ is fully faithful. We will often regard $\Rep_\Oscr(J)$ as the full subcategory of $\Mod(\Dcal_\Oscr(J))$ via $i_{\Dcal_\Oscr(J)}$, unless there is risk of confusion.

The functor $i_{\Dcal_\Oscr(J)}$ admits a right adjoint
\[
\infty_{\Dcal_\Oscr(J)}:\Mod(\Dcal_\Oscr(J))\ra \Rep_\Oscr(J),
\]
defined as follows. For any pro-$p$ open subgroup $K\subset J$, we have a natural idempotent $e_K\in\Dcal_\Oscr(J)$ supported in $K$, such that $e_K|_K$ is the Haar measure on $K$ with total volume~$1$. Then we define
 \[
\infty_{\Dcal_\Oscr(J)}(M)\coloneqq \varinjlim_K e_K*M.
\] 
Note that $\infty_{\Dcal_\Oscr(J)}$ is \emph{exact} (contrary to $\infty_J$), which is the motivation for working with modules over $\Dcal_\Oscr(J)$.

Since $\Oscr[J^\disc]\hra\Dcal_\Oscr(J)$, we have a functor $\kappa_J:\Mod(\Dcal_\Oscr(J))\ra \Rep_\Oscr(J^\disc)$. But note that $\infty_{\Dcal_\Oscr(J)}$ may not be equal to $\infty_J\circ\kappa_J$.

\begin{defnsub}\label{def:DerivedZelevinsky}
Note that $\Hcal_\Oscr(J)$ is a bi-$\Dcal_\Oscr(J)$-module, so we can define a contravariant functor $\Dbf^m:\Rep_\Oscr(J)\ra\Mod(\Dcal_\Oscr(J))$ by sending  $V$ to $\Hom_{J}(V,\Hcal_\Oscr(J))$, where the $J$-equivariance is with respect to the left regular $J$-action on $\Hcal_\Oscr(J)$, and the right regular $J$-action gives rise to the $\Dcal_\Oscr(J)$-action on $\Hom_J(V,\Hcal_\Oscr(J))$. More explicitly, for $\mu\in\Dcal_\Oscr(J)$ and $\rho:V\ra \Hcal_\Oscr(J)$, we define the left regular action to be $(\mu*\rho)(v)\coloneqq \rho(v)*\check\mu$, where $\mu\mapsto\check\mu$ is the involution on $\Dcal_\Oscr(J)$ induced by $g\mapsto g\iv$.

We define the following contravariant functor 
\[\Dbf\coloneqq \infty_{\Dcal_\Oscr(J)}\circ\Dbf^m:\Rep_\Oscr(J)\ra\Rep_\Oscr(J),\] 
and we consider the derived functor $\Rrm\Dbf = \infty_{\Dcal_\Oscr(J)} \circ\Rrm\Dbf^m:D^-(\Rep_\Oscr(J))\ra D^+(\Rep_\Oscr(J))$. (Recall that $\infty_{\Dcal_\Oscr(J)}$ is exact.)
\end{defnsub}

Let $K\subset J$ be a compact open subgroup. Then $\Dcal_\Oscr(K)$ is the completed group algebra over $\Oscr$, and can naturally viewed as a $\Oscr$-subalgebra of $\Dcal_\Oscr(J)$.
For any $\Dcal_\Oscr(K)$-module $M$ we set
\begin{equation}
	\cInd^{\Dcal_\Oscr(J)}_{\Dcal_\Oscr(K)} M\coloneqq  \Dcal_\Oscr(J)\otimes_{\Dcal_\Oscr(K)}M.
\end{equation}
In fact, this construction is compatible with compact induction of smooth representations of $K$ via $i_{\Dcal_\Oscr(J)}$ and $\infty_{\Dcal_\Oscr(J)}$; to be precise, we have a natural isomorphism $i_{\Dcal_\Oscr(J)}\circ\cInd^J_K\cong \cInd^{\Dcal_\Oscr(J)}_{\Dcal_\Oscr(K)}\circ i_{\Dcal_\Oscr(K)}$ of functors $\Rep_\Oscr(K)\ra\Mod(\Dcal_\Oscr(J))$, and a natural isomorphism $\infty_{\Dcal_\Oscr(J)}\circ\cInd^{\Dcal_\Oscr(J)}_{\Dcal_\Oscr(K)}\cong \cInd^J_K\circ \infty_{\Dcal_\Oscr(K)}$ of functors $\Mod(\Dcal_\Oscr(K))\ra\Rep_\Oscr(J)$. See \cite[\S1]{Fargues:Zelevinsky} and \cite[\S2]{Mieda:Zelevinsky} for the proof.

\begin{rmksub}\label{rmk:Finiteness}
Under some finiteness hypothesis on $\Rep_\Oscr(J)$, we can show that $\Rrm\Dbf$ satisfies some finiteness. We recall the precise statement  although this will not be crucially used in this paper.

Let $D^+_\ft(\Rep_\Oscr(J))$ (respectively, $D^-_\ft(\Rep_\Oscr(J))$; respectively, $D^b_\ft(\Rep_\Oscr(J))$) denote the triangulated subcategories of the derived category generated by bounded-below complexes  (respectively, 
bounded-above complexes; respectively, bounded complexes) whose cohomology is finitely generated smooth $J$-representations.

Let us make the following finiteness assumptions on $\Rep_\Oscr (J)$:
\begin{description}
\item[local noetherianness]
Any subrepresentation of finitely generated smooth $J$-representation over $\Oscr$ is again finitely generated. In particular, $\Rep_\Oscr(J)$ is locally noetherian.
\item[finite projective dimension]
Any finitely generated smooth $J$-representation over $\Oscr$ admits a finitely generated projective resolution with finitely many terms. (Note that this hypothesis is only satisfied when $\Oscr$ is a field.)
\end{description}
Both assumptions are satisfied when $\Oscr$ is a field of characteristic~$0$ and $J$ is a $p$-adic reductive group, which can be read off from \cite{Bernstein:Centre} (or alternatively, \cite{SchneiderStuhler:CohSymmSp}). If $\Oscr$ is an artin local ring of positive residue characteristic $l\ne p$, then the ``local noetherianness'' hypothesis is satisfied if $J$ is either a classical group with $p>2$ or of semi-simple rank $\leqslant1$ (\emph{cf.} \cite[Th\'eor\`eme~1.5]{Dat:Finitude}). If furthermore $\Oscr$ is a field of positive \emph{banal} characteristic (in other words, the index of any open subgroup inside an open compact subgroup of $J$ is invertible in $\Oscr$), then the ``finite projective dimension'' assumption can be verified by repeating the argument in \cite{SchneiderStuhler:CohSymmSp}.

Now, let us list some known finiteness results for $\Rrm\Dbf$:
\begin{enumerate}
	\item (\emph{cf.} \cite[Corollaire~1.8]{Fargues:Zelevinsky}) Under the ``local noetherianness'' hypothesis, the derived functor $\Rrm\Dbf$ restricts to $\Rrm\Dbf:D^+_\ft(\Rep_\Oscr(J)) \ra D^-_\ft(\Rep_\Oscr(J))$.
	\item (\emph{cf.} \cite[Proposition~1.18]{Fargues:Zelevinsky}) Under the ``local noetherianness'' and ``finite  projective dimension'' hypotheses, $\Rrm\Dbf$ restricts to $\Rrm\Dbf:D^b_\ft(\Rep_\Oscr(J)) \ra D^b_\ft(\Rep_\Oscr(J))$. Furthermore, the natural morphism of functors $\id \ra \Rrm\Dbf\circ\Rrm\Dbf$ is an isomorphism.
\end{enumerate}
\end{rmksub}

Let us now state the theorem of Mieda's:
\begin{thmsub}[{\cite[Theorems~4.1, 4.11]{Mieda:Zelevinsky}}]\label{thm:Poincare}
	Let $\Xfr$ be a formal scheme locally formally of finite type over $O$, and let $\Oscr$ be $\overline\QQ_l$, a finite extension of $\QQ_l$, or an artinian quotient of a finite extension of $\ZZ_l$. Let us assume the following:
	\begin{enumerate}
		\item\label{thm:Poincare:sm} The adic space generic fibre $\Xcal\coloneqq \Xfr^\ad_C$ is smooth pure of dimension~$d$.
		\item\label{thm:Poincare:cont} The locally pro-$p$ group $J$ acts continuously on $\Xfr$ in the sense of \cite[D\'efinition~2.3.10]{Fargues:AsterisqueLLC}.
		\item\label{thm:Poincare:ParProper} The underlying reduced scheme $\Xfr^\red$ is partially proper over $\kappa$.
		\item\label{thm:Poincare:Finiteness} There exists a quasi-compact open subset $U\subset\Xfr^\red$ such that $\Xfr^\red = \bigcup_{g\in J}gU$, and the subgroup $\{g\in J|\ g\overline U\cap \overline U \ne\emptyset\}\subset J$ is compact.
		\item\label{thm:Poincare:LocAlg} The formal scheme $\Xfr$ is locally algebraisable in the sense of \cite[Definition~3.19]{Mieda:FormalNearbyCycles}; in other words, there exists a Zariski open covering $\{\Ufr\}$ of $\Xfr$ such that each $\Ufr$ can be obtained as a completion of some $O$-scheme.
	\end{enumerate}
	Then we have a natural $J$-equivariant isomorphism for each $i\in \ZZ$: 
	\[
	\coh{2d-i}_{\Ccal_\Xfr}(\Xcal,\Oscr)(d) \cong \Rrm^{-i}\Dbf(\Rrm\Gamma_c(\Xcal,\Oscr)).
	\]
	(In particular, the isomorphism does not depend on the choice of $U$ in (\ref{thm:Poincare:Finiteness}).)
\end{thmsub}
\begin{rmksub}\label{rmk:Poincare}
When $\Xfr$ is proper and $J$ is trivial, the above theorem reduces to the usual Poincar\'e duality (as $\Xcal$ is required to be smooth and equi-dimensional).	

Let us comment on the assumptions on $\Xfr$.
We need the $J$-action on $\Xcal$ to extend continuously to the formal model $\Xfr$ in order to be able to define the natural $J$-action on $\coh i_{\Ccal_\Xfr}(\Xcal,\Oscr)$. 

The last two conditions of Theorem~\ref{thm:Poincare} are to ensure that $\Rrm\Gamma_c(\Xcal,\Oscr)\in D^b_\ft(\Rep_\Oscr(J))$. If $K\subset J$ is a compact open subgroup stabilising $\spbf\iv(\overline U)^\circ$ (which exists by Berkovich's Key Lemma), then Theorem~\ref{thm:Poincare}(\ref{thm:Poincare:Finiteness}) ensures that $\{g\cdot\spbf\iv(\overline U)^\circ\}_{g\in J/K}$ is a locally finite open covering of $\Xcal$. Using this, we consider a \v{C}ech resolution of the constant sheaf $\Oscr$; \emph{cf.} \cite[Proposition~3.27 i)]{Mieda:Zelevinsky}. Now, by \v{C}ech-to-derived functor spectral sequence, the desired claim $\Rrm\Gamma_c(\Xcal,\Oscr)\in D^b_\ft(\Rep_\Oscr(J))$ reduces to showing that for any quasi-compact closed subset $Z\subset\Xfr^\red$, we have 
\[
\Rrm\Gamma_c(\spbf\iv(Z)^\circ,\Oscr)\in D^b_{\ft}(\Mod(\Oscr));
\]
\emph{cf.} \cite[Remark~4.12]{Mieda:Zelevinsky}.
It is plausible that this statement should hold in general, but we only know how to prove this claim when $\Xfr$ is locally algebraisable; \emph{cf.} \cite[Proposition~3.21, Theorem~4.35]{Mieda:FormalNearbyCycles}.

Let us now outline the main idea of proof of Theorem~\ref{thm:Poincare}, following \cite[\S~4]{Mieda:Zelevinsky}. Just like the usual non-equivariant cohomology, there are some convenient $J$-equivariant resolutions of $\Fscr\in\Oscr\hyph\widetilde\Xcal_\et/J$, such as the Godement resolution (\emph{cf.} \cite[\S3.3.3]{Mieda:Zelevinsky}) and some variant of \v{C}ech resolutions (\emph{cf.} \cite[Proposition~3.27]{Mieda:Zelevinsky}). The theorem is proved by explicitly constructing certain complexes that compute $\coh{2d-i}_{\Ccal_\Xfr}(\Xcal,\Oscr)(d)$ and $\Rrm^{-i}\Dbf(\Rrm\Gamma_c(\Xcal,\Oscr))$, and match them.
In a sense, what is actually proved is the following (stronger) isomorphism: 
\[
\Rrm\Gamma_{\Ccal_\Xfr}(\Xcal,\Oscr)(d)[2d] \cong \Rrm\Dbf(\Rrm\Gamma_c(\Xcal,\Oscr)) \in D^b_\ft(\Rep_\Oscr(J)).	
\]
Note that we have not defined $\Rrm\Gamma_{\Ccal_\Xfr}(\Xcal,\Oscr)$ in $D^+(\Rep_\Oscr(J))$ (but only in $D^b_\ft(\Mod(\Oscr))$), but in the proof,  Mieda actually constructed a complex of smooth $J$-representations which represents $\Rrm\Gamma_{\Ccal_\Xfr}(\Xcal,\Oscr)$.
\end{rmksub}

\subsection{Derived duality on towers with Hecke correspondence}\label{subsec:ZelevinskyTower}

Let $G$ and $J$ be locally pro-$p$ groups.  In this section, we always work with $\overline{\QQ}_l$-coefficients and write $
\Hcal(J) = \Hcal_{\overline{\QQ}_l}(J)$ and $\Dcal(J) = \Dcal_{\overline{\QQ}_l}(J)$, etc. To avoid confusion, we specify the group in the notation 
\[\Rrm\Dbf_J(V)\coloneqq \infty_{\Dcal(J)}\circ\RHom_{\Dcal(J)}(i_{\Dcal(J)}(V),\Hcal(J)),\]
instead of denoting it by $\Rrm\Dbf$ as in the previous subsections.

We consider  $V\in\Rep(G\times J)$. In the intended situation, $V$ will be obtained from the compactly supported cohomology of some Rapoport-Zink tower, where $G$ and $J$ are some $p$-adic reductive groups. 
Given any $\rho\in\Rep_{\overline{\QQ}_l}(J)$,  we would like to extract some kind of ``derived $\rho$-isotypic parts'' from $V$ as \emph{smooth} $G$-representations. Note that the natural $G$-action on $\Hom_J(V,\rho)$ is not necessarily smooth, and the ``smoothening'' functor $\infty_G$ is not exact in general.
Under some suitable finiteness assumption, we will now fix this problem using a similar idea to the definition of $\Rrm\Dbf_J$.

We first claim that we have a natural isomorphism
\begin{equation}\label{eqn:DistOnProdGrps}
\Dcal(G)\otimes\Dcal(J) \riso \Dcal(G\times J).
\end{equation}
Indeed, this isomorphism essentially boils down to the following natural isomorphism 
\begin{multline*}
\Hcal(G\times J) = \varinjlim_{\Ksf_G\times\Ksf_J} \Hcal(G\times J//\Ksf_G\times\Ksf_J) \\
\cong \big(\varinjlim_{\Ksf_G} \Hcal(G//\Ksf_G)\big)\otimes\big(\varinjlim_{\Ksf_J} \Hcal(J//\Ksf_J)\big) = 
\Hcal(G)\otimes\Hcal(J),
\end{multline*}
where the nontrivial isomorphism is obtained from $(\Ksf_G\times\Ksf_J)\backslash (G\times J) /(\Ksf_G\times\Ksf_J) = (\Ksf_G\backslash G/\Ksf_G)\times (\Ksf_J\backslash J/\Ksf_J)$.

To simplify the notation, we view $\Rep(J)$ and $\Rep(G\times J)$ as full subcategories of $\Mod(\Dcal(J))$ and $\Mod(\Dcal(G\times J))$ via $i_{\Dcal(J)}$ and $i_{\Dcal(G\times J)}$, respectively. 
By (\ref{eqn:DistOnProdGrps}) we may view $V$ (or rather, $\iota_{\Dcal(G\times J)}V$) as a module over $\Dcal(G)\otimes\Dcal(J)$. Therefore, for any $\rho\in\Rep(J)$ the following extension group
\[
\Ext^i_{\Dcal(J)}(V,\rho) (= \Ext^i_{\Dcal(J)}(\iota_{\Dcal(G\times J)}V,\iota_{\Dcal(J)}\rho) )
\]
has a natural structure of a $\Dcal(G)$-module.

Let us now make the following definition:
\begin{defnsub}\label{def:Ecal}
For any $\rho\in\Rep(J)$, we define 
\[
\Ecal^i(V,\rho)\coloneqq \infty_{\Dcal(G)}\Ext^i_{\Dcal(J)}(V,\rho) \in \Rep_G(\overline{\QQ}_l).
\]
\end{defnsub}

\begin{propsub}\label{prop:Infty2Fin}
Let us assume the following properties:
\begin{enumerate}
\item\label{subsec:ZelevinskyHecke:G} For any open pro-$p$ subgroup $K\subset G$, $V^K$ is finitely generated smooth $J$-representation.
\item\label{subsec:ZelevinskyHecke:J} For a finitely generated smooth $J$-representation, any $J$-subrepresentation is also finitely generated. (\emph{Cf.} the ``local noetherianness'' hypothesis in Remark~\ref{rmk:Finiteness}.)
\end{enumerate}
Then for any  $\rho\in\Rep(J)$ we have
\[
\Ecal^i(V,\rho)^K = \Ext^i_J(V^K,\rho)
\]
In particular, we have
\[
\Ecal^i(V,\rho) \cong \varinjlim_{K\subset G}\Ext^i_J(V^K,\rho),
\]
and it can be identified with the subspace of $G$-smooth vectors in $\Ext^i_J(V,\rho)$.
\end{propsub}
We remark that (\ref{subsec:ZelevinskyHecke:G}) is satisfied if $V^K = \coh i_c(\Xfr_{K,C}^{\red},\overline{\QQ}_l)$ where $\Xfr_K$ is a formal scheme with continuous $J$-action that satisfies the conditions in Theorem~\ref{thm:Poincare}. In practice,  (\ref{subsec:ZelevinskyHecke:J}) is satisfied if $J$ is a $p$-adic reductive group; \emph{cf.} Remark~\ref{rmk:Finiteness}.
\begin{proof}
To ease the notation, we continue to identify $V\in\Rep(G\times J)$ with $\iota_{\Dcal(G\times J)}(V)$.  We have
\[
e_K*\Ext^i_{\Dcal(J)}(V,\rho) \cong \Ext^i_{\Dcal(J)}(e_K*V,\rho) = \Ext^i_{\Dcal(J)}(V^K,\rho) \cong \Ext^i_J(V^K,\rho).
\]
Indeed, the first isomorphism holds since $e_K\in\Dcal(G)$ is an idempotent, and the last isomorphism follows from Lemma~\ref{lem:Ext} below. Recall that for any smooth $G$-representation $V$ we have $V^K = e_K*V$. 

This shows that 
\[
\Ecal^i(V,\rho) = \varinjlim_K e_K*\Ext^i_{\Dcal(J)}(V,\rho)  \cong \varinjlim_K\Ext^i_J(V^K,\rho).
\]

Now it remains to show that $\Ecal^i(V,\rho)^K=\Ext^i_J(V^K,\rho) $. Indeed, we have 
\[
\Ecal^i(V,\rho)^K = \varinjlim_{K'\subset K}(e_{K'}*\Ext^i_{\Dcal(J)}(V,\rho))^{K/K'} \cong  \varinjlim_{K'\subset K}\Ext^i_J(V^{K'},\rho)^{K/K'} = \Ext^i_J(V^K,\rho),
\]
where $K'$ runs through open normal subgroups of $K$. (To see the last isomorphism, since the invariance under the finite group action on $\QQ$-vector space is exact, we can commute the $K/K'$-invariance and $\Ext^i_J$.)
This concludes the proof.
\end{proof}

\begin{lemsub}\label{lem:Ext}
Assume that $\Rep(J)$ satisfies Proposition~\ref{prop:Infty2Fin}(\ref{subsec:ZelevinskyHecke:J}).
Then for any $V,\rho\in\Rep(J)$ with $V$ finitely generated,  we have a natural isomorphism
\[
 \RHom_J(V,\rho) \cong \RHom_{\Dcal(J)}(V,\rho) \cong \Rrm\Dbf_J(V)\otimes^\Lrm_{\Hcal(J)} \rho.
\]
\end{lemsub}
\begin{proof}
Let us first assume that  $V = \cInd_K^J\tau$ for some pro-$p$ open subgroup $K\subset J$ and a finitely generated representation $\tau\in\Rep(K)$. Note that $\cInd_K^J\tau$ is a projective object in $\Rep(J)$; indeed, it suffices to show that $\tau$ is a projective object in $\Rep(K)$ by Frobenius reciprocity, and this follows since the category of  finitely generated smooth $K$-representations over $\overline\QQ_\ell$ is semi-simple ($K$ being profinite).

Then we have 
\begin{equation}\label{eqn:c-ind}
\Rrm\Dbf_J( \cInd_K^J\tau) \cong   \cInd_K^J(\tau)^*;
\end{equation}
indeed, the vanishing of the higher cohomology of the left hand side follows from
 \cite[Corollaire~1.8]{Fargues:Zelevinsky} together with projectivity of  $\cInd_K^J\tau$, and the computation of the zeroth cohomology follows from  \cite[Proposition~1.16]{Fargues:Zelevinsky}.

Next, we show the following isomorphism
\begin{multline*}
\Rrm\Dbf( \cInd_K^J\tau)\otimes^\Lrm_{\Hcal(J)} \rho \cong  \cInd_K^J(\tau^*) \otimes_{\Hcal(J)} \rho \\
\cong \Hom_J ( \cInd_K^J\tau,\rho) \cong \RHom_J ( \cInd_K^J\tau,\rho).
\end{multline*}
The first and last isomorphisms follow from the projectivity of compact induction together with (\ref{eqn:c-ind}). To show the remaining isomorphism, note that we have
\[
\Hom_J ( \cInd_K^J\tau,\rho) \cong \Hom_K(\tau,\rho) \cong (\tau^*\otimes \rho)^K,
\]
and analogously we have
\[
\cInd_K^J(\tau^*) \otimes_{\Hcal(J)} \rho  \cong (\tau^* \otimes \rho)_K.
\]
So our claim follows from $(\tau^* \otimes \rho)_K \cong (\tau^* \otimes \rho)^K$.

It  easily follows that the lemma holds for a bounded-above complex  $V^\bullet$ where each $V^{-i}$ is of the form $\cInd_K^J\tau$ for some open pro-$p$ subgroup $K\subset J$ and a finitely generated representation $\tau\in\Rep_K(\overline{\QQ}_l)$. It now remains to show that  any finitely generated $J$-representation $V$ admits such a resolution. Indeed, we choose finitely many generators of $V$ and let $K$ be an open pro-$p$ group that fixes each of them. Thus we have a surjection from a finite sum $\bigoplus \cInd_K^J\mathbf{1}$ onto $V$, and by assumption any $J$-subrepresentation of $\cInd_K^J\mathbf{1}$ is finitely generated. This concludes the proof.
\end{proof}

\setcounter{axiom}{0}
  \section{Integral models for Shimura varieties of Hodge type} \label{sect kisin-pappas}

 In this section we  recall the construction of integral models for Shimura varieties of Hodge type with Bruhat-Tits level structure at $p$ as presented in \cite{KisinPappas:ParahoricIntModel} and recall the Newton stratification on the special fibre.

\subsection{Review of Dieudonn\'e display theory}\label{ssect displays}

We  review and set up the notation for the theory of Dieudonn\'e displays for \emph{$p$-torsion free} complete local noetherian rings with residue field $\k$, following \cite[\S3.1]{KisinPappas:ParahoricIntModel}. For more standard references for Dieudonn\'e display theory, we refer to Zink's original paper \cite{Zink:DieudonnePDivGpCFT} and Lau's paper \cite{Lau:DisplayCrystals}.

Let $R$ be a complete local noetherian ring with residue field $\k$. In \cite[\S2]{Zink:DieudonnePDivGpCFT}, Zink introduced a $p$-adic subring $\WW(R)$ of the ring of Witt vectors $W(R)$, which fits in the following short exact sequence (of algebras not necessarily with unity):
\[
\xymatrix@1{
0 \ar[r] & \widehat W(\mfr_R) \ar[r] & \WW (R) \ar[r] & W(\Fpbar) \ar[r] &0,
}
\]
where $\widehat W(\mfr_R)\coloneqq \{(a_i)\in W(\mfr_R)|\ a_i\to 0\}$.(Here, $W(\mfr_R)$ is a mild but standard abuse of notation for the kernel of $W(R)\twoheadrightarrow W(\Fpbar)$.) If $p>2$ then $\WW(R)$ is stable under the Frobenius and Verschiebung operators of $W(R)$. Recall that $\WW(R)$ is $p$-torsion free if $R$ is $p$-torsion free (since in that case $W(R)$ is $p$-torsion free).

Let $\II_R$ denote the kernel of the natural projection $\WW(R)\thra R$, which coincides with the injective image of the Verschiebung operator (if $p>2$). We let $\sigma: \WW(R)\to\WW (R)$ denote the Witt vector Frobenius. Furthermore  if $p>2$ then we have a canonical divided power structure $\delta: \II_R \to \WW(R)$ inherited by restricting the canonical divided power structure on $W(R)$ (\cite[Lemma~1.16]{Lau:DisplayCrystals}), which makes $\WW(R)$ a $p$-adic divided power thickening of $R$. (In particular, one can evaluate a crystal over $R$ at $\WW(R)$.)

\begin{defnsub}\label{def:display}
A \emph{Dieudonn\'e display} over $R$ is a tuple
\[
(M, M_1, \Phi, \Phi_1),
\]
where
\begin{enumerate}
\item $M$ is a finite free $\WW(R)$-module;
\item $M_1\subset M$ is a $\WW(R)$-submodule containing $\II_R M$, such that $M/M_1$ is free over $R$;
\item $\Phi:M\to M$ and $\Phi_1:M_1\to M$ are $\sigma$-linear morphism such that $p\Phi_1 = \Phi|_{M_1}$ and the image of $\Phi_1$ generates $M$.
\end{enumerate}
The \emph{Hodge filtration} associated to the display is $M_1/\II_R M \subset M/\II_R M$ viewed as the $0$th filtration (so $M/M_1$ is viewed as the $(-1)$th grading)\footnote{We would like to view $M/\II_R M$ as the ``first de~Rham \emph{homology}'' of some {\BT}, and $M/M_1$ as the Lie algebra of the {\BT}, which is the $(-1)$th grading of the first de~Rham homology. We follow the standard notation to let $M_1\subset M$ denote the $\WW(R)$-submodule defining the Hodge filtration (although $M^0$ may be more natural notation as it defines the $0$th filtration).}.
\end{defnsub}

Now, for any $\WW(R)$-module $M$, we write $M^\sigma\coloneqq M\otimes_{\WW(R),\sigma}\WW(R)$. 
Given a Dieudonn\'e display $(M,M_1,\Phi,\Phi_1)$ over $R$, we let $\widetilde M_1$ denote the image of $M_1^\sigma$ in $M^\sigma$. Then the linearisation of $\Phi_1$ induces  the following isomorphism
\[
\Psi:\widetilde M_1 \riso M.
\]
If $\WW(R)$ is $p$-torsion free (for example, if $R$ is $p$-torsion free), then given $M$ and $M_1\subset M$ where $M/M_1$ is projective over $R$, giving a $\sigma$-linear map $\Phi$ and $\Phi_1$ that makes $(M,M_1,\Phi,\Phi_1)$ a Dieudonn\'e display is equivalent to giving an isomorphism $\Psi$ as above; \emph{cf.} \cite[Lemma~3.1.5]{KisinPappas:ParahoricIntModel}.

For any complete local noetherian ring $R$ with perfect residue field of characteristic $p>2$,
Zink \cite{Zink:DieudonnePDivGpCFT} constructed a natural equivalance between the category of {\BT}s over $R$ and the category of Dieudonn\'e displays over $R$. We denote this functor by 
\[\Xscr\rightsquigarrow M(\Xscr)\]
for a {\BT} $\Xscr$  over $R$. Lau proved that $M(\Xscr)$ is naturally isomorphic to  $\DD(\Xscr)_{\WW(R)}^\vee$, where $\DD(\Xscr)$ denotes the contravariant Dieudonn\'e crystal of $\Xscr$. Moreover, the isomorphism matches the Hodge filtrations on both sides and allows us to recover the crystalline Frobenius and Verschiebung maps from $\Phi$ and $\Phi_1$ on $M(\Xscr)$
\cite[Theorem~B]{Lau:DisplayCrystals}. If $\WW(R)$ is $p$-torsion free, then we can obtain the Dieudonn\'e display $M(\Xscr)$ from the filtered Dieudonn\'e crystal $\DD(\Xscr)$ as follows. (For simplicity, we will write $M(\Xscr)\coloneqq (M,M_1,\Phi,\Phi_1)$.)
\begin{enumerate}
\item $M\coloneqq \DD(\Xscr)_{\WW(R)}^\vee$ and $M_1\coloneqq \ker \left(M \thra \DD(\Xscr)_R^\vee \thra \Lie(\Xscr)\right)$;
\item The contragradient of the crystalline Frobenius $F:M^\sigma[\frac{1}{p}]\to M[\frac{1}{p}]$ restricts to  $\Psi:\widetilde M_1\to M$.
\end{enumerate}

 \subsection{Group theoretic background} \label{sect group theoretic background}

 Let $G$ and $G'$ be connected reductive groups over $\QQ_p$ where either $G' = \GL_n$ or $G' = \GSp_n$. Assume that we have a faithful minuscule representation $\rho:G \mono \GL_n$ which factorises through $G'$. By \cite[1.2.26, 2.3.3]{KisinPappas:ParahoricIntModel} this induces a toral Galois-equivariant embedding of Bruhat Tits buildings $\iota: \Bcal(G,\breve\QQ_p) \mono \Bcal(G',\breve\QQ_p)$, depending on some auxiliary choices, where $\breve\QQ_p$ denotes the completion of the maximal unramified extension of $\QQ_p$. Now assume $x \in \Bcal(G,\breve\QQ_p)^{\Gal(\breve\QQ_p/\QQ_p)}$ with image $x' \in \Bcal(G',\breve\QQ_p)$. We denote by $\Gscr_x,\Gscr'_{x'}$ and $\Gscr_x^\circ,\Gscr_{x'}'^\circ = \Gscr'_{x'}$ the Bruhat-Tits group schemes over $\ZZ_p$ given by the stabilisers of $x, x'$ and their corresponding parahoric group schemes, respectively.

 \begin{propsub}[{\cite[Prop.~1.3.3]{KisinPappas:ParahoricIntModel}}]
  The embedding $G \mono G'$ extends to $\Gscr_x \mono \Gscr'_{x'}$.
 \end{propsub}

 Let $\{\mu\}$ be a geometric conjugacy class of minuscule cocharacters of $G$ and let $E$ be its field of definition. Denote by $Fl_\mu = G/P_\mu$, where $P_\mu$ is a parabolic subgroup corresponding to an element $\mu \in \{\mu\}$.
  Denote by $\Mscr^{\rm loc}_{G,\{\mu\},x}$ the local model given by Pappas and Zhu, which is a certain integral model of $Fl_\mu$ parametrising the singularities of the integral model of the Shimura variety. Rather than recalling its definition here, we give an explicit description for the cases in which it occurs below.

 \begin{propsub}[{\cite[Prop.~2.3.7]{KisinPappas:ParahoricIntModel}}]
  Assume that $\{\mu'\} \coloneqq \rho_\ast \{\mu\}$ is minuscule. Then the natural inclusion $Fl_{\{\mu\}} \mono Fl_{\{\mu'\}}$ extends to a closed embedding $\Mscr^{\rm loc}_{G,\{\mu\},x} \mono \Mscr^{\rm loc}_{G',\{\mu'\},x'} \otimes_{\ZZ_p} O_E$.
 \end{propsub}

 We assume from now on that $x'$ is hyperspecial. That is, $\Gscr'_{x'} = \GL(\Lambda)$ if $G' = \GL_n$ or $\Gscr'_{x'} = \GSp(\Lambda)$ if $G' = \GSp_n$ where $\Lambda \subset \QQ_p^n$ is some (self-dual) lattice. Now the local models have the following description (see \cite[4.1.5]{KisinPappas:ParahoricIntModel}). Since $\{\mu'\}$ is defined over $\ZZ_p$,  $Fl_{\{\mu'\}}$ admits a model ${\Flag}_{\{\mu'\}}$ over $\ZZ_p$. We have $\Mscr^{\rm loc}_{G,\{\mu'\},x'} = {\Flag}_{\{\mu'\}}$. Thus by above proposition,
 \[
  \Mscr^{\rm loc}_{G,\{\mu\},x} = \overline{G\cdot y} \subset \Mscr^{\rm loc}_{G',\{\mu'\},x'} \otimes_{\ZZ_p} O_E = {\Flag}_{\{\mu'\}}\otimes_{\ZZ_p} O_E
 \]
 where $y$ is a point corresponding to a cocharacter in $\{\mu\}$.

 \subsection{Local geometry} \label{sect local geom}

 With the notation above, assume that $\mu' = (0^{(n-d)}, -1^{d})$ for some $d$ and furthermore $d = n/2$ if $G' = \GSp_n$. Let $X$ be a {\BT} of height $n$ and dimension $d$ over $\FFbar_p$, together with a polarisation in the case $G' = \GSp_n$. We choose an identification $\DD(X)_{\breve\ZZ_p}^\vee \cong \Lambda \otimes \breve\ZZ_p$ which we assume to be a symplectic similitude in the case that $X$ is polarised. By Grothendieck-Messing theory, its deformation space $\Def(X)$ is canonically identified with the formal neighbourhood $\widehat\Mscr^{\rm loc}_{G', \{\mu'\},x',y'}$ of $\hat\Mscr^{\rm loc}_{G', \{\mu'\},x'}$ at the point $y'$ which corresponds to the Hodge filtration $\Fil^0(X) \subset \DD(X)_{\FFbar_p}^\vee$. (We regard $\DD(X)_{\k}^\vee/\Fil^0(X) = \Lie(X)$ as the $(-1)$th grading of the Hodge filtration. This is compatible with the indexing of Hodge filtrations for Dieudonn\'e displays; \emph{cf.} Definition~\ref{def:display}.)
We assume that  the following conditions hold.
\begin{itemize}
\item $G$ splits after some tame extension of $\Qp$ and $p\nmid|\pi_1(G^{\rm der})|$.
\item The adjoint group of $G$ does not have any factor of type $E_8$.
\item  The image of the embedding $G\mono G'$ contains the scalar matrices; \emph{cf.} \cite[(3.2.4)]{KisinPappas:ParahoricIntModel}.
\end{itemize}
These assumptions imply the technical hypothesis \cite[(3.2.3)]{KisinPappas:ParahoricIntModel} on extending $\Gscr_x^\circ$-torsors over the closed point of the spectrum of a $2$-dimensional regular local ring by \cite[Proposition~1.4.3]{KisinPappas:ParahoricIntModel}. Furthermore, 
$\Mscr^{\rm loc}_{G,\{\mu\},x}$ is known to be normal, and its special fibre $\Mscr^{\rm loc}_{G,\{\mu\},x}\otimes_{O_E}\kappa_E$ is reduced and each irreducible component is normal and Cohen-Macaulay; \emph{cf.} \cite[Theorem~9.1]{PappasZhu:LocMod}.

 We choose a family of tensors $(s_\alpha)$ in $\Lambda^\otimes$ such that $\Gscr_x$ is their stabiliser. Now assume that $X$ is equipped with a family of crystalline Tate-tensors $(t_\alpha)$ satisfying the following conditions.
 \begin{assertionlist}
  \item There exists an isomorphism $\DD(X)_{\breve\ZZ_p}^\vee \cong \Lambda \otimes \breve\ZZ_p$ as above which sends $t_\alpha$ to $s_\alpha \otimes 1$.
  \item The point $y \in \Mscr^\loc_{G',\{\mu'\},x'}$ given by $\Fil^0(X)$ lies in $\Mscr^{\rm loc}_{G,\{\mu\},x}$.
 \end{assertionlist}
 Note that while the point $y$ depends on our choice of isomorphism in (a), it does so only up to $\Gscr_x(\ZZ_p)$-action; in particular condition (b) is independent of this choice.

In this setting, the technical hypotheses (3.2.2)--(3.2.4) in \cite{KisinPappas:ParahoricIntModel} are satisfied so we can apply the result in \cite[\S~3]{KisinPappas:ParahoricIntModel} that gives a deformation-theoretic interpretation of the formal neighbourhood of $\Mscr^{\rm loc}_{G,\{\mu\},x}$ at a closed point.
\begin{defnsub}\label{def KP defor space}
 We denote by $\Def_{\Gscr_x}(X;(t_\alpha)) \subset \Def(X)\otimes_{\W}O_{\breve{E}}$ the formal subscheme corresponding to the inclusion $\widehat\Mscr^{\rm loc}_{G,\{\mu\},x,y} \subset \widehat\Mscr^{\rm loc}_{G',\{\mu'\},x',y}\otimes_{\breve\ZZ_p}O_{\breve E}$ of formal neighbourhoods of the local models. If $(t_\alpha)$ is understood, we write  $ \Def_{\Gscr_x}(X)\coloneqq\Def_{\Gscr_x}(X;( t_\alpha ))$.

 Let $R$ and $R_G$ be the rings of global sections of $\Def(X)$ and $\Def_{\Gscr_x}(X)$, respectively. Denote by $\Xscr^{\rm def}$ the universal deformation of $X$ over $R$ and denote by $\Xscr_G^{\rm def}$ its restriction to $R_G$.
 
 \end{defnsub}
 Note that $R_G$ is \emph{normal} (since we assumed that $\Mscr^{\rm loc}_{G,\{\mu\},x}$ is normal). Under this normality assumption, Kisin and Pappas constructed $\Phi$-invariant tensors $t_{\alpha}^{\rm def} \in M(\Xscr^{\rm def}_G)^\otimes$ lifting $t_\alpha$ so that there exists a $\WW(R_G)$-linear isomorphism  $M(\Xscr^{\rm def}_G)\cong M(X) \otimes_{\W} \WW(R_G) $ sending $(t^{\rm def}_{\alpha})$ to $(t_\alpha\otimes1)$.

To explain the ``universal property'' of the tensors $(t_{\alpha}^{\rm def})$, let us set up some notations. For any finite extension $K/\breve{E}$ and $\Xscr \in \Def(X)(O_K)$, we have a unique $\Psi$-equivariant isomorphism
\begin{equation}\label{eq rigidity quasiisogeny}
M(X) \otimes_{\W} \WW(O_K)_\QQ \cong M(\Xscr)_\QQ
\end{equation}
lifting the identity map on $M(X)_\QQ$ (\emph{cf.} \cite[Lem.~3.1.17]{KisinPappas:ParahoricIntModel}). Let  $u_\alpha\in M(\Xscr)^\otimes_\QQ$ be the tensors obtained as the images of $t_\alpha\otimes 1$.

\begin{propsub}[{\cite[Cor.~3.2.11,~Prop.~3.2.17]{KisinPappas:ParahoricIntModel}}]
Let $\Xscr \in \Def(X)(O_K)$ for some finite extension $K/\breve{E}$, and define the tensors $u_\alpha\in M(\Xscr)^\otimes_\QQ$ as above. Then we have $\Xscr \in \Def_{\Gscr_x}(X)(O_K)$ and the tensors $t_{\alpha}^{\rm def} \in M(\Xscr_G^{\rm def})^\otimes$ pull back to $u_\alpha\in M(\Xscr)^\otimes_\QQ$, if and only if the tensors $u_\alpha\in M(\Xscr)^\otimes_\QQ$ satisfy the following conditions
  \begin{subenv}
   \item We have $u_\alpha\in M(\Xscr)^\otimes$; i.e.,~$u_\alpha$ are integral.
    \item There exists  a $\WW(O_K)$-linear isomorphism $M(\Xscr) \cong  M(X) \otimes_{\W} \WW(O_K)$ sending $u_\alpha$ to $t_\alpha\otimes1$.
    \item We have $(\Fil^0 \Xscr) \otimes_{O_K} K \in \Mscr^{\rm loc}_{G,\{\mu\},x}(K)$, where we view $(\Fil^0 \Xscr) \otimes_{O_K} K$ as a filtration of $M(X)_\QQ\otimes_{\breve{\QQ}_p}K$ (hence, a $K$-point of $\Mscr^{\rm loc}_{G',\{\mu'\},x}$) via the isomorphism (\ref{eq rigidity quasiisogeny}).
  \end{subenv}
 \end{propsub}

 For any {\BT} $\Xscr$ over some ring $R$ (with $R[\frac{1}{p}]\ne0$), we let $T(\Xscr)$ denote the integral Tate module (viewed as a $\ZZ_p$-local system over $\Spec R[\frac{1}{p}]$).
Let us recall the following variant for \'etale tensors.
\begin{propsub}[{\cite[Prop.~3.3.13]{KisinPappas:ParahoricIntModel}}]\label{prop et univ prop}
For any finite extension $K/\breve{E}$ and $\Xscr \in \Def(X)(O_K)$, assume that there exists $\Gal(\ol K/K)$-invariant tensors $t_{\alpha,\et} \in T(\Xscr)^\otimes$ with the following properties.
  \begin{subenv}
   \item \label{prop et univ prop comparison} The $p$-adic comparison isomorphism matches $t_{\alpha,\et}$ with $t_\alpha\in M(X)_\QQ^\otimes$.
    \item Viewing $(\Fil^0 \Xscr) \otimes_{O_K} K\in\Mscr^{\rm loc}_{G',\{\mu'\},x}(K)$ via the isomorphism (\ref{eq rigidity quasiisogeny}), we have $(\Fil^0 \Xscr) \otimes_{O_K} K \in \Mscr^{\rm loc}_{G,\{\mu\},x}(K)$.
  \end{subenv}
Then we have $\Xscr \in \Def_{\Gscr_x}(X)(O_K)$.
\end{propsub}
%

 \subsection{Integral models} \label{sect integral models}

 Let $(\Gsf,\Xsf)$ be Shimura datum of Hodge type and $\Ksf = \Ksf_p\Ksf^p \subset \Gsf(\AA_f)$ be a small enough open compact subgroup where $\Ksf_p \subset \Gsf(\QQ_p)$ is the stabiliser of a point $x \in \Bcal(\Gsf,\QQ_p)$.  Denote by $\Esf$ the Shimura field (which is a subfield of $\CC$), and let $E$ be its $p$-adic completion for some fixed embeddings $\overline\QQ \mono \CC$ and 
 $\overline\QQ \mono \overline\QQ_p$. Let $O_E$ and $\kappa_E$ respectively denote the ring of integers and the residue field of $E$.

We assume that $G \coloneqq \Gsf_{\QQ_p}$ splits over a tame extension of $\QQ_p$.
 Since $(\Gsf,\Xsf)$ is of Hodge type, there exists an embedding of Shimura data $(\Gsf,\Xsf) \mono ({\sf GSp}_{2g},\Ssf^\pm)$. Using the assumption that $G$ splits over a tame extension of $\QQ_p$, we may alter this embedding such that the image $x' \in \Bcal(\GSp_{2g},\QQ_p)$ of $x\in\Bcal(G,\QQ_p)$  is hyperspecial (\cite[4.1.5]{KisinPappas:ParahoricIntModel} and Zarhin's trick, cf.\ \cite[(1.3.3)]{Kisin:LanglandsRapoport}). Let $\Lambda\subset\QQ_p^{2g}$ denote the unimodular lattice corresponding to $x'$. It follows from \cite[Lemma~1.3.3]{KisinPappas:ParahoricIntModel} that the closed immersion $G\hookrightarrow \GSp_{2g}$ over $\QQ_p$ extends to a closed immersion
\begin{equation}\label{eq PrasadYu}
 \Gscr_x\hookrightarrow  \GSp(\Lambda)
\end{equation}
of smooth algebraic groups over $\ZZ_p$. In particular, we have a closed immersion of the local models $\Mscr^{\rm loc}_{G,\{\mu\},x} \hookrightarrow\Mscr^{\rm loc}_{\GSp_{2g},\{\mu'\},x'}\otimes_{\ZZ_p}O_E$.

In addition to the assumption that $G=\Gsf_{\QQ_p}$ splits over a tame extension of $\QQ_p$, we assume that  $p \nmid |\pi_1(G^{\rm der}) |$. Note that the adjoint group of $G$ cannot have a factor of type $E_8$ (by classification), and the image of the embedding $G\mono \GSp_{2g}$ contains all the scalar matrices. Therefore, the technical conditions imposed at the beginning of \S\ref{sect local geom} are satisfied.

 We choose a unimodular $\ZZ$-lattice $\Lambda_\ZZ \subset \Lambda\cap \QQ^{2g}$ with respect to the standard symplectic pairing. Using this choice, we can interpret the Shimura variety for $(\mathsf{GSp}_{2g},\Ssf^\pm)$ as the moduli space of principally polarised abelian varieties of dimension~$2g$ with some symplectic level structure.

Given $\Ksf\subset\Gsf(\AA_f)$ as in the beginning of the section, we choose $\Ksf' = \Ksf'_p\Ksf'^p \subset {\sf GSp}_{2g}(\AA_f)$ to be a small enough open subgroup such that $\Ksf'_p \subset \GSp_{2g}(\QQ_p)$ is the  stabiliser of $x'$, $\Ksf'^p \subset {\sf GSp}_{2g}(\AA_f^p)$ contains $\Ksf^p$, and $\Sh_\Ksf(\Gsf,\Xsf) \to \Sh_{\Ksf'}({\sf GSp}_{2g},\Ssf^\pm)$ is a closed embedding (\cite[Lemma~2.1.2]{Kisin:IntModelAbType}). If the Shimura data and level are understood, we may abbreviate $\Sh_\Ksf(\Gsf,\Xsf)$ and $\Sh_{\Ksf'}({\sf GSp}_{2g},\Ssf^\pm)$ by $\Sh$ and $\Sh'$, respectively.

 By the standard argument (\emph{cf.} \cite[4.1.7]{KisinPappas:ParahoricIntModel}), the pull-back of the universal abelian variety $\Asf_\Gsf$ comes equipped with certain families of Hodge cycles, which we recall now. (See \cite[2.2]{Kisin:IntModelAbType} or \cite[4.1.7]{KisinPappas:ParahoricIntModel} for the full details.) Let $\Lambda_{\ZZ_{(p)}}\subset \QQ^{2g}$ denote the unimodular $\ZZ_{(p)}$-lattice that descends the unimodular $\ZZ_p$-lattice $\Lambda\subset \QQ_p^{2g}$. From the existence of the closed immersion (\ref{eq PrasadYu}), one can see that the Zariski closure of $\Gsf$ in $\GSp(\Lambda_{\ZZ_{(p)}})$ is a smooth $\ZZ_{(p)}$-model of $\Gscr_x$.
By \cite[Proposition~1.3.2]{Kisin:IntModelAbType} there exist finitely many tensors $s_\alpha\in\Lambda_{\ZZ_{(p)}}^\otimes$ whose pointwise stabiliser is the Zariski closure of $\Gsf$ in $\GL(\Lambda_{\ZZ_{(p)}})$.

Let $T_p(\Asf_\Gsf)$ denote the integral Tate module viewed as a $\Zp$-local system over $\Sh_{\Ksf}(\Gsf,\Xsf)$. Then we can associate to $(s_\alpha)$ the global sections
\[
t_{\alpha,\et} \in\Gamma (T_p(\Asf_\Gsf)^\otimes),
\]
such that for any geometric point $\bar{z}$ of $\Sh_{\Ksf}(\Gsf,\Xsf)$, there exists an isomorphism $\Lambda\cong T_p(\Asf_\Gsf)_{\tilde z}$ sending $s_\alpha$ to the fibre of $(t_{\alpha,\et})$ at $\bar z$. Indeed, over  $\Sh_{\Ksf}(\Gsf,\Xsf)_\CC$ we have an explicit complex analytic construction of such tensors $(t_{\alpha,\et})$  (or to be precise, we construct the tensors on the Betti (co)homology, and view them as tensors on \'etale (co)homology by scalar extension), and by the theory of absolutely Hodge cycles on abelian schemes we can descend the tensors over $\Sh_{\Ksf}(\Gsf,\Xsf)$. Similarly, for $l\ne p$, we can construct the tensors $t_{\alpha,l}\in\Gamma(T_l(\Asf_\Gsf)^\otimes)$.

For $\Ksf'=\Ksf^{\prime p}\Ksf'_p$ as above, denote by $\Sscr'_{\Ksf'}$ the moduli space of principally polarised abelian varieties over $\ZZ_{(p)}$ with $\Ksf'^p$-level structure,  and let $(\Ascr^\univ,\lambda^\univ,\eta^\univ)$ be its universal object. Then $\Sscr'_{\Ksf'}$ is a smooth scheme whose generic fibre is canonically isomorphic to $\Sh_{\Ksf'}(\mathsf{GSp}_{2g},\Ssf^\pm)$ (see e.g.\ \cite[\S~5]{Kottwitz:PtShimuraVarFinFields}). The Kisin-Pappas integral model $\Sscr_{\Ksf}$ of $\Sh_\Ksf(\Gsf,\Xsf)$ is given by the normalisation of the closure $\Sscr_{\Ksf}^-$ of $\Sh_{\Ksf}(\Gsf,\Xsf)$ in $\Sscr'_{\Ksf'} \otimes O_E$. We denote the corresponding morphism by $\iota: \Sscr_{\Ksf} \to \Sscr'_{\Ksf'}$ and by $\Ascr_\Gsf \coloneqq \iota^\ast\Ascr^\univ$. The canonical projection $\Sh(\Gsf,\Xsf)_{\Ksf_1^p \Ksf_p} \epi\Sh(\Gsf,\Xsf)_{\Ksf^p \Ksf_p}$ for $\Ksf_1^p \subset \Ksf^p$ and the $\Gsf(\AA_f^p)$-action on the tower $\{\Sh_{\Ksf^p \Ksf_p}(\Gsf,\Xsf)\}_{\Ksf^p\subset \Gsf(\AA_f^p)}$ both extend uniquely to the integral models. We will omit the level in the subscript (\emph{e.g.} $\Sscr=\Sscr_\Ksf$, $\Sscr'=\Sscr_{\Ksf'}$) if it is given by context or not relevant.

Let $z \in \Sscr(\FFbar_p)$, and choose a lift $\tilde z\in\Sscr(O_K)$ of $z$ for some finite extension $K/\breve{E}$. Then by crystalline-\'etale comparison theorem, the fibre of $t_{\alpha,\et}$ at some geometric point over $\tilde z$ gives rise to $t_{\alpha,z}\in M(\Ascr_{\Gsf,z}[p^\infty])_\QQ^\otimes$.  By \cite[Cor.~3.3.6]{KisinPappas:ParahoricIntModel}  together with standard results from the theory of Kisin modules (\emph{cf.} \cite[Theorem~3.3.2]{KisinPappas:ParahoricIntModel}), it follows that the tensors  $t_{\alpha,z}$ are integral and satisfy the condition (a) and (b) in \S~\ref{sect local geom}. (Indeed, using the notation of \cite[Theorem~3.3.2]{KisinPappas:ParahoricIntModel}, the tensors on the Kisin module $\Mfr(t_{\alpha,\et}) \in \Mfr(T_p(\Ascr_{\Gsf,\tilde z}))^\otimes$ lift $t_{\alpha,z}$, so we get the integrality of $t_{\alpha,z}$. Now, \cite[Cor.~3.3.6]{KisinPappas:ParahoricIntModel} implies the condition (a) in  \S~\ref{sect local geom}, and the condition (b) follows from the relationship between $D_{dR}(V_p(\Ascr_{\Gsf,\tilde z}))$ and $\Mfr(T_p(\Ascr_{\Gsf,\tilde z}))$, together with the fact that the Hodge filtration on the generic fibre of $\Ascr_{\Gsf,\tilde z}$ is given by some cocharacter in the geometric conjugacy class $\{\mu\}$.)    
Therefore, we can define a closed formal subscheme $\Def_{\Gscr_x}(\Ascr_{\Gsf,z}[p^\infty];(t_{\alpha,z})) \subset \Def(\Ascr_{\Gsf,z}[p^\infty])\widehat\otimes_{\Zpbr} O_{\breve E}$; \emph{cf.} Definition~\ref{def KP defor space}.
By the isomorphism (\ref{eq rigidity quasiisogeny}) and \cite[Lemma~3.2.13]{KisinPappas:ParahoricIntModel}, the tensors $t_{\alpha,z}$ do not depend on the choice of lift $\tilde z$ of $z$. Furthermore, the existence of \'etale tensors $t_{\alpha,\et}$ implies that
the deformation of $\Ascr_{\Gsf,z}$ given by a lift $\tilde z\in \Sscr(O_K)$ of $z$ gives rise to an $O_K$-point of $\Def_{\Gscr_x}(\Ascr_{\Gsf,z}[p^\infty];(t_{\alpha,z}))$ (\emph{cf.} Proposition~\ref{prop et univ prop}).

Kisin and Pappas gave the following description of the formal neighbourhood of  $z \in \Sscr(\FFbar_p)$.
 \begin{propsub}[{\cite[Prop.~4.2.2,\ Cor.~4.2.4]{KisinPappas:ParahoricIntModel}}] \label{prop normalisation on formal neighbourhoods}
  Let $z \in \Sscr(\k)$ and $z' = \iota(z)\in\Sscr'(\k)$. The identification $\Sscr_{\Zpbr,z'}^{\prime\wedge}\cong \Def_{\Gscr'_{x'}}(\Ascr^\univ_{z'}[p^\infty];\lambda^\univ_{z'})$, given via the Serre-Tate theorem, yields a canonical isomorphism  $\Sscr_{O_{\breve{E}},z}^\wedge\cong \Def_{\Gscr_x}(\Ascr_{\Gsf,z}[p^\infty]; (t_{\alpha,z}))$ such that the following diagram
   \begin{center}
    \begin{tikzcd}[column sep = small]
     \iota^{-1}(\Sscr_{O_{\breve{E}},z'}^{- \wedge}) \arrow{d} \arrow{r}{\sim} & \coprod\limits_{\iota(z) = z'} \Def_{\Gscr_x}(\Ascr_{\Gsf,z}[p^\infty]; (t_{\alpha,z})) \arrow{d} & \\
     \Sscr_{O_{\breve{E}},z'}^{- \wedge} \arrow{r}{\sim} \arrow[hook]{d} & \bigcup\limits_{\iota(z) = z'} \Def_{\Gscr_x}(\Ascr_{z'}^\univ [p^\infty]; (t_{\alpha,z})) \arrow[hook]{d} \\
     \Sscr'^\wedge_{O_{\breve{E}},z'} \arrow{r}{\sim} & \Def_{\Gscr'_{x'}}(\Ascr_{z'}^\univ[p^\infty];\lambda_{z'}^\univ)\widehat{\otimes}_{\Zpbr}O_{\breve{E}}
    \end{tikzcd}
   \end{center}
   commutes, where
    the vertical maps are induced by the natural inclusions
   \[\Def_{\Gscr_x}(\Ascr_{\Gsf,z}[p^\infty];(t_{\alpha,z})) \mono \Def_{\Gscr'_{x'}}(\Ascr^\univ_{z'}[p^\infty];\lambda^\univ_{z'}).\]
 \end{propsub}

Since $\Def_{\Gscr_x}(\Ascr_{\Gsf,z}[p^\infty]; (t_{\alpha,z})$ is represented by the formal neighbourhood of the local model $\Mscr^{\rm loc}_{G,\{\mu\},x}$, we obtain the following corollary from the local property of $\Mscr^{\rm loc}_{G,\{\mu\},x}$:
\begin{corsub}[\emph{cf.} {\cite[Corollary~0.3]{KisinPappas:ParahoricIntModel}}]\label{cor normalisation on formal neighbourhoods}
The integral model $\Sscr$ is normal. Its geometric special fibre $\Ssb$ is reduced, and each irreducible component is normal and Cohen-Macaulay.
\end{corsub}

  \subsection{Stratifications of the special fibre}

 Let $\Ssb$ denote the geometric special fibre of the integral model $\Sscr$. For any closed point $z \in \Ssb(\FFbar_p)$, choosing an isomorphism $\DD(\Ascr_{\Gsf,z}[p^\infty])_\W \cong \Lambda \otimes_\O \W$ respecting the additional structure carries the Frobenius $F$ to an element $g_z\sigma$ with $g_z \in G(\L)$. Due to the choice of an isomorphism only the $\Gscr_x(\W)$-$\sigma$-conjugacy class is well-defined, which we denote by $\pot{g_z}$. Similarly, the isogeny class of $(\Ascr_{\Gsf,z}[p^\infty],\lambda_{\Gsf,z}, (t_{\alpha,z}))$ is determined by the $G(\L)$-$\sigma$-conjugacy class $[g_z]$. We denote for $b \in G(\L)$
 \begin{align*}
   \Cbar^{b}(\k) &\coloneqq \{ z \in \Ssb(\k) \mid g_z \in \pot{b}\} \\
   \Ssb^b(\k) &\coloneqq \{ z \in \Ssb(\k) \mid g_z \in [b]\}.
 \end{align*}
 We want that these sets are locally closed, so that we may equip them with reduced structure and view them as subvarieties.  The work of Oort, and Rapoport and Richartz show that these sets are indeed locally closed, provided that the pointwise tensors on the cohomology are induced by some sort of global structure. This is proven in the next section.

 \section{Construction of global tensors on the display of $\Ascr_\Gsf[p^\infty]$}\label{sect tensors}
 
 Since $\Sscr$ is normal, it follows that for $l\ne p$ the $l$-adic \'etale components of the Hodge cycles $\{t_{\alpha,l}\}$ on the generic fibre $\Asf_\Gsf$ extend to the integral model $\Ascr_\Gsf$. On the other hand, this property does not hold for the $p$-adic \'etale tensors $\{t_{\alpha,\et}\}$. When $x\in\Bcal(G,\QQ_p)$ is a hyperspecial vertex, then it is possible to obtain the crystalline Tate tensor on the integral model (as the analogue of the good reduction of $l$-adic \'etale component), which crucially uses the smoothness of $\Sscr$. In general (especially, when $\Sscr$ is not smooth), we have an analogue of this in terms of displays, which we will obtain in Proposition~\ref{prop global tensors}.

\subsection{Review of $p$-adic comparison morphism for Barsotti-Tate groups}
Let $C$ be an algebraically closed complete extension of $\QQ_p$ with ring of integers $O_C$. In practice, $C = \widehat{\overline{K}}$, where $K$ is a   complete discrete valuation field of mixed characteristic $(0,p)$. Let $\Xscr$ be a {\BT} over $O_C$.
We briefly recall the $p$-adic comparison morphism in this case.

Let us briefly recall Fontaine's construction of crystalline period rings.
We define
\[
O_C^\flat \coloneqq \varprojlim_\Phi O_C/p \thra O_C/p
\]
to be the inverse perfection of $O_C/p$. 
And one has the following lift of the natural projection
\[
 \theta\colon W(O_C^\flat) \to O_C
\]
which is surjective and its kernel is generated by a regular element.

Let $\Acris(O_C)\thra O_C$  be the $p$-adically completed divided power hull of $\theta$ over $\ZZ_p$ (with respect to the usual divided power structure on $p\ZZ_p$). It is known that $\Acris(O_C)$ is $p$-torsion free. 

One can check that the Witt vector Frobenius $\sigma$ of $W(O_C^\flat)$ naturally extends to $\Acris(O_C)$, which we also denote by $\sigma$.  If $C = \widehat{\overline{K}}$ then the natural $\Gamma_K$-action on $W(O_C^\flat)$ extends to $\Acris(O_C)$, and the natural projection $\Acris(O_C)\to O_C/p$ and $\sigma:\Acris(O_C)\to\Acris(O_C)$ are $\Gamma_K$-equivariant.

Recall that $\Acris(O_C)\to O_C$ is initial among $p$-adic divided power thickenings of $O_C$ over $\ZZ_p$ (with the natural divided power structure on $p\ZZ_p$); \emph{cf.} \cite[\S\S~2.2]{fontaine:Asterisque223ExpII}. As an application, note that $W(O_C)$ is a $p$-adic divided power thickening of $O_C$ over $\ZZ_p$. Therefore, by the universal property, we have a natural divided power morphism 
\begin{equation}\label{eq Acris to Witt}
\Acris(O_C)\epi W(O_C).
\end{equation}
This morphism commutes with the natural Frobenius endomorphisms on the both sides, and if $C = \widehat{\overline{K}}$ then it is $\Gamma_K$-equivariant.

We choose a compatible sequence of primitive $p^n$th root of unity $\epsilon\in O_C^\flat$, and write $t\coloneqq\log ([\epsilon])\in\Acris(O_C)$, which is a regular element. 
We set  $\Bcris(O_C)\coloneqq\Acris(O_C)[\frac{1}{pt}] = \Acris(O_C)[\frac{1}{t}]$.

Let us  review the crystalline Dieudonn\'e theory over $O_C$. By \cite[Theorem~A]{ScholzeWeinstein:RZ} the covariant Dieudonn\'e module functor for Barsotti-Tate groups over $O_C/p$
\[
\Xscr \rightsquigarrow (\DD(\Xscr)_{\Acris(O_C)}^\vee,  \Phi),
\]
is fully faithful. Note that there is no ambiguity in the notation as the covariant Dieudonne crystal of $\Xscr$ can be canonically identified with the dual of $\DD(\Xscr)$; however the Frobenius satisfies $\Phi = p \cdot \varphi$ where $\varphi$ is the contragradient of the contravariant crystalline Frobenius (after inverting $p$).%

From \cite[Theorem~A]{ScholzeWeinstein:RZ} and the Grothendieck-Messing deformation theory (using $p>2$), we obtain the following \emph{fully faithful} functor
\begin{equation}\label{eq Lau functor}
\Xscr \rightsquigarrow (\DD(\Xscr)_{\Acris(O_C)}^\vee, \Fil^0 (\DD(\Xscr)_{\Acris(O_C)}^\vee), \Phi),
\end{equation}
associating to a Barsotti-Tate group $\Xscr$ over $O_C$ the the covariant Dieudonn\'e module $(\DD(\Xscr)_{\Acris(O_C)}^\vee,  \Phi)$ of $\Xscr_{O_C/p}$ together with the filtration
\[\Fil^0 (\DD(\Xscr)_{\Acris(O_C)}^\vee)\coloneqq \ker( \DD(\Xscr)_{\Acris(O_C)}^\vee \twoheadrightarrow \Lie(\Xscr)).\]
Note that this functor sends $\QQ_p/\ZZ_p$ to $(\Acris(O_C), \Acris(O_C), p\sigma)$. With further work, one can show that $\Phi_1 = \frac{\Phi}{p}:\Fil^0(\DD(\Xscr)_{\Acris(O_C)}^\vee)\to (\DD(\Xscr)_{\Acris(O_C)}^\vee)$ is well defined and its image generates the target, hence defining a ``crystalline Dieudonn\'e window'' in the sense of \cite[Theorem~6.3]{Lau:DieudonneSemiperfect}; \emph{cf.} \cite[Lemma~4]{Kim:ClassifFSmErratum}.

Let $\Xscr$ be a Barsotti-Tate group over $O_C$, and let $T(\Xscr_C)$ denote the integral Tate module. Recall that we have natural  isomorphisms
\[
T(\Xscr_C) \cong  \Hom_{O_C} (\QQ_p/\ZZ_p,\Xscr).
\]
Now, applying the fully faithful functor (\ref{eq Lau functor}), we obtain a natural isomorphism
\[
 T(\Xscr_C) \cong (\Fil^0\DD(\Xscr)_{\Acris(\tilde S)}^\vee)^{\Phi_1=1};
\]
In particular, we obtain a natural comparison morphism
\begin{equation}\label{eq integral Ccris}
\beta_\Xscr:T(\Xscr_C)\otimes_{\Zp}\Acris(O_C) \to  \DD(\Xscr)_{\Acris(O_C)}^\vee,
\end{equation}
which is compatible with the filtrations (where the filtration on the left hand side is defined by setting $\Fil^1\Acris(O_C)$ to be the kernel of the natural projection onto $O_C$). Furthermore, the natural Frobenius endomorphism $1\otimes\sigma$ on the source corresponds to $\Phi_1$ on the target. If  $C = \widehat{\overline{K}}$ and $\Xscr$ is defined over $O_K$, then the morphism (\ref{eq integral Ccris}) is $\Gamma_K$-equivariant in the natural way.

Let us recall the following result of Faltings \cite[\S~6, Theorem~7]{Faltings:IntegralCrysCohoVeryRamBase}.\footnote{Although Faltings \cite{Faltings:IntegralCrysCohoVeryRamBase} works with the divided power completion of $\Acris(O_C)$ in place of $\Acris(O_C)$, the essentially same proof work over $\Acris(O_C)$.}
\begin{propsub}\label{prop p-adic comparision}
The comparison morphism $\beta_\Xscr$ (\ref{eq integral Ccris}) induces a natural isomorphism
\[
T(\Xscr_C)\otimes_{\Zp}\Bcris(O_C) \cong \DD(\Xscr)_{\Acris(O_C)}^\vee[\tfrac{1}{t}]
\]
by inverting $t$.
\end{propsub}

\subsection{Integrality of tensors: the case of complete discrete valuation ring}
Let $K$ be a complete discretely valued extension of $\QQ_p$ with ring of integers $O_K$, and let $\Xscr$ be a {\BT} over $O_K$. Let $C\coloneqq \widehat{\overline{K}}$. In this section, we will investigate the relationship between $\Gamma_K$-invariant tensors $t_{\et}\in T(\Xscr_C)^\otimes$ and Frobenius-invariant tensors $t\in(\DD(\Xscr)_{W(O_K)}^\vee)^\otimes$. Our starting point is the $p$-adic comparison theorem (Proposition~\ref{prop p-adic comparision}).

By the previous section, any tensor $t_{\et}\in T(\Xscr_C)^\otimes$ induces a tensor $t_{\et} \otimes 1$ of $\DD(\Xscr)^\vee_{\Acris(O_C)} [\frac{1}{t}]$. 
We would like to show that (under some extra assumption)  $t_{\et} \otimes 1$ lies in $\DD(\Xscr)^\otimes_{\Acris(O_C)}$, without inverting $t$. This claim does not follow from the integrality of $t_{\et}$ and $\beta_\Xscr(T(\Xscr_C)) \subset \DD(\Xscr)_{\Acris(O_C)}^\vee$ since we have the opposite inclusion for their duals. %

\begin{propsub}\label{prop integrality Tate tensors}
Let $O_K$ be a complete discrete valuation ring of mixed characteristic~$(0,p)$ with fraction field $K$, and write $C\coloneqq \widehat{\overline{K}}$. Let $\Xscr$ be a {\BT} over $O_K$.
Then for any $\Gamma_K$-invariant tensor $t_{\et}\in T(\Xscr_C)^\otimes$, the tensor $t_{\et}\otimes 1\in \DD(\Xscr)_{\Acris(O_C)}^\otimes[\tfrac{1}{t}]$, constructed above, lies in $\DD(\Xscr)_{\Acris(O_C)}^\otimes$.
\end{propsub}
\begin{proof}
To prove the proposition, we may replace $K$ with some discretely valued subextension of $K$ in $C$. Therefore, it is sufficient to handle the case when the residue field $\kappa$ of $O_K$ is perfect.

Let us choose a uniformiser $\varpi\in O_{K}$ and its $p^n$th root $\varpi_n\in O_C$ in a compatible manner (with $\varpi_0=\varpi$), and set $\varpi^\flat\coloneqq(\varpi_n)_{n\geqslant0}\in O_C^\flat$. Then we have a map $W(\kappa)\pot{u}\to W(O_C^\flat)$ by sending $u$ to $[\varpi^\flat]$, so we also obtain a map $W(\kappa)\pot{u}\to\Acris(O_C)$. Note that this map is Frobenius-equivariant where the Frobenius acts on $W(\kappa)\pot{u}$ by $\sum [a_i] u^i \mapsto \sum [a_i^p] u^{i\cdot p}$%

By the theory of Kisin modules over $O_K$, we have a covariant fully faithful $\otimes$-functor $\Mfr(\bullet)$ from Galois-stable lattices in some crystalline representation of $K$ to Kisin modules (\emph{cf.} \cite[\S\S~2.1]{kisin:fcrys} or \cite[Theorem~3.3.2]{KisinPappas:ParahoricIntModel}).
Therefore, we obtain a finite free $W(\kappa)\pot{u}$-module $\Mfr(T(\Xscr_C))$ and $\tilde t\coloneqq \Mfr(t_\et)\in \Mfr(T(\Xscr_C))^\otimes$. 

On the other hand, we have $\Mfr(T(\Xscr_C))^\vee\otimes_{W(\kappa)\pot u,\sigma}S\cong \DD(\Xscr)_S$ where $S$ is the $p$-adically completed divided power hull of $W(\kappa)\pot u\epi O_K$ with $u\mapsto \varpi$; \emph{cf.} \cite[\S\S~2.2]{kisin:fcrys}.
 (See also \cite[Theorem~1.4.2]{Kisin:IntModelAbType}.)
Since the map $W(\kappa)\pot u\to\Acris(O_C)$ extends to $S\to\Acris(O_C)$, we obtain the following natural isomorphism
\begin{equation}\label{eq Kisin to Breuil}
\Mfr(T(\Xscr_C))\otimes_{W(\kappa)\pot{u},\sigma}\Acris(O_C) = \DD(\Xscr)_{\Acris(O_C)}^\vee,
\end{equation} 
respecting the Frobenius operators and filtrations (where the filtration on $\Mfr(T(\Xscr_C))$ is defined as in \cite[\S\S~1.2]{Kisin:IntModelAbType}).

It remains to show that the isomorphism (\ref{eq Kisin to Breuil}) matches $\tilde t\otimes1$ and $t_{\et}\otimes 1$. We can verify this claim after inverting $t$. 
Now, recall that for any crystalline $\Gamma_K$-representation $V$ and a $\Gamma_K$-stable $\ZZ_p$-lattice $T\subset V$ we have a natural isomorphism %

\[
\Fil^0\left( \Mfr(T)^\otimes\otimes_{W(\kappa)\pot u,\sigma}B_\cris(O_C) \right)^{\varphi = 1} \cong \Fil^0\left( D_\cris(V)^\otimes\otimes_{W(\kappa)}B_\cris(O_C) \right)^{\varphi = 1},
\]
where the filtration on $\Mfr(T)^\sigma$ was defined as in   \cite[3.3.1]{KisinPappas:ParahoricIntModel}. (This isomorphism can be read off from \cite[proof of Proposition~2.1.5]{kisin:fcrys}; indeed, the proof in \emph{loc.~cit.} implies that each of the injective morphisms \cite[(2.1.6), (2.1.7), (2.1.8)]{kisin:fcrys} is an isomorphism, which implies the desired isomorphism above.)  In particular, the above isomorphism should match $\tilde t\otimes 1$ and $D_\cris(t_\et)$, which are fixed by the Frobenius operator and lie in the $0$th filtration. Therefore, the isomorphism (\ref{eq Kisin to Breuil}) matches $\tilde t\otimes1$ and $t_{\et}\otimes 1$.
\end{proof}

Recall that we have a  divided power morphism $\Acris(O_C)\epi W(O_C)$ commuting with the natural Frobenius endomorphisms and $\Gamma_K$-action; \emph{cf.} (\ref{eq Acris to Witt}).
So we have a natural $\Gamma_K$-equivariant isomorphism 
\[
\DD(\Xscr)^\vee_{\Acris(O_C)}\otimes_{\Acris(O_C)}W(O_C) \cong \DD(\Xscr)^\vee_{W(O_C)},
\]
commuting with the natural Frobenius operators.
Therefore, we may view $t_\et\otimes1$ as a  tensor of $ \DD(\Xscr)^\vee_{W(O_C)}$, invariant under the Frobenius- and $\Gamma_K$- actions. 

\begin{corsub}\label{cor integrality Tate tensors}
The tensor $t_\et\otimes1\in (\DD(\Xscr)^\vee_{W(O_C)})^\otimes$, defined above, lies in  $(\DD(\Xscr)^\vee_{W(O_K)})^\otimes$.
\end{corsub}
\begin{proof}
Note that $\DD(\Xscr)^\otimes_{W(O_C)}\cong \DD(\Xscr)^\otimes_{W(O_K)}\otimes_{W(O_K)}W(O_C)$, so to prove the corollary it suffices to show that $O_C^{\Gamma_K} = O_K$. If the residue field of $O_K$ admits a finite $p$-basis, this assertion follows from  \cite[Theorem~1]{Hyodo:ImperfectHT}.
\end{proof}

So far, we have associated, to a $\Gamma_K$-invariant tensor $t_{\et} \in T(\Xscr)^\otimes$, a Frobenius-invariant tensor $t_\et\otimes 1\in (\DD(\Xscr)^\vee_{W(O_K)})^\otimes$. We finish this section by discussing the way to reverse engineering.

We continue to assume that $O_K$ is a complete discrete valuation ring of mixed characteristic~$(0,p)$, and we additionally assume that its residue field $\kappa$ is perfect. Then for any {\BT} $\Xscr$ over $O_K$, we have a Frobenius-equivariant isomorphism
\[
\DD(\Xscr)_{W(O_K)} \cong \DD(\Xscr_{O_K/p})_{W(O_K)},
\]
where the right hand side makes sense since $W(O_K)$ is a $p$-adic divided power thickening of $O_K/p$.

\begin{subequations}
Recall that $\Xscr_{O_K/p}$ is isogenous to $\Xscr_{O_K/p}^{\sigma^n}$ for any $n$ by $n$th iterated Frobenius, and for $n\gg1$ we have $\Xscr_{O_K/p}^{\sigma^n}\cong (\Xscr_\kappa^{\sigma^n})_{O_K/p}$, which is in turn isogenous to $(\Xscr_\kappa)_{O_K/p}$ by the $n$th iterated Frobenius. Therefore, choosing $n\gg1$ we obtain a Frobenius-equivariant isomorphism
\begin{equation}\label{eqn:DworkTrick}
\DD(\Xscr_\kappa)^\vee_{W(\kappa)}\otimes_{W(\kappa)}W(O_K)[\tfrac{1}{p}] \cong \DD(\Xscr)^\vee_{W(O_K)}[\tfrac{1}{p}],
\end{equation}
lifting the identity map on $\DD(\Xscr_\kappa)^\vee_{W(\kappa)}[\frac{1}{p}]$. One can  show without difficulty that there is at most one Frobenius-equivariant lift (\ref{eqn:DworkTrick}) of the the identity map on $\DD(\Xscr_\kappa)^\vee_{W(\kappa)}[\frac{1}{p}]$, by the same proof as \cite[Lemma~3.1.17]{KisinPappas:ParahoricIntModel}. Similarly, by scalar extension with respect to $W(O_K)_\QQ\twoheadrightarrow K$, we obtain
\begin{equation}
\DD(\Xscr_\kappa)^\vee_{W(\kappa)}\otimes_{W(\kappa)}K  \cong\DD(\Xscr)_{O_K}^\vee\otimes_{O_K}K,
\end{equation}
which coincides with the isomorphism defining the admissible filtered isocrystal structure on $\DD(\Xscr_\kappa)^\vee_{W(\kappa)}[\frac{1}{p}]$.
\end{subequations}

Then the following statements are immediate from the above discussion;
\begin{lemsub}\label{lem:Ccris}
Let $t$ be a Frobenius-invariant tensor of $\DD(\Xscr)^\vee_{W(O_K)}[\frac{1}{p}]$.
Then $t$ is uniquely determined by its image  $\bar t\in (\DD(\Xscr_\kappa)^\vee_{W(\kappa)}[\frac{1}{p}])^\otimes$.
\end{lemsub}
\begin{proof}
To show that $t$ is uniquely determined by $\bar t$, it suffices to show that the  isomorphism (\ref{eqn:DworkTrick}) sends $\bar t\otimes 1$ to $t$, which can be seen by repeating the proof of \cite[Lemma~3.2.13]{KisinPappas:ParahoricIntModel}. The rest of the claim is clear.
\end{proof}

\begin{corsub}\label{cor:Ccris}
We continue to use the setting of Lemma~\ref{lem:Ccris}, and assume in addition that the image of $t$ in $(\DD(\Xscr)_{O_K}^\vee\otimes_{O_K}K)^\otimes$ lies in the $0$th filtration with respect to the Hodge filtration. We consider the $\Gamma_K$-invariant tensor $t_\et\coloneqq  V_\cris(t)\in T(\Xscr_{\overline K})_\QQ^\otimes$ (with the usual covariant definintion of $V_\cris$ as in  \cite[5.5.1]{fontaine:Asterisque223ExpIII}).

Then, $t$ is uniquely determined by $t_\et$. If furthermore $t_\et$ is integral (i.e., if $t_\et\in T(\Xscr_{\overline K})^\otimes$), then $t$ coincides with the tensor obtained from $t_{\et}\otimes 1$ in Corollary~\ref{cor integrality Tate tensors}.
\end{corsub}
\begin{proof}
We use the notation of Lemma~\ref{lem:Ccris}, and assume that the image of $t$ lies in the $0$th filtration with respect to the Hodge filtration. Then $t_\et$ determines $\bar t$, which in turn determines $t$ by Lemma~\ref{lem:Ccris}. 

Assume furthermore that $t_\et$ is integral. To show the last claim, it suffices to compare  $t$ and $t_\et\otimes 1$ after inverting $p$. Then by  Lemma~\ref{lem:Ccris}, it suffices to compare their images in $(\DD(\Xscr_\kappa)^\vee_{W(\kappa)}[\frac{1}{p}])^\otimes)$. On the other hand, both tensors reduces to $\bar t$ by construction.
\end{proof}

\subsection{Descent to display tensors over $\Sscr$}
For any {\BT} $\Xscr$  over a $p$-adic ring $S$, we set
\[
P(\Xscr)\coloneqq \DD(\Xscr)_{W(S)}^\vee
\]
for the simplicity of notation.

Let $\Sfr$ be a $p$-adic formal scheme. Then by the fpqc descent theory for Witt rings \cite[\S1.3, Lemma~30]{Zink:DisplayFormalGpAsterisq278}, we obtain a sheaf of rings $W(\Ocal_\Sfr)$ over $\Sfr$ such that for any open affine formal subscheme $\Spf S\subset \Sfr$ we have $\Gamma(\Spf S,W(\Ocal_\Sfr)) = W(S)$. Furthermore, also from \emph{loc.~cit.} it follows that given any {\BT} $\Xscr$ over $\Sfr$ we obtain a locally free $W(\Oscr_\Sfr)$-module $P(\Xscr)$ such that for any open affine formal subscheme $\Spf S\subset \Sfr$ we have
\[
\Gamma(\Spf S, P(\Xscr)) = P(\Xscr_S) = \DD(\Xscr_S)_{W(S)}^\vee.
\]
When $\Sfr$ is affine, Lau \cite[Proposition~2.1]{Lau:Smoothness} gave a natural structure of display to $P(\Xscr)$ compatible with the crystalline Frobenius and Verschiebung operator. When $\Sfr$ is any $p$-adic formal scheme, the display structure on $P(\Xscr)$ over any affine open formal subschemes patches together by fpqc descent theory for displays \cite[\S1.3, Theorem~32]{Zink:DisplayFormalGpAsterisq278}; in other words, we obtain a $W(\Ocal_\Sfr)$-submodule $P_1(\Xscr)$ of $P(\Xscr)$, and $\sigma$-linear morphisms $\Phi\colon P(\Xscr)\to P(\Xscr)$ and  $\Phi_1\colon P_1(\Xscr)\to P(\Xscr)$ by glueing Zariski-locally defined display structure over affine open formal subschemes. And just as in the case of Dieudonn\'e displays,
if $\Sfr$ is topologically flat over $\ZZ_p$ then $\Phi$ and $\Phi_1$ are determined by the $W(\Ocal_\Sfr)$-linear isomorphism $\Psi\colon \widetilde P_1(\Xscr) \isom P(\Xscr)$, where $ \widetilde P_1(\Xscr)$ is the image of $P_1(\Xscr)^\sigma$ in $P(\Xscr)^\sigma$ and $\Psi$ is induced by the linearisation of $\Phi_1$.

If $S$ is a complete local noetherian ring with perfect residue field and $\Xscr$ is a {\BT} over $S$, then we have a Dieudonn\'e display $\M(\Xscr)$ (\S\S~\ref{ssect displays}). Then by comparison with the Dieudonn\'e crystal, it follows that 
\[P(\Xscr) \cong M(\Xscr)\otimes _{\WW(S)}W(S).\] Therefore, when $S=R_G$ is the Kisin-Pappas deformation ring (Definition~\ref{def KP defor space}), the tensors $t_{\alpha}^{\rm def}\in M(\Xscr^{\rm def})^\otimes$ constructed by Kisin-Pappas (\emph{cf.} the paragraph below Definition~\ref{def KP defor space}) can be viewed as elements in $P(\Xscr^{\rm def})^\otimes$.

The main result of this section is the following proposition:
  \begin{propsub} \label{prop global tensors}
 We continue to assume that $p>2$. Let $\wh\Sscr$ denote the $p$-adic completion of $\Sscr\otimes_{O_E}O_{\breve{E}}$, and let $\Xscr_\Gsf$ as the universal $p$-divisible group over $\widehat\Sscr$.  
Then  there exists a family of $\Psi$-invariant tensors $(t_{\alpha})$ in $P(\Xscr_{\Gsf})^\otimes$ such that for every $z \in \Sscr(\k)$ the restriction of $t_{\alpha}$ to the formal neighbourhood of $z$ coincides with $t_{\alpha,z}^{\rm def}$.
 \end{propsub}

The proof   uses the absolutely Hodge cycles on the universal abelian varieties, which we briefly review now.
We view the Tate module $T(\Xscr_\Gsf)$ of $\Xscr_\Gsf$ as a $\Zp$-local system over $\Sh_\Ksf(\Gsf,\Xsf)$, so that its fibre at any geometric generic point $\bar\eta$ of $\Sh_\Ksf(\Gsf,\Xsf)$ is $T(\Xscr_{\Gsf,\bar\eta})$. Then from the \'etale components of absolute Hodge cycles $(t_{\alpha,\et})$ on the universal abelian schemes, we obtain the following maps of $\Zp$-local systems
\[
t_{\alpha,\et}:\mathbf{1} \to T(\Xscr_\Gsf)^\otimes.
\]
or equivalently, $\pi_1(\Sh_\Ksf(\Gsf,\Xsf),\bar\eta)$-invariant sections $(t_{\alpha,\et})$ of $T(\Xscr_{\Gsf,\bar\eta})^\otimes$.

Furthermore, as a byproduct of the construction of Kisin-Pappas integral models $\Sscr$, the de~Rham components of absolute Hodge cycles extend integrally and give the following tensors
\[
t_{\alpha,\dR} \in (P(\Xscr_\Gsf) \otimes_{W(\Ocal_{\widehat\Sscr})}\Ocal_{\widehat\Sscr})^\otimes.
\]  
Furthermore, the tensors $(t_{\alpha,\dR})$ globalise the ``de~Rham components'' of $(t^{\rm def}_{\alpha,z})$ constructed at the formal neighbourhood at any $z\in\Sscr(\overline{\FF}_p)$. (This statement can be read off from the statements and proof of  \cite[Propositions~4.2.2,~4.2.6]{KisinPappas:ParahoricIntModel}.)

We now show that $(t_{\alpha,\et})$ and $(t_{\alpha,\dR})$ are associated via the rigid analytic de~Rham comparison isomorphism in the sense of \cite[Theorem~8.8]{Scholze:CdR}. To explain, let $\pi:\Ascr_{\Gsf}\to\Sscr$ denote the structure sheaf. We set $\Scal\coloneqq \widehat\Sscr^\ad_{\breve E}$ and pull back the abelian scheme to obtain $\pi:\Acal\to\Scal$. We use the basic notation from \cite{Scholze:CdR} without reviewing it here.

By \cite[Theorem~8.8]{Scholze:CdR}, we have the following horizontal isomorphism of pro\'etale  $\Ocal\BB_{\dR,\Scal}$-modules:
\[
(\Rrm^1\pi_{\et,\ast}\ZZ_p)^\wedge \otimes_{\hat\ZZ_p}\Ocal\BB_{\dR,\Scal} \cong (\Rrm^1\pi_{\dR,\ast}\Ocal_{\Acal}) \otimes_{\Ocal_\Scal}\Ocal\BB_{\dR,\Scal}.
\]
(If $\Lscr$ is an \'etale $\Zp$-local system, $(\Lscr)^\wedge$ denote the associated pro\'etale $\hat\ZZ_p$-local system.)
Furthermore, the horizontal parts (i.e., sections annihilated by $\nabla$) we obtain an isomorphism of $\BB_{\dR,\Scal}$-local systems:
\[
(\Rrm^1\pi_{\et,\ast}\ZZ_p)^\wedge \otimes_{\hat\ZZ_p}\BB_{\dR,\Scal} \cong \left( (\Rrm^1\pi_{\dR,\ast}\Ocal_{\Acal}) \otimes_{\Ocal_\Scal}\Ocal\BB_{\dR,\Scal}\right)^{\nabla = 0}.
\]

 Both $(t_{\alpha,\et})$ and $(t_{\alpha,\dR})$ induce tensors on this $\BB_{\dR}$-local system. Indeed, the \'etale tensors $(t_{\alpha,\et})$ can be pulled back to the tensors on the $\BB_{\dR}$-local system, and the \emph{horizontal} tensors $(t_{\alpha,\dR})$ of $\Rrm^1\pi_{\dR,\ast}\Ocal_{\Acal} $ give rise to horizontal tensors of $(\Rrm^1\pi_{\dR,\ast}\Ocal_{\Acal}) \otimes_{\Ocal_\Scal}\Ocal\BB_{\dR}$. 

\begin{propsub}\label{prop:CdR}
The tensors of $\BB_{\dR,\Scal}$-local systems defined by $(t_{\alpha,\et})$ and $(t_{\alpha,\dR})$ are equal.
\end{propsub}
\begin{proof}
Let $\Scal_{\Gamma(p^\infty)}$ denote the preimage of $\Scal$ in the perfectoid Shimura variety associated to $(\Gsf,\Xsf)$; \emph{cf.} \cite[Theorem~IV.1.1]{Scholze:TorsionCycles}. Since the $\BB_{\dR,\Scal}$-tensors are determined pro\'etale-locally and $\Scal_{\Gamma(p^\infty)}$ is a pro\'etale covering of $\Scal$, it suffices to compare the tensors over $\Scal_{\Gamma(p^\infty)}$.

By construction the \'etale $\Zp$-local system $\Rrm^1\pi_{\et,\ast}\ZZ_p$ pulls back to a constant local system over $\Scal_{\Gamma(p^\infty)}$, so the associated $\BB_{\dR,\Scal_{\Gamma(p^\infty)}}$-local system is also constant. Therefore, the proposition can be verified at the fibre at a point in each connected component. On the other hand,  the $\BB_{\dR,\Scal_{\Gamma(p^\infty)}}$-tensors $(t_{\alpha,\et})$ and $(t_{\alpha,\dR})$ have the same fibre at the preimage in $\Scal_{\Gamma(p^\infty)}$ of any classical point of $\Scal$, by the theorem of Blasius and Wintenberger; \emph{cf.} \cite[proof of Proposition~4.2.2]{KisinPappas:ParahoricIntModel}.
\end{proof}

Let $\Spf S\subset\widehat\Sscr$ be an open affine connected formal subscheme.
Since $\Sscr$ is normal, it follows that  $S$ is a normal domain. We choose a minimal open prime $\pfr\subset S$. Then the completion $\widehat S_\pfr$ is a $p$-adic discrete valuation ring. 

Let $C$ be the completion of the algebraic closure of $\Frac(\widehat S_\pfr)$. Then, by Proposition~\ref{prop p-adic comparision} we have a natural isomorphism \[T(\Xscr_{\Gsf,C})\otimes_{\ZZ_p}\Bcris(O_C)\cong \DD(\Xscr_{\Gsf,O_C})_{\Acris(O_C)}[\tfrac{1}{t}].\] 
Therefore we have the following natural isomorphism:
\begin{equation}\label{eqn:CdR}
\begin{aligned}
T(\Xscr_{\Gsf,C})\otimes_{\ZZ_p} B_\dR(O_C)
&\cong \DD(\Xscr_{\Gsf,O_C})^\vee_{\Acris(O_C)}\otimes_{\Acris(O_C)}B_\dR(O_C) \\
&\cong \DD(\Xscr_{\Gsf,\kappa_\pfr})^\vee_{W(\kappa_\pfr)}\otimes_{W(\kappa_\pfr)}B_\dR(O_C) \\
& \cong \DD(\Xscr_{\Gsf,\widehat S_\pfr})^\vee_{\widehat S_\pfr}\otimes_{\widehat S_\pfr}B_\dR(O_C),
\end{aligned}
\end{equation}
where $\kappa_\pfr$ is the residue field of $\widehat S_\pfr$. (To see the last two isomorphisms, we choose a complete discretely valued subfield $K\subset C$ containing $\widehat S_\pfr$, with perfect residue field $\kappa$. Then for a {\BT} $\Xscr$ over $O_K$, we have $\DD(\Xscr)_{\Acris(O_C)}[\frac{1}{p}]\cong \DD(\Xscr_\kappa)_{W(\kappa)}\otimes_{W(\kappa)} \Bcris^+(O_C)$ and $\DD(\Xscr_\kappa)_{W(\kappa)}\otimes K \cong \DD(\Xscr)_{O_K}\otimes_{O_K}K$. We apply these isomorphisms to $\Xscr = \Xscr_{\Gsf, O_K}$.)

\begin{corsub}\label{cor:CdR}
The isomorphism (\ref{eqn:CdR}) matches $(t_{\alpha,\et}\otimes1)$ and $(t_{\alpha,\dR}\otimes1)$.
\end{corsub}
\begin{proof}
Let $v_{\pfr}$ denote a $p$-adic norm on $\widehat S_\pfr[\frac{1}{p}]$. Then 
$v_{\pfr}$ restricts to a continuous rank-$1$ valuation of $S[\frac{1}{p}]$, so it defines a rank-$1$ point $\xi\in\Spa(S[\frac{1}{p}],S) \subset \Scal$. By construction, the residue field at $\xi$ is $(\widehat S_\pfr[\frac{1}{p}],\widehat S_\pfr)$.

We choose a geometric point $\bar\xi$ of $\Scal$ over $\xi$, valued in $(C,O_C)$.  Then the isomorphism (\ref{eqn:CdR}) can be obtained as the fibre at $\bar\xi$ of the de~Rham comparison isomorphism $
(\Rrm^1\pi_{\et,\ast}\hat\ZZ_p)^\wedge \otimes_{\hat\ZZ_p}\BB_{\dR,\Scal} \cong ((\Rrm^1\pi_{\dR,\ast}\Ocal_{\Acal}) \otimes_{\Ocal_\Scal}\Ocal\BB_{\dR,\Scal})^{\nabla=0}$. Therefore, the corollary now follows from Proposition~\ref{prop:CdR}.
\end{proof}

Let $\pfr\subset S$ be a minimal open prime as above. Then the residue field $\kappa_\pfr\coloneqq \Frac(S/\pfr)$ admits a finite $p$-basis given by local coordinates of the smooth locus in $\Spec S/pS$ containing $\pfr$. Therefore, by Corollary~\ref{cor integrality Tate tensors} we obtain tensors $(t_{\alpha,\et}\otimes 1)$ of $P(\Xscr_{\Gsf,\widehat S_\pfr})$. 

We want to compare these tensors with the ``display tensors'' $t^{\rm def}_{\alpha, z}\in M(\Xscr_{\Gsf,\widehat S_z})^\otimes$ at the formal neighbourhood of $z\in(\Spf S)(\k)$ constructed via the deformation theory (\emph{cf.} \S\S~\ref{sect local geom}, Proposition~\ref{prop normalisation on formal neighbourhoods}). 
Since we can associate to $t_{\alpha,\et}$  ``display tensors''  only  over $\widehat S_\pfr$ for minimal open primes $\pfr\subset S$, we first compare the tensors over the formal neighbourhood of $\pfr\widehat S_z$. (Note that $\pfr\widehat S_z$ is a prime ideal; i.e., the completion $(S/\pfr)^\wedge_z \cong \widehat S_z/\pfr\widehat S_z$ is a domain. This is because $S/\pfr$ is normal by Corollary~\ref{cor normalisation on formal neighbourhoods}.)

\begin{propsub} \label{prop local tensors}
The two tensors $(t_{\alpha,\et}\otimes1)$ and $(t_{\alpha,z}^{\rm def}\otimes1)$ have the same image over $W((\widehat S_z)^\wedge_\pfr)$.
\end{propsub}
\begin{proof}
We may verify the proposition over $W(O_K)$ for some extension $O_K$ of $(\widehat S_z)^\wedge_\pfr$ as complete discrete valuation rings.
Let us choose a complete discrete valuation ring  $O_K$ over $ (\widehat S_z)^\wedge_\pfr$ with perfect residue field $\kappa$, and  write $\Xscr\coloneqq \Xscr_{\Gsf,O_K}$.

Let $C:=\widehat{\overline K}$. Then by Corollary~\ref{cor:Ccris} the tensors $(t^{\rm def}_{\alpha,z})$ induce $\Gamma_K$-invariant tensors, say $(t^{\rm def}_{\alpha,\et, K})$, of $T(\Xscr_C)[\frac{1}{p}]$. By construction of $(t^{\rm def}_{\alpha,\et, K})$,
it follows from Corollary~\ref{cor:CdR} that the natural isomorphism
\[
T(\Xscr_{C})\otimes_{\ZZ_p}B_{\dR}(O_C)\cong  \DD(\Xscr)_{O_K}^\vee\otimes_{O_K}B_{\dR}(O_C)
\]
matches $(t^{\rm def}_{\alpha,\et, K})$ and the image of the display tensors $(t^{\rm def}_{\alpha,z})$ in $ \DD(\Xscr)_{O_K}^\otimes$.

Now, recall that the images of  $t^{\rm def}_{\alpha, z}$ and $t_{\alpha,\dR}$ in $\DD(\Xscr)_{O_K}^\otimes$ coincide; \emph{cf.} \cite[proof of Proposition~4.2.6]{KisinPappas:ParahoricIntModel}. Since $t_{\alpha,\et}$ and $t_{\alpha,\dR}$ are related by the $p$-adic de~Rham comparison (by the theorem of Blasius and Wintenberger), we may  apply Corollary~\ref{cor:CdR} to $t_{\alpha,\et}$ and conclude that the  tensors $(t_{\alpha,\et})$ and $(t^{\rm def}_{\alpha,\et, K})$  have the same images in $(T(\Xscr_{C})\otimes_{\ZZ_p} B_\dR(O_C))^\otimes$. This shows that  $t_{\alpha,\et} = t^{\rm def}_{\alpha,\et, K}$. Now, we conclude the proof since $\Psi$-invariant tensors of $P(\Xscr)$ are uniquely determined by the associated $\Gamma_K$-invariant \'etale tensors (if they exist) by Corollary~\ref{cor:Ccris}.
\end{proof}

We want to deduce from Proposition~\ref{prop local tensors}  that $(t_{\alpha,\et})\subset P(\Xscr_{\Gsf,\widehat S_\pfr})^\otimes$ extends to $P(\Xscr_{\Gsf,S})^\otimes$. For this, we need the following lemma:
\begin{lemsub}\label{lem glueing}
Let $S$ be a topologically finitely generated normal flat $O_{\breve E}$-algebra such that $\Spec(S\otimes_{O_{\breve{E}}}\k)$ is reduced with normal irreducible components. Let $t\in \prod_\pfr \widehat{S}_\pfr$ where the product is over all minimal open prime ideals $\pfr\in\Spf S$, and assume that for any $z\in(\Spf S)(\k)$, the image of $t$ in $\prod_\pfr (\widehat S_z)^\wedge_\pfr$ lies in $\widehat S_z$. Then $t\in S$.

The same statement holds if we replace $S$, $\widehat S_z$ and $\widehat{S}_\pfr$ with the rings of Witt vectors thereof.
\end{lemsub}

\begin{proof}
The claim for the Witt vectors follows from the first claim by inspecting each Witt component.

For convenience, we choose a uniformiser $\varpi$ of $\breve E$. Then $\varpi$ is a uniformiser of $\widehat{S}_\pfr$ for any minimal open prime ideal $\pfr\in\Spf S$. Let $\kappa_\pfr=\widehat{S}_\pfr/(\varpi)$ denote the residue field. Let $t\in \prod_\pfr \widehat{S}_\pfr$ be as in the statement. Then since $S$ is closed in $\prod_\pfr \widehat{S}_\pfr$ for the $\varpi$-adic topology, it suffices to show that $t\bmod{(\varpi^n)}$ lies in $S/\varpi^nS$ for each $n$.

Let us write $\bar S\coloneqq S/\varpi S$ and $\widehat{\bar S}_z\coloneqq\widehat S_z/\varpi\widehat S_z$. For any $f$ in $S$, $\widehat S_z$ or $\widehat{S}_\pfr$, we write $\bar f\coloneqq f\bmod{(\varpi)}$.
For $t$ as in the statement, let us first show that $\bar t\in \bar S$. By reducedness,  $\bar S$ injects into $\prod_\pfr \kappa_\pfr$, which is the total fraction ring of $\bar S$. Therefore, 
$\bar t$ defines a section of the structure sheaf over some non-zero open $U\subset \Spec \bar S$. On the other hand, the condition on the image of $\bar t$ on $\prod_\pfr \Frac(\widehat{\bar S}_z/\pfr\widehat{\bar S}_z)$ implies that $\bar t$ lies in the local ring $\bar S_z$ for any closed point $z \in (\Spec\bar S)(\k)$, so we may take $U$ to contain all closed points. Now, since $\bar S$ is finitely generated over $\k$ (so $\bar S$ is Jacobson), we may take $U = \Spec\bar S$ by Zariski density. This shows that $\bar t\in \Gamma(\Spec \bar S,\Ocal) = \bar S$.

Now, assuming that $t\bmod{(\varpi^n)}$ lies in $S/\varpi^n S$,  we want to show that $t\bmod{(\varpi^{n+1})}$ lies in $S/\varpi^{n+1} S$. We choose a lift $t^{(n)}\in S$ of $t\bmod{(\varpi^n)}$, and set $\bar u\coloneqq \varpi^{-n}(t-t^{(n)}) \in \prod_\pfr \kappa_\pfr$. Then by the assumption on $t$, it follows that the image of $\bar u$ in $\prod_\pfr \Frac(\widehat{\bar S}_z/\pfr\widehat{\bar S}_z)$ lies in $\widehat{\bar S}_z$ for any $z\in(\Spf S)(\k)$. Therefore, by the previous paragraph, we have $\bar u\in \bar S$. Choosing a lift $u\in S$ of $\bar u$, $t^{(n)} + \varpi^n u \in S$ is congruent to $t$ modulo $(\varpi^{n+1})$. Now by sending $n\to \infty$, we conclude that $t\in S$ by completeness.
\end{proof}

\begin{proof}[Proof of Proposition~\ref{prop global tensors}]
To prove the proposition, it suffices to construct the tensors $(t_{\alpha})$ over any open affine formal subscheme $\Spf S\subset \wh\Sscr$ with the desired  description at each formal neighbourhood; indeed, the description at each formal neighbourhood allows the Zariski-locally constructed tensors to glue via the fpqc descent theory for flat modules over Witt rings (\emph{cf.} \cite[\S1.3, Lemma~30]{Zink:DisplayFormalGpAsterisq278}). We may and do assume that each $S$ is a normal domain, and the display $P(\Xscr_{\Gsf,S})$ is free over $W(S)$. (To ensure the freeness, we refine the open covering $\{\Spf S\}$ so that $P(\Xscr)_{\Gsf,S}\otimes_{W(S)}S$ is free over $S$. Then it follows that $P(\Xscr_{\Gsf,S})$ is also free over $W(S)$ by Nakayama's lemma.)

For any minimal open prime ideal $\pfr\subset S$, we have already constructed the tensors $(t_{\alpha,\et}\otimes 1)$ of $P(\Xscr_{\Gsf,\widehat S_\pfr})$ associated to the \'etale tensors $(t_{\alpha,\et})$ via Proposition~\ref{prop integrality Tate tensors}. Furthermore, by Proposition~\ref{prop local tensors} these tensors $(t_{\alpha,\et}\otimes 1)$ coincide with Kisin-Pappas tensors $(t_{\alpha,z}^{\rm def}\otimes1)$  over $W((\widehat S_z)^\wedge_\pfr)$ for any $z\in(\Spf S)(\overline\FF_p)$.

Note that $\Spec(S\otimes_{O_{\breve{E}}}\k)$ is reduced with normal irreducible components (\emph{cf.} Corollary~\ref{cor normalisation on formal neighbourhoods}). So by fixing a $W(S)$-basis of $P(\Xscr_{\Gsf,S})$ and inspecting the individual coordinates of
\[
t_{\alpha,\et}\otimes 1 \in \prod_\pfr  P(\Xscr_{\Gsf,S})\otimes_{W(S)} W(\widehat{S}_\pfr),
\]
it follows from Lemma~\ref{lem glueing} that $t_{\alpha,\et}\otimes 1\in P(\Xscr_{\Gsf,S})^\otimes$. And by construction, it has the desired description at the  formal neighbourhood of $z\in(\Spf S)(\k)$.
  \end{proof}

\begin{corsub} \label{cor global tensors}

 Let $ \bar\Sscr^\perf$ denote the perfection of the geometric special fibre $\bar\Sscr$, and let $\nu: \bar\Sscr^\perf \to \bar\Sscr$ denote the natural projection. Then there exists a family of crystalline Tate tensors $(\bar{t}_{\alpha})$ on $\nu^\ast \bar\Ascr_{\Gsf}[p^\infty]$ such that for every $z \in\Ssb(\k)$ the restriction of $(\bar{t}_{\alpha})$ to the formal fibre over $z$ coincides with the pullback of $(\bar{t}_{\alpha,z}^{\rm def})$. Furthermore, for any  $\tilde z\in\Sscr(O_K)$ lifting $z\in\Ssb(\k) = \Ssb^\perf(\k)$, where $K$ is a finite extension of $\breve E$, the crystalline comparison matches $t_{\alpha,\et,\tilde z}\in T_p(\Asf_{\Gsf,\tilde z})^\otimes$ with $ \bar{t}_{\alpha,z}\in M(\Ascr_{\Gsf,z}[p^\infty])^\otimes$.
  \end{corsub}
\begin{proof}
We use the notation from the proofs of Proposition~\ref{prop global tensors} and Lemma~\ref{lem glueing}. We set
\[\{\bar t_{\alpha}\coloneqq t_{\alpha}\otimes 1 \in P(\Xscr_{\Gsf})^\otimes\otimes_{W(\Ocal_{\wh\Sscr})}W(\Ocal_{\Ssb^\perf}) \}\]
over the perfection $\Ssb^\perf$. By construction of $\{t_\alpha\in P(\Xscr_{\Gsf}^\otimes)\}$, these tensors at the formal neighbourhood of $z\in (\Ssb)(\k)$ coincide with the pullback of $\{t_{\alpha,z}^{\rm def}\}$.  Finally, to obtain the comparison of $\{\bar t_\alpha\}$ with the \'etale tensors $\{t_{\alpha,\et}\}$ at $\tilde z\in\Sscr(O_K)$ (lifting a closed point $z$), it suffices to restrict the tensors to the formal neighbourhood of $z$. Now, using the description of the formal neighbourhoods given in Proposition~\ref{prop normalisation on formal neighbourhoods}, the claim follows from Proposition~\ref{prop local tensors}.
\end{proof}

\begin{corsub}\label{cor stratifications}
 Let $b \in G(\L)$.
 \begin{subenv}
  \item The Newton stratum $\Ssb^b(\k)$ is locally closed in $\Ssb(\k)$.
  \item The central leaf $\Cbar^b(\k)$ is closed in $\Ssb^b(\k)$.
 \end{subenv}
\end{corsub}
\begin{proof}
  For simplicity let $S \coloneqq \Ssb^\perf$ and $\Vfr \coloneqq \DD(\Ascr_G[p^\infty]_S)(W(\Ocal_S))$. Note that $(S,W(\Ocal_S))$ is a formal scheme, more precisely the unique flat lift of $S$ over $\Spf \ZZ_p$. Since $\underline\Isom\big((\Vfr,(t_\alpha)),(M\otimes W(\Ocal_S), (s_\alpha \otimes 1))\big)$ is a closed subfunctor of $\Isom(\Vfr,M\otimes W(\Ocal_S))$, it is representable by a formal scheme $\Gfr$, which is affine and of finite type over $(S,W(\Ocal_S))$. %
 As $\Gfr$ is isomorphic to $\Gscr_x$ over the (Witt vectors of the) formal fibre of any $z\in \Ssb(\k)$, we get that it is flat over any $z \in S(\k)$ by flat descent. Since $S$ is Jacobson, it follows that $\Gfr \to (S,W(\Ocal_S))$ is flat and thus a $\Gscr_x$-torsor. Since $\Gscr_x$ is smooth, $\Gfr$ is a torsor for the \'etale topology; in other words, $(\Vfr,t_\alpha)$ becomes isomorphic to $(M,s_\alpha)$ after an \'etale base change.

 Denote by $LG$ the sheaf on the \'etale site of $S$ given by $LG(U) = G(W(\Ocal_S(U))[\frac{1}{p}])$. This is indeed a sheaf since $W(\Ocal_S)[\frac{1}{p}]$ is (cf.~\cite[009F]{AlgStackProj}). By the above observation $\Pcal \coloneqq \underline\Isom((\Vfr[\frac{1}{p}],t_\alpha),(M \otimes W(\Ocal_S)[\frac{1}{p}],s_\alpha \otimes 1))$ is an \'etale $LG$-torsor and the Frobenius corresponds to an isomorphism $\tau\colon\sigma^\ast\Pcal \isom \Pcal$. For any $V \in {\rm Rep}_{\QQ_p} G$ let 
 \[
  \Fcal(V) \coloneqq (\Pcal \times LV)/LG, \tau \times \id)
 \]
where $LV \coloneqq V \otimes_{\QQ_p} W(\Ocal_S)[\frac{1}{p}]$ and the $LG$-action on $\Pcal \times LV$ is given by $g.(p,v) = (g.p,g^{-1}(v))$. Then $\Fcal$ is an $F$-isocrystal with $G$-structure in the sense of Rapoport and Richartz, whose isomorphism class over any $\k$-point $z$ is given by $[g_z]$. Thus the Newton stratum $\Ssb^b(\k) \subset \Ssb(\k)$ is locally closed \cite[Thm.~3.6]{RapoportRichartz:Gisoc}. Now $\Cbar^b(\k)$ is closed in $\Ssb^b(\k)$ by \cite[Thm.~2.2]{Oort:Foliations} and \cite[sect.~2.3, Lemma~2]{Hamacher:DeforSpProdStr}.
\end{proof}
 
 \section{Rapoport-Zink spaces of Hodge type with parahoric level} \label{sect RZ}

 In this section we define Rapoport-Zink spaces with parahoric and Bruhat-Tits level. We follow the construction of Howard and Pappas in the hyperspecial case in \cite{HowardPappas:GSpin}; however we will not obtain a nice moduli description as in the hyperspecial case since the Rapoport-Zink spaces with parahoric level fail to be smooth.

 \subsection{Local Shimura data}
 We recall the following definition from the article of Rapoport and Viehmann \cite{RapoportViehmann:LocShVar}

 \begin{defnsub}[{\cite[Def.~5.1]{RapoportViehmann:LocShVar}}]
  A local Shimura datum over $\QQ_p$ is a triple $(G,[b],\{\mu\})$ in which $G$ is a connected reductive group, $[b]$ is a $\sigma$-conjugacy class in $G(\breve\QQ_p)$ and $\{\mu\}$ is a $G(\QQbar_p)$-conjugacy class of cocharacters $\mu:\GG_{m,\QQbar_p} \to G_{\QQbar_p}$ such that the following conditions are satisfied.
  \begin{numlist}
   \item $\{\mu\}$ is minuscule.
   \item $[b]$ is contained in the set of neutral acceptable elements $B(G,\{\mu\})$.
  \end{numlist} 
 \end{defnsub}

 We will consider a subclass of local Shimura data with the following two additional assumptions.

 \begin{defnsub}
  Let $(G,[b],\{\mu\})$ be a local Shimura datum over $\QQ_p$.
  \begin{enumerate}
   \item $(G,[b],\{\mu\})$ is called tamely ramified if $G$ splits over a tamely ramified extension of $\QQ_p$.
   \item $(G,[b],\{\mu\})$ is called of Hodge type if there exists a faithful representation $\rho: G \mono \GL_n$ such that the central torus is contained in the image of $\rho$ and $(\GL_n,\rho([b]),\{\rho(\mu)\})$ is a local Shimura datum. We call the quadruple $(G,[b],\{\mu\}, \rho)$ an embedded local Shimura datum of Hodge type.
  \end{enumerate}
 \end{defnsub}

 Let $S \subset G$ be a maximal $\L$-split torus defined over $\F$ and $T$ its centraliser. By Steinberg's theorem $G_{\L}$ is quasi-split, thus $T$ is a maximal torus. Let $N$ be its normaliser and let $T(\L)_1 \subset T(\L)$ be its (unique) parahoric subgroup. Denote by
 \[
  W_0 \coloneqq N(\L)/T(\L)
 \]
 the finite Weyl group of $G_{\L}$, and denote by
 \[
  \Wtilde \coloneqq N(\L)/T(\L)_1
 \]
 its Iwahori-Weyl group. Choosing a special vertex, we have $\Wtilde = X_\ast(T)_I \rtimes W_0$. Moreover, the affine Weyl group embeds into $\Wtilde$ as  semi direct factor; a splitting corresponds to the choice of an alcove in the apartment $\Acal(T,\L)$ of $T$. Fixing such a choice, we extend the Bruhat order on the affine Weyl group to $\Wtilde$, denoted by $\leq$, via the canonical choice. Recall that $\Wtilde$ parametrises double cosets of $G(\L)$ by parahoric subgroups corresponding to points in $\Acal(T,\L)$. More precisely, let $W^x \coloneqq N(\L) \cap \Gscr_x^\circ(\W)/T(\L)_1$ for $x \in \Acal(T,\L)$. By \cite[Appendix, Prop.~8]{PappasRapoport:AffineFlag} the canonical embedding $N(\L) \mono G(\L)$ induces an isomorphism
 \[
  W^{x_1}\backslash \Wtilde / W^{x_2} \isom \Gscr_{x_1}^\circ(\W)\backslash G(\L) / \Gscr_{x_2}^\circ(\W).
 \]
We will mostly be interested in the double cosets corresponding the $\mu$-admissible set defined in \cite[(3.4)]{Rapoport:GuideRednShVar}, which can be described as
\[
 \Adm(\mu) \coloneqq \{\wtilde \in \Wtilde \mid \exists w \in W_0: \wtilde \leq w(\mu)\}.
\]

 We will later on require the existence of a completely slope divisible {\BT} with additional structure corresponding to some embedded Shimura datum. The group theoretic analogue of ``completely slope divisible'' for an element $b \in G(\L)$ is being descent, i.e.\ for some integer $s$ we have
  \[
   (b\sigma)^s = (s\cdot\nu)(p) \sigma^s.
  \]

 \begin{lemsub} \label{lemma decent}
  There exists a decent element $b_0 \in [b] \cap \bigcup_{\wtilde \in \Adm(\mu)} \Gscr_x^\circ(\breve\ZZ_p) \wtilde \Gscr_x^\circ(\breve\ZZ_p)$.
 \end{lemsub}
 \begin{proof}
  By \cite[Lemma~2.2.10]{Kim:ProductStructure} any element of $\Wtilde$ has a decent lift to $N(\L)$. Thus it suffices to prove that $\Adm(\mu) \cap [b]$ is non-empty. This is proven in \cite[Theorem~2.1]{He:KottwitzRapoportConj}.
 \end{proof}

 \subsection{Affine Deligne-Lusztig varieties of Hodge type} \label{sect ADLV}
 We fix a  embedded local Shimura datum $(G,[b],\{\mu\},\rho)$ with the property that $(G,[b],\{\mu\})$ is tamely ramified for the remainder of this section and denote $[b'] \coloneqq \rho([b]), \{\mu'\} \coloneqq \{\rho(\mu)\}$ and by $c\colon \GG_m \to G$ the cocharacter giving scalar multiplication. Since we are interested in the affine Deligne-Lusztig varieties  associated to Shimura varieties of Hodge type, we moreover assume that $\rho$ is a symplectic embedding and that $\mu' = (0^{(g)},1^{(g)})$. Let $x,x',\Gscr_x^\circ,\Gscr_x,\Gscr'_{x'}$ be as in section~\ref{sect group theoretic background} where we set $G' = \GSp_n$ and assume $x'$ hyperspecial. We denote by $(\Lambda,\psi) \subset \QQ_p^n$ the self-dual lattice corresponding to $x'$.

 For any smooth affine group scheme $\Gscr$ over $\ZZ_p$ with reductive generic fibre we denote by $\Grass_{\Gscr}$ the corresponding Witt vector affine Grassmannian (\cite[\S1.4.1]{Zhu:Satake}). The functor $\Grass_{\Gscr}$ is represented by an inductive limit of perfection of projective varieties by \cite[Cor.~9.6]{BhattScholze:WittVectorAffGr}. In particular, any locally closed perfect subscheme is uniquely determined by its $\FFbar_p$-valued points. 
 
 First consider the classical affine Deligne-Lusztig variety $X_{\mu'}(b')_{\Ksf_p'} \subset \Grass_{\Gscr'_{x'}}$ given by
 \[
  X_{\mu}(b')_{\Ksf_p'}(\FFbar_p) \coloneqq \{ g\cdot \Ksf_p' \in G'(\breve\QQ_p)/\Ksf_p' \mid g^{-1}b\sigma(g) \in \Ksf_p'\mu'(p)\Ksf_p' \}
 \] 
  We assume that $b \in G(\L)$ is a decent element satisfying the constraints of Lemma~\ref{lemma decent} and that $b' = \rho(b)$. We denote by $(\XX,\lambda_\XX)$ the {\BT} with polarised Dieudonn\'e module $(\Lambda \otimes_\O \W,\psi \otimes 1, b'\sigma^\ast)$. Using Dieudonn\'e theory (see e.g.\ \cite[Prop.~3.11]{Zhu:Satake}), one obtains that $X_{\mu'}(b')$ is the perfection of the reduced special fibre of the Rapoport-Zink space $\Mfr^{b'}$ which is the formal scheme over $\Spf \W$ given by
 \[
  \Mfr^{b'}(R) = \left\{(\Xscr,\lambda,\varphi) \mid \begin{array}{l} (\Xscr,\lambda) \textnormal{ polarised BT of height } 2g \textnormal{ over } R \\ \varphi\colon (\XX,\lambda_\XX) \otimes R/p \to (\Xscr,\lambda) \otimes R/p \textnormal{ quasi-isogeny} \end{array} \right\}.
 \]
 
 Similarly as above, we define $X_\mu(b) \subset \Grass_{\Gscr_x}$ by
 \[
  X_{\mu}(b)_{\Ksf_p}(\FFbar_p) \coloneqq \{ g\cdot \Ksf_p \in G(\breve\QQ_p)/\Ksf_p \mid g^{-1}b\sigma(g) \in \bigcup_{\wtilde \in \Adm(\mu)} \Ksf_p \wtilde \Ksf_p \}.
 \]
  Let $J_b$ denote the twisted centraliser of $b\sigma$, i.e. it denotes the linear algebraic group over $\F$ given by
 \[
  J_b(R) = \{g\in G(R\otimes_\F \L) \mid g^{-1}b(1 \otimes \sigma)^\ast(g) = b\}.
 \]
 Then $J_b(\F)$ acts on $X_{\mu}(b)_{\Ksf_p}$ via left multiplication, as this does not change the value of $g^{-1}b\sigma(g)$.
 
 \begin{rmksub} \label{rmk parahoric level}
  To compare our definition with the more standard definition of the affine Deligne-Lusztig variety with parahoric level, let $\Gscr_x^\circ$ denote the parahoric group scheme corresponding to $x$, $\Ksf_p^\circ = \Gscr_x^\circ(\ZZ_p)$ and let
   \[
    X_\mu(b)_{\Ksf_p^\circ}(\FFbar_p) = \{ g\cdot \Ksf_p^\circ \in G(\L)/\Ksf_p^\circ \mid g^{-1} b \sigma(g) \in \bigcup_{\wtilde\in \Adm(\mu)} \Ksf_p^\circ \wtilde \Ksf_p^\circ\}.
 \]  
 Since the positive loop group of the parahoric group scheme $L^+\Gscr_x^\circ$ is the unit component of $L^+\Gscr_x$ by \cite[Appendix, Prop.~3]{PappasRapoport:AffineFlag}, the canonical morphism $\Grass_{\Gscr_x^\circ} \to \Grass_{\Gscr_x}$ can be identified with $\coprod_{\Ksf_p/\Ksf_p^\circ} \Grass_{\Gscr_x} \to \Grass_{\Gscr_x}$. In particular the (standard) affine Deligne-Lusztig variety $X_\mu(b)_{\Ksf_p^\circ}$ with parahoric level is also a disjoint union of copies of $X_\mu(b)_{\Ksf_p}$.
 \end{rmksub}

 In order to give a modular description of $X_\mu(b)_{\Ksf_p}$, we choose tensors $(s_\alpha)$ in $\Lambda^\otimes$ such that $\Gscr_x$ is their stabiliser. This yields crystalline Tate-tensors $t_{\XX,\alpha} \coloneqq s_\alpha \otimes 1$ on $\XX$. Note that $\rho$ induces an embedding $\rho_{\Grass}\colon \Grass_{\Gscr_x} \mono \Grass_{\Gscr'_{x'}}$ mapping $X_\mu(b)_{\Ksf_p}$ into $X_{\mu'}(b')_{\Ksf_p'}$. Moreover, it was proven in \cite[Prop.~3.4]{Zhou:ModpPts} that the image of $\bigcup_{\wtilde\in \Adm(\mu)} \Ksf_p \wtilde \Ksf_p$ in $\Ksf_p'\mu(p) \Ksf_p'/ \Ksf_p' \cong (\Mscr_{G',\mu',x',\k}^{\loc})^{\perf}$ equals $(\Mscr_{G,\mu,x,\k}^{\loc})^{\perf}$
 \footnote{Note that the two conditions Zhou imposed at the beginning of section 3  in \cite{Zhou:ModpPts}, i.e.\ that $G$ is quasi-split and $\Ksf_p = \Ksf_p^\circ$, are not needed for this proposition. Since the construction of local models commutes with base change, we may exchange $\QQ_p$ by an unramified extension $F$ such that $G_F$ is quasi-split. Moreover $\Ksf_p = \Ksf_p^\circ$ is not needed in the proof; while \cite[Prop.~3.4]{Zhou:ModpPts} then proves the claim for $\Ksf_p^\circ$, we know that the canonical morphism $\bigcup_{\wtilde\in \Adm(\mu)} \Ksf_p^\circ \wtilde \Ksf_p^\circ /\Ksf_p^\circ \to \bigcup_{\wtilde\in \Adm(\mu)} \Ksf_p \wtilde \Ksf_p$ is an isomorphism by the argument of Remark~\ref{rmk parahoric level}}.

 \begin{deflemsub}
  $X_\mu(b)_{\Ksf_p} \subset X_{\mu'}(b')_{\Ksf_p'}$ is the subfunctor parametrising quasi-isogenies $\varphi$ which satisfy the following additional assumptions.
  \begin{assertionlist}
   \item $(\varphi_\ast t_{\XX,\alpha})$ are a collection of crystalline Tate tensors on $\Xscr$ such that
   \[
    \Isom((\DD(\Xscr),\psi\otimes 1,\varphi_\ast t_{\XX,\alpha}),(\Lambda \otimes_{\ZZ_p} W(R),\psi\otimes 1, s_\alpha \otimes 1))
   \]
   is a crystal of $\Gscr_x$-torsors.
   \item For some (or equivalently any) \'etale cover $U  \to \Spec R$ such that there exists an isomorphism $\DD(\Xscr_U)_U \cong \Lambda \otimes_{\ZZ_p} \Ocal_U$ respecting tensors, the Hodge filtration $\Fil^0(\Xscr_U) \subset \DD(\Xscr_U)_U$ gets identified with a point of $\Mscr^{\rm loc}_{G,\{\mu\},x}$.
  \end{assertionlist}
   We denote by $(\overline\Xscr_G, \overline\lambda_G, \overline t_{\alpha,G},\overline\varphi_G)$ the universal object over $X_\mu(b)_{\Ksf_p}$.
 \end{deflemsub}
 \begin{proof}
  The proof is the same as that of \cite[Prop.~3.11]{Zhu:Satake}.
 \end{proof}

 \subsection{Definition of Rapoport-Zink spaces} \label{ssect construction of RZ}

 In the following we assume that $\rho$ is a symplectic representation and that $(G,\{\mu\},[b],\rho)$ is induced by an embedding of Shimura data $(\Gsf,\Xsf) \to ({\sf GSp}_{n}, \Ssf^\pm)$ as described in section~\ref{sect integral models}. At the beginning of our construction, we choose a pair $\ztilde =(z,j)$ where $z \in \Ssb^b(\FFbar_p)$ and $ j:\XX\to \Ascr_{\Gsf,z}[p^\infty]$ is a quasi-isogeny compatible with additional structure. We will later show that the constructed Rapoport-Zink space is independent of the choice of $\ztilde$.

 Let $\Theta_{\tilde{z}}': \Mfr^{b'} \to \Sscr'_{\breve\ZZ_p}$ denote the Rapoport-Zink uniformisation map, which associates to a quasi-isogeny $\varphi$ the abelian variety $\Ascr_{\Gsf,z}/\ker(\varphi\circ j^{-1})$ with the induced polarisation and level structure. The following lifting property of $\Theta_{\tilde{z}}'$ is a key ingredient in the construction of the Rapoport Zink space. 
As it is not yet proven for all cases, we formulate it as an axiom.
 \begin{axiom} \label{axiom RZ}
  The induced map on $\k$-valued points
  \[
   X_{\sigma(\mu)}(b)_{\Ksf_p}(\k) \mono X_{\mu'}(b')(\k) \stackrel{\Theta_{\tilde{z}}'}{\longrightarrow} \Ascr_g(\k)
  \]
  factors through $\Sscr^-(\k)$. Moreover, there exists a unique lift $f_{\tilde{z}}: X_{\sigma(\mu)}(b)(\k) \to \Sscr(\k)$ such that $(\overline t_{\alpha,G})_P = f^\ast (t_{\alpha,f(P)})$ for every $P \in X_{\sigma(\mu)}(b)_{\Ksf_p}(\k)$.
 \end{axiom}

 Kisin has proven this statement in the case that $\Ksf_p$ is hyperspecial in \cite[Prop.~1.4.4]{Kisin:LanglandsRapoport} and currently Zhou has finished its proof under the assumption that $G$ is residually split (\cite[Prop.~6.7]{Zhou:ModpPts}). Here we have consider $\sigma(\mu)$ rather than $\mu$ because we work with the contravariant Dieudonn\'e module (see \cite[\S~5.4]{Zhou:ModpPts} for details).

 \begin{rmksub} \label{rmk nonemptiness of central leaves}
  Note that if Axiom~\ref{axiom RZ} holds true, then the non-emptiness of $\Ssb^b$ implies the non-emptiness of $\Cbar^b$ as $f_{\tilde{z}}(\id_{\XX}) \in \Cbar^b$.
 \end{rmksub}

 In order to define the Rapoport-Zink space with Bruhat-Tits level, we first construct larger formal scheme containing additional connected components. We define the pull-back
 \begin{center}
  \begin{tikzcd}
   \Mfr^{b,\diamond} \arrow{r} \arrow{d}{\Theta_{\ztilde}^\diamond} & \Mfr^{b'} \arrow{d}{\Theta_{\ztilde}'} \\
   \Sscr \arrow{r}{\iota} & \Sscr'.
  \end{tikzcd}
 \end{center}

 \begin{lemsub} \label{lemma M diamond}
  The formal scheme is $\Mfr^{b, \diamond}$ is normal and locally of finite type and the morphism $\Mfr^{b,\diamond} \to \Mfr^{b'}$ is finite.
 \end{lemsub}
 \begin{proof}
  Since $\iota:\Sscr \to \Sscr'$ is finite so is its pull-back $\Mfr^{b,\diamond} \to \Mfr^{b'}$. Hence $\Mfr^{b,\diamond}$ is locally of finite type.  The normality can be checked on the formal neighbourhood of closed points. By the rigidity of quasi-isogenies the formal neighbourhood of $\Mfr^{b'}$ at any point is canonically isomorphic to the deformation space of the corresponding polarised {\BT} and hence $\Theta_{\ztilde}'$ induces isomorphisms of formal neighbourhoods. Thus $\Theta_z^\diamond$ induces isomorphisms $(\Mfr^{b,\diamond})^\wedge_y \isom (\Sscr_{G,O_{\breve{E}}})^\wedge_{\Theta_{\tilde{z}}^\diamond(y)}$, proving the normality of formal neighbourhoods.
 \end{proof}

 We denote by $(\Xscr^\diamond,\lambda^\diamond,\varphi^\diamond)$ the pull-back of the universal object over $\Mfr^{b'}$ to $\Mfr^{b,\diamond}$. Note that $\DD(\Xscr^\diamond)[\frac{1}{p}]$ is equipped with two different families of tensors. We denote by $(t_\alpha^\diamond)$ the family of tensors obtained from $(t_{\XX,\alpha})$ via the identification $\DD(\XX)[\frac{1}{p}] \cong_{\varphi^\diamond} \DD(\Xscr^\diamond)[\frac{1}{p}]$ and denote $u_{\alpha,y}^\diamond \coloneqq \Theta_{\tilde{z}}^{\diamond\, \ast}(t_{\alpha,\Theta_{\tilde{z}}^\diamond(y)})$ for every geometric point $y \in \Mfr^{b, \diamond}(\k)$.

 \begin{lemsub} \label{lemma clopen}
  The set of $y \in \Mfr^{b,\diamond}(\k)$ with $t_{\alpha,y}^\diamond = u_{\alpha,y}^\diamond$ is a union of connected components.
 \end{lemsub}
 \begin{proof}
  It suffices to prove the claim after taking perfection. Then by Corollary~\ref{cor global tensors} the $(u_{\alpha,y})$ are interpolated by $u_\alpha \coloneqq \Theta_{\ztilde}^{\diamond, \perf, \ast} (\overline{t}_\alpha)$.

   Let $T_\Lambda \subset \Lambda^\otimes$ be a \emph{finite} direct sum of combinations of tensor products, duals, alternating sums and symmetric sum of $\Lambda$ such that it contains $(s_\alpha)$.  Then the corresponding sub-$F$-crystal $T \subset \DD(\Xscr^\diamond_{\Mfr^{b,\diamond,{\rm red},\perf}})^\otimes$ contains both families of crystalline Tate-tensors. Since the perfect scheme representing $\Hom(\mathbf{1},T)$ is totally disconnected by Proposition~\ref{prop homomorphisms of F-crystals}, the locus where the tensors coincide is indeed a union of connected components.
 \end{proof}

 \begin{defnsub}
  The Rapoport-Zink space $\Mfr^b$ is defined as the formal subscheme of $\Mfr^{b,\diamond}$ corresponding to the union of connected components $\{t_{\alpha,y}^\diamond = u_{\alpha,y}^\diamond\}$. We define the uniformisation morphism $\Theta_{\tilde{z}}: \Mfr^b \to\Sscr$ as the restriction of $\Theta_{\tilde{z}}^\diamond$.
 \end{defnsub}

 \begin{propsub} \label{prop RZ space}
  The canonical morphism $\Mfr^b \to \Mfr^{b'}$ is a closed immersion and is uniquely determined by the following two properties.
  \begin{assertionlist}
   \item Its restriction to the perfection of the underlying reduced subscheme coincides with $X_{\mu}(b)_{\Ksf_p} \subset X_{\mu'}(b')$.
   \item For any $y \in X_{\mu}(b)_{\Ksf_p}(\k)$, its restriction to the formal neighbourhood at $y$ coincides with $\Def(\Xscr^\diamond_y,\lambda^\diamond_y,(t_{\alpha,y}^\diamond)) \mono \Def(\Xscr^\diamond_y,\lambda^\diamond_y)$.
  \end{assertionlist}
  In particular the formal scheme $\Mfr^b$ and the embedding $\Mfr^b \to \Mfr^{b'}$ only depend on $b,\mu$ and $\iota$ and not on the choice of $\ztilde$.
 \end{propsub}
 \begin{proof}
  Recall that we showed in the proof of Lemma~\ref{lemma M diamond} that $(\Mfr^{b})^\wedge_y \isom (\Sscr_{G,O_{\breve{E}}})^\wedge_{\Theta_{\tilde{z}}^\diamond(y)}$. Since we have $(\Sscr_{G,O_{\breve{E}}})^\wedge_{\Theta_{\tilde{z}}^\diamond(y)} \cong \Def(\Xscr^\diamond_y,\lambda^\diamond_y,(t_{\alpha,y}^\diamond))$ by Proposition~\ref{prop normalisation on formal neighbourhoods}, it follows that $\Mfr^b$ satisfies (b).
 
 To show that it also satisfies (a), we examine the $\k$-point of $\Mfr^b$. By definition, we have 
   \[
    \Mfr^b(\k) = \{(X,\lambda,\varphi;x) \in \Mfr^{b'}(\k) \times_{\Sscr'(k)} \Sscr(k) \mid \varphi_\ast (t_{\XX,\alpha}) = \Theta^{\diamond, \ast}_{\ztilde} t_{\alpha,x} \in \DD(X)^\otimes\}. 
   \]
   In particular, we get $(X,\lambda,\varphi) \in X_\mu(b)_{\Ksf_p}(\k)$. On the other hand, Axiom~\ref{axiom RZ} says that for any  $(X,\lambda,\varphi) \in X_\mu(b)_{\Ksf_p}(\k)$ there is a unique lift $x \in \Sscr(\k)$ of $\Theta'_{\ztilde}(X,\lambda,\varphi)$ such that $\varphi_{\ast} (t_{\XX,\alpha}) = \Theta_{\ztilde}^{\diamond, \ast} (t_{\alpha,x})$. Hence the canonical morphism $\Mfr^b \to \Mfr^{b'}$ maps $\Mfr^b(\k)$ bijectively onto $X_\mu(b)_{\Ksf_p}(\k)$. Thus in order to deduce (a) it remains to show  that $\Mfr^{b} \to \Mfr^{b'}$ is a closed immersion. As we already know that $\Mfr^{b} \to \Mfr^{b'}$ is finite by Lemma~\ref{lemma M diamond}, the locus where it is a closed immersion is open on the target. By (b) this locus must contain all closed points, thus the morphism is indeed a closed immersion.

   To show uniqueness, assume that there exist two closed formal subschemes $Z_1, Z_2$ satisfying (a) and (b). Let $Z$ denote their intersection. Then $Z \mono Z_i$ is a surjective closed immersion which is an isomorphism on formal neighbourhoods, which implies that it is also \'etale and in particular open. Thus $Z_1 = Z = Z_2$.
 \end{proof}

\begin{rmksub}
In general, it is not yet known whether the integral model $\Sscr=\Sscr_{\Ksf^p\Ksf_p}$ of $\Sh_{\Ksf^p\Ksf_p}(\Gsf,\Xsf)$, constructed by Kisin and Pappas when $\Ksf_p = \Gscr_x(\Zp)$, is independent of the choice of embedding $(\Gsf,\Xsf)\mono (\mathsf{GSp}_{2g},\Ssf)$. (This is only known if  $x$ is a \emph{very special} vertex of $G=\Gsf_{\QQ_p}$ and $G^{\ad}$ is absolutely simple; \emph{cf.} \cite[Prop.~4.6.28]{KisinPappas:ParahoricIntModel}). Therefore, it is not known in general whether $\Mfr^b$ is independent of the choice of the embedding $\rho:(G,[b],\{\mu\})\mono(\GL_n,[b'],\{\mu'\})$ of local Shimura data coming from some global embedding of Shimura data.
\end{rmksub}

 \subsection{Group actions on $\Mfr^{b}$}
 In \cite{RapoportZink:RZspace} Rapoport and Zink equipped $\Mfr^{b'}$ with the action of the group of self quasi-isogenies $J_{b'}(\QQ_p)$ of $(\XX,\lambda_\XX)$ and a Frobenius-twisted action of the Weil group of $\QQ_p$. In this section we discuss how to restrict these actions to $\Mfr^b$.

 First consider the group functor
  \begin{align*}
   \QIsog_{G'}(\XX)\colon ({\rm Nilp}_{\W}) &\to ({\rm Sets}) \\
   R &\mapsto \{\varphi\colon \XX_{R/p} \to \XX_{R/p} \textnormal{ quasi-isogeny, commuting with } \lambda_\XX \textnormal{ up to scalar }\}.
  \end{align*}
  It is representable by a flat formal group scheme over $\Spf \W$ with perfect special fibre by \cite[Lemma~4.2.10]{CaraianiScholze:ShVar} (see also \cite[\S~3.2]{Kim:ProductStructure}. By Dieudonne theory, the group of locally constant self quasi-isogenies is represented by a locally pro-$p$ group $J_{b'}(\QQ_p)$ viewed as a locally constant formal scheme over $\Spf \W$. By \cite[Lemma~4.2.11]{CaraianiScholze:ShVar} the canoncial embedding $J_{b'}(\QQ_p) \mono \QIsog_{G'}(\XX)$ splits, yielding an isomorphism $\QIsog_{G'}(\XX) \cong \QIsog_{G'}(\XX)^\circ \rtimes J_{b'}(\QQ_p)$, where $\QIsog_{G'}(\XX)^\circ$ denotes the identity component of $\QIsog_{G'}(\XX)$. Obviously, the group $\QIsog_{G'}(\XX)$ acts via precomposition on $\Mfr^{b'}$, exteding the action of $J_b(\QQ_p)$ defined in \cite{RapoportZink:RZspace}.

 As the quasi-isogenies in $\QIsog_{G'}(\XX)$ do not need to fix $(t_{\XX,\alpha})$, the Rapoport-Zink space $\Mfr^b$ is in general not $\QIsog_{G'}(\XX)$ stable. We denote by $\QIsog_G(\XX) \subset \QIsog_{G'}(\XX)$ the formal subgroup defined in \cite[Def.~3.2.1]{Kim:ProductStructure}. While a moduli description for general $R \in {\rm Nilp}_{\W}$ is unknown, in the case where $R/p$ is semiperfect it parametrises precisely those quasi-isogenies which preserve $(t_{\XX,\alpha})$.

 \begin{lemsub} \label{lemma J_b-action}
  $\Mfr^b$ is stable under the action of the subgroup $\QIsog_G(\XX)$.
 \end{lemsub}
 \begin{proof}
 By Proposition~\ref{prop RZ space}, we need to show that the action of $\QIsog_G(\XX)$ on $\Mfr^{b'}$ preserves conditions (a) and (b) of Proposition~\ref{prop RZ space}.
  Since $\QIsog_G(\XX) \cong \QIsog_G(\XX)^\circ \rtimes J_b(\QQ_p)$ by \cite[Rmk.~3.2.5]{Kim:ProductStructure}, we may handle the action of  $J_b(\QQ_p)$ and $\QIsog_G(\XX)^\circ$ separately.

 Let us first verify that for every $g \in J_b(\QQ_p)$ the formal subscheme $g_\ast\Mfr^{b} \subset \Mfr^{b'}$ satisfies  conditions (a) and (b) of Proposition~\ref{prop RZ space}. Condition (a) follows from Axiom~\ref{axiom RZ} in section \ref{ssect construction of RZ}. To check condition (b), from the canonical isomorphism $g^\ast(\Xscr^\diamond,\lambda^\diamond,(t_\alpha^\diamond)) \cong (\Xscr^\diamond,\lambda^\diamond,(t_\alpha^\diamond))$  we obtain
  \[
   (g_\ast\Mfr^b)_{g\cdot z}^\wedge = g_\ast\Def(\Xscr^\diamond_z,\lambda^\diamond_z,(t_{\alpha,z}^\diamond)) = \Def(g^\ast(\Xscr^\diamond_z,\lambda^\diamond_z,(t_{\alpha,z}^\diamond))) = \Def(\Xscr^\diamond_{g\cdot z},\lambda^\diamond_{g\cdot z},(t_{\alpha,g\cdot z}^\diamond)).
  \]

It remains to show that the action of $\QIsog_G(\XX)^\circ$ on $\Mfr^{b'}$ preserves conditions (a) and (b) of Proposition~\ref{prop RZ space}. Condition (a) is automatic since $\QIsog_G(\XX)^{\circ, \red} = \{\id\}$, and condition (b) is just a reformulation of \cite[Thm.~4.3.1]{Kim:ProductStructure}.
 \end{proof}

 \begin{rmksub}
  In particular we get an action of $J \coloneqq J_b(\QQ_p) \subset \QIsog_G(\XX)$ on $\Mfr^b$. Let us briefly remark why the conditions of Theorem~\ref{thm:Poincare} are satisfied. First note that the action of $J_{b'}(\QQ_p)$ on $\Mfr^{b'}$ satisfies these properties by \cite[\S~5.2]{Mieda:Zelevinsky}. Since $\Mfr^b \subset \Mfr^{b'}$ is a closed formal subscheme, it immediately follows that the $J$-action is continuous in the sense of Fargues and that $\Mfr^{b,{\rm red}}$ is partially proper. 
  
  Moreover, following Mieda's argument, we see that the finiteness condition (\ref{thm:Poincare:Finiteness}) of Theorem~\ref{thm:Poincare} follows from the finiteness of $J$-orbits on the set of irreducible components of $\Mfr^b$. The latter assertion was proven in \cite{RapoportZink:BTbuildings} (see also \cite[Lemma~1.3]{HamacherViehmann:IrredComp}).
  
  Finally, the local algebraisability of $\Mfr^b$, and the smoothness and equidimensionality of the generic fibre can be deduced from the Rapoport-Zink uniformisation. Following the proof of \cite[Cor.~3.1.4]{Fargues:AsterisqueLLC}, we see for any quasi-compact open $U' \subset \Mfr^{b'}$ there exists a level $\Ksf'^p \subset \mathsf{GSp}_{2g}$ such that $\Theta_{\ztilde}'$ maps $U'$ isomorphically onto the completion of a closed subscheme of $\Sscr_{G'}$. Thus $\Theta_{\ztilde}$ identifies $U \coloneqq U' \cap \Mfr^b$ with the completion of a closed subscheme of $\Sscr_G$; in particular $\Mfr^b$ is locally algebraisable and $U^{\rm ad}$ is an open subspace of $(\widehat\Sscr_G)^{\rm ad}$, thus proving that the generic fibre of $\Mfr^b$ is smooth of pure dimension $\langle 2\rho, \mu\rangle$.
 \end{rmksub} 

 Next, we define the action of the Weil group $W_E$ on $\Mfr^b$. Since $\Mfr^b$ is a formal $O_{\Ebreve}$-scheme, the action of the inertia group will be trivial. Thus the $W_E$-action is given by a Weil descent datum, that is a $\sigma_{\kappa_E}$-semilinear isomorphism $\alpha_\F\colon \Mfr^{b'} \isom {\Mfr^{b'}}$ defining the action of the Frobenius (cf.~\cite[Def.~3.45]{RapoportZink:RZspace}). The Weil descent datum for $\Mfr^{b'}$ is given by
 \[
  (\Xscr,\varphi) \mapsto (\Xscr,\varphi \circ F^{-1}).
 \]
 Since any quasi-isogeny commutes with the Frobenius, the action of $\QIsog_{G'}(\XX)$ and $W_{\QQ_p}$ commute.

 Note that since $\Mscr^{\rm loc}_{G,\mu,x}$ is defined over $O_E$ and not $\ZZ_p$, we should not expect $\Mfr^b$ to be stable under $W_{\QQ_p}$. However, the following holds.

 \begin{lemsub} \label{lem:RZWeilDescentDatum}
  The $W_{\QQ_p}$ descent datum on $\Mfr^{b'}$ restricts to a $W_E$ descent datum in $\Mfr^b$.
 \end{lemsub}
 \begin{proof}
  The proof is basically the same as the proof of Lemma~\ref{lemma J_b-action}. A short calculation shows that the diagram
  \begin{center}
   \begin{tikzcd}
    X_{\mu'}(b') \arrow{r} \arrow{d}{\alpha_{\QQ_p}} &  \left[\Gscr'_{x'} \backslash {\Mscr^{\rm loc}_{G',\mu'-c,x'}}\right] \arrow{d}{\sigma_{\QQ_p}}  \\
    X_{\mu'}(b') \arrow{r}  &  \left[\Gscr'_{x'}\backslash\Mscr^{\rm loc}_{G',\mu'-c,x'} \right]
   \end{tikzcd}
  \end{center}
  commutes, implying condition (a) in Proposition~\ref{prop RZ space} for $\alpha_E(\Mfr^b)$ since $\Mscr^{\loc}_{G,\mu,x}$ is defined over $O_E$. The embedding of formal neighbourhoods at $(X,\varphi) \in \Mfr^b(\k)$ is identified with $\Def_{\Gscr_x}(X) \subset \Def_{\Gscr'_{x'}}(X)$ which under $\alpha_E$ is mapped to $\Def_{\Gscr_x}(X^{(\sigma_E)}) \subset \Def_{\Gscr'_{x'}}(X)^{(\sigma_E)}$. As $\Def_{\Gscr_x}(X^{\sigma_E}) = \Def_{\Gscr_x}(X)^{\sigma_E}$ by construction, condition (b) in Proposition~\ref{prop RZ space} follows.
 \end{proof}

 \begin{rmksub}
  Since any quasi-isogeny commutes with the actions of Frobenius, the action of $\QIsog_{G'}(\XX)$ and $W_{\QQ_p}$ on $\Mfr^{b'}$ commute. In particular, the actions of $\QIsog_G(\XX)$ and $W_E$ on $\Mfr^b$ commute.
 \end{rmksub}

 Naturally, we do not have an action of $\Gsf(\AA_f^p)$ on $\Mfr^b$. However, we can say that the construction in section~\ref{ssect construction of RZ} is compatible with the $G(\AA_f^p)$-action.

 \begin{lemsub}
  Let $g \in \Gsf(\AA_f^p)$ and $\Ksf_1^p \subset \Gsf(\AA_f^p)$ small enough compact open such that $g^{-1} \Ksf_1^p g \subset \Ksf^p$. We fix $z_1 \in \Sscr_{\Ksf_1^p\Ksf_p}(\k)$ and a quasi-isogeny $j_1:\XX \to \Ascr_{\Gsf,z_1}[p^\infty]$ which is compatible with additional structure and let $z = g(z_1) \in \Sscr_{\Ksf^p\Ksf_p}(\k)$ and $j$ the concatenation of $j_1$ with the canonical isomorphism $\Ascr_{\Gsf,z_1}[p^\infty] \cong \Ascr_{\Gsf,z}[p^\infty]$. Then the diagram
  \begin{center}
   \begin{tikzcd}
    & \Sscr_{\Ksf_1^p\Ksf_p} \arrow{d}{g} \\
    \Mfr^b \arrow{ru}{\Theta_{(z_1,j_1)}} \arrow{r}[swap]{\Theta_{(z,j)}} & \Sscr_{\Ksf^p\Ksf_p}
   \end{tikzcd}
  \end{center}
  commutes.
 \end{lemsub}
 \begin{proof}
  For $\Gsf = \Gsf'$, the commutativity follows from the moduli description (cf.~\cite[\S~6.13]{RapoportZink:RZspace}). Since $g$ induces a canonical isomorphism $\Ascr_{\Gsf,z_1}[p^\infty] \cong \Ascr_{\Gsf,g(z_1)}[p^\infty]$, which is compatible with crystalline Tate-tensors, for any $z_1 \in \Sscr_{G,\Ksf_1^p}(\k)$ and a canonical isomorphism on formal neighbourhoods since it is \'etale, the claim now follows from the uniqueness of the lift in Proposition~\ref{prop RZ space}.
 \end{proof}

 \section{The geometry of Newton strata}\label{sect Almost Product}
 The almost product structure on the Newton stratification is the key input from geometry to prove Mantovan's formula for Shimura varieties of Hodge type. We take the product structure over $(\Ssb')^{b'}$ constructed by Caraiani and Scholze as a starting point and show that one can restrict it to obtain the product structure over  $\Ssb^{b}$.

 \subsection{Igusa varieties}\label{ssect Igusa Var}
 We continue to assume that $b$ is a decent element as in Lemma~\ref{lemma decent} and that $\Ssb^b$ is non-empty. In particular $\Cbar^b$ is non-empty by Remark~\ref{rmk nonemptiness of central leaves}. We denote by $\Cbar^{b'} \subset\Ssb'$ the central leaf in the Siegel moduli space. Let $\bar\Ig^{b'}$ denote the special fibre of the Igusa variety over $\Cbar^{b'}$, i.e. the scheme parametrising isomorphisms $(\XX,\lambda_\XX) \isom (\Ascr^\univ[p^\infty],\lambda^\univ)$. We define
 \[
  \bar\Ig^{b,\diamond} \coloneqq \bar\Ig^{b'} \times_{(\Ssb')^{\rm perf}} \Ssb^{\rm perf}.
 \]

 Note that $\bar\Ig^{b'}$ is a perfect scheme by \cite[Prop.~4.3.8]{CaraianiScholze:ShVar}, hence $\bar\Ig^{b,\diamond}$ is also perfect. Indeed since the absolute Frobenius defines compatible isomorphisms on the base and both factors of the fibre product, it also defines an isomorphism on the fibre product itself.
Therefore, we have unique flat lifts $\Igfr^{b'}$ and $\Igfr^{b,\diamond}$ over $\Spf O_{\breve{E}}$, which are given by
 \[
  \Igfr^{b'}(R) = \bar\Ig^{b'}(R/\varpi) \text{ and }   \Igfr^{b,\diamond} (R) = \bar\Ig^{b,\diamond}(R/\varpi)
 \]
 for any $O_{\breve{E}}$-algebra $R$ with $\varpi$ nilpotent in $R$.
 By construction, we have an isomorphism $j: \XX_b \times \bar\Ig^{b,\diamond} \isom \Ascr_{\Gsf,\Cbar^b}[p^\infty] \times_{\Cbar^b} \bar\Ig^{b,\diamond}$.

 \begin{deflemsub}
  The locus of geometric points in $\bar\Ig^{b}$ where $j$ respects the crystalline Tate-tensors is a closed union of connected components. We denote by $\bar\Ig^{b,\diamond} \subset \bar\Ig^{b,\diamond}$ and $\Igfr^b \subset \Igfr^{b, \diamond}$ the corresponding (formal) subschemes.
 \end{deflemsub}
 \begin{proof}
  This follows from the existence of global tensors and Proposition~\ref{prop homomorphisms of F-crystals} by an identical argument as in the proof of Lemma~\ref{lemma clopen}.
 \end{proof}
Recall that $\QIsog_{G'}(\XX)$ naturally acts on $\bar\Ig^{b'}$ by \cite[Corollary~4.3.5]{CaraianiScholze:ShVar}, and clearly this action extends to the flat lift $\Igfr^{b'}$.
 \begin{propsub} \label{prop Igusa variety}
  $\bar\Ig^b \to \bar\Ig^{b'}$ is a closed immersion with $\QIsog_G(\XX)_{\FFbar_p}$-stable image. On $\FFbar_p$-points this action is given as follows. For any $g \in J_b(\F)$, $\tilde{z}=(z;j) \in \bar\Ig_b(\FFbar_p)$ the $\QIsog_G(\XX)$ action is given by
  \[
   g.\ztilde = (\Theta_{\tilde{z}}(g^{-1}.z);(g^{-1})_\ast j).
  \]
  In particular, $\Igfr^b \to \Igfr^{b'}$ is a closed immersion.
 \end{propsub}
 \begin{proof}
  The proof is identical to the proof of \cite[Prop.~4.10]{Hamacher:ShVarProdStr}, which gives the proof in the hyperspecial case. We give a sketch of how the arguments apply for the reader's convenience.
  
  To see injectivity, assume that $\ztilde_1 = (z_1,j_1), \ztilde_2 = (z_2,j_2) \in \bar\Ig^b(\k)$ have the same image in $\Ig^{b'}(\k)$. Note that since $\bar\Ig^b \subset \Ig^{b'} \times_{\Sscr'} \Sscr$, it suffices to show that $z_1 = z_2$. This follows from the uniqueness of the lift in Axion~\ref{axiom RZ}, as it implies $\Theta_{\ztilde_1} = \Theta_{\ztilde_2}$ and in particular $z_1 = \Theta_{\ztilde_1}(\id_\XX) = \Theta_{\ztilde_2}(\id_\XX) = z_2$. Thus $\bar\Ig^b \to \bar\Ig^{b'}$ is universally injective. As it is finite, it is universally closed; thus a universal homeomorphism onto its (closed) image. It now follows from \cite[Lemma~3.8]{BhattScholze:WittVectorAffGr} that $\bar\Ig^b \to \bar\Ig^{b'}$ is a closed immersion. Now the group theoretic description of the $J_b(\QQ_p)$-action follows from the analogue formula for the action of $J_{b'}(\QQ_p)$ on $\bar\Ig^{b'}$ given by the moduli description. 
 \end{proof}

 \begin{corsub} \label{cor:IgusaGroupAction}
  The canonical morphism $\Igfr^{b} \to \Igfr^{b'}$ is a closed immersion with $\QIsog_G(\XX)$-stable image.
 \end{corsub}
 \begin{proof}
  By the universal property of the flat lift of perfect $\FFbar_p$-schemes to $\W$, it suffices to check on special fibre. Hence the claim follows by the above proposition.
 \end{proof}

 We equip $\Igfr^b$ with the $\QIsog_G(\XX)$-action defined by above corollary. We also define a twisted Weil group action given by the Weil descent datum obtained by the Frobenius lift.

 If we wish to emphasize the level away from $p$, we will write $\Igfr_{\Ksf^p}^{b}$ and $\bar\Ig_{\Ksf^p}^b$ for the respective Igusa varieties over $\Sscr_{\Ksf^p\Ksf_p}$. The construction of Igusa varieties is compatible with the $\Gsf(\AA_f^p)$-action on $\{\Sscr_{\Ksf^p\Ksf_p}\}_{\Ksf^p \subset G(\AA_f^p)}$ in the following sense.

 \begin{lemsub} \label{lemma Hecke on Igusa tower}
  Let $g \in \Gsf(\AA_f^p)$ and $\Ksf_1^p \subset \Gsf(\AA_f^p)$ small enough compact open such that $g^{-1}\Ksf_1^p g \subset \Ksf^p$. Then the morphism $g: \Sscr_{\Ksf_1^p\Ksf_p} \to \Sscr_{\Ksf^p\Ksf_p}$ extends canonically to a morphism $\bar\Ig_{\Ksf_1^p}^b \to \bar\Ig_{\Ksf^p}^b$ and the diagram
  \begin{center}
   \begin{tikzcd}
    \bar\Ig^b_{\Ksf_1^p} \arrow{r}{g} \arrow{d} & \bar\Ig^b_{\Ksf^p} \arrow{d} \\
    \Sscr_{\Ksf_1^p\Ksf_p} \arrow{r}{g} & \Sscr_{\Ksf^p\Ksf_p}
   \end{tikzcd}
  \end{center}
 is cartesian. In particular, $\Gsf(\AA_f^p)$ acts on the tower $\{\Igfr_{\Ksf^p}^b\}_{\Ksf^p \subset \Gsf(\AA_f^p)}$
 \end{lemsub}
 \begin{proof}
  For $\Gsf=\Gsf'$, this follows from the moduli description. Thus the assertion of the lemma holds if we replace $\bar\Ig^b$ by $\bar\Ig^{b, \diamond}$. Now $\bar\Ig^b \subset \bar\Ig^{b,\diamond}$ is cut out by conditions on the crystal, which is preserved under the $\Gsf(\AA_f^p)$-action. Thus the claim follows.
 \end{proof}

 \subsection{The product structure of Newton strata} \label{ssect product structure}

 Recall that in \cite[\S~4]{CaraianiScholze:ShVar} Caraiani and Scholze constructed an isomorphism
 \begin{equation} \label{eq moduli structure}
  \Igfr^{b'} \times \Mfr^{b'} \isom \Xfr^{b'}
 \end{equation}
 where $\Xfr^{b'}$ is defined by
 \begin{multline*}
  \Xfr^{b'}(R) \coloneqq \{(\Ascr,\lambda,\eta; \psi) \mid (\Ascr,\lambda,\eta) \in \Sscr'(R), \\ \psi\colon (\XX,\lambda_\XX) \otimes R/\varpi \to (\Ascr[p^\infty],\lambda) \otimes R/\varpi \textnormal{ quasi-isogeny} \}.
 \end{multline*}
 for every $O_E$-algebra $R$ with $\varpi \in R$ nilpotent. More precisely, the isomorphism maps a tuple $(\Ascr,j;\Xscr,\varphi)$ to $(\Ascr',\lambda',\eta',\psi)$ constructed as follows. Let us first assume that $\varphi$ is an isogeny. Then we define
 \[
  \Ascr' \otimes R/\varpi \coloneqq \bigslant{(\Ascr \otimes R/\varpi)}{j(\ker\varphi)}
 \]
  with the induced polarisation $\bar\lambda'$ and level structure $\eta'$. Then $j$ identifies $(\Ascr' \otimes R/\varpi)[p^\infty]$ with $\Xscr \otimes R/\varpi$; thus the Serre-Tate theorem yields a deformation $(\Ascr',\lambda')$. Since $\Ascr$ and $\Ascr'$ have isomorphic $\AA_f^p$-Tate modules, a level-$\Ksf^p$-structure on $\Ascr' \otimes R/\varpi$ is the same as a level-$\Ksf^p$-structure on $\Ascr'$. We define $\psi$ as the composition
 \begin{center}
  \begin{tikzcd}
   \XX \otimes R/\varpi \arrow{r}{j} & \Ascr[p^\infty] \otimes R/\varpi \arrow[two heads]{r}{\varphi} & \Ascr'[p^\infty] \otimes R/\varpi.
  \end{tikzcd}
 \end{center}
Observe that we have $(\Ascr,j;\Xscr,p^n\varphi)\mapsto (\Ascr', \lambda',\eta',p^n\psi)$ for any $n\geqslant0$, so this construction naturally extends to the case when $\varphi$ is not necessarily an isogeny but a quasi-isogeny.
 We denote by $\pi_\infty'\colon \Xfr^{b'} \to \Sscr'$ the canonical projection.

 \begin{rmksub}
  This construction is slightly different from the one in \cite{CaraianiScholze:ShVar}, where they fix a lift of $\XX$ to characteristic $0$. This is important for our purposes since we may not be able to lift $\XX$  together with its crystalline Tate tensors to $O_{\Ebreve}$ as $\Def_G(\XX)$ may not be smooth. However, it is easily seen that the isomorphism is the same as in \cite{CaraianiScholze:ShVar} due to the rigidity of quasi-isogenies.
 \end{rmksub}

 Let $\Xfr^b \subset \Xfr^{b'}$ denote the image of the closed formal subscheme $\Mfr^b \times \Igfr^b \subset \Mfr^{b'} \times \Igfr^{b'}$ with respect to the above isomorphism. Note that the canonical $\QIsog_{G'}(\XX)$ action on $\Xfr^{b'}$ restricts to a $\QIsog_{G}(\XX)$-action on $\Xfr^b$ by Lemma~\ref{lemma J_b-action} and Proposition~\ref{prop Igusa variety}.

 Since our schemes are not necessarily smooth, crystalline structures are not so well behaved and we lack such a moduli interpretation for $\Xfr^b$. However, we will construct a lift $\pi_\infty\colon \Xfr^b \to \Sscr$, which behaves as one would expect from a moduli space parametrising quasi-isogenies $\XX \to \Ascr_\Gsf[p^\infty]$ respecting additional structures.

 In order to lift $\pi_\infty'$ to a morphism $\Xfr^b \to \Sscr$, we first study its restriction to a formal neighbourhood. Recall that the formal neighbourhood $\Xfr_x^\wedge$ of a formal scheme $\Xfr$ at a point $x \in \Xfr$ is the set-valued functor on the category  of zero-dimensional local rings given by
 \[
  \Xfr_x^\wedge(R) = \{\varphi \in \Xfr(R) \mid \image \varphi^{\rm top} = \{x\}\}.
 \]
 For our purposes it suffices to consider the case where the maximal ideal $\mfr_{\Xfr,x}$ of the stalk $\Ocal_{\Xfr,x}$ contains a finitely generated ideal $I \subset \mfr_{\Xfr,x}$ such that $\sqrt{I} = \mfr_{\Xfr,x}$. One easily checks that in this case $\Xfr_x^\wedge$ is pro-represented by the $I$-adic completion $\hat\Ocal_{\Xfr,x}$ of $\Ocal_{\Xfr,x}$. Moreover, for $\Xfr = \Xfr^b$ the following holds.
 
 \begin{lemsub} \label{lem formal nbhd of X^b}
  Let $x = (y,\ztilde) \in \Xfr^b(\k)$. Then there exists a finitely generalted ideal $I \subset \Ocal_{\Xfr^b,x}$ such that $\sqrt{I} = \mfr_{\Xfr^b,x}$. Moreover, for any $n$ the completion $\Ocal_{\Xfr^b,x}/p^n \to \hat\Ocal_{\Xfr^b,x}/p^n$ is faithfully flat.
 \end{lemsub}
 \begin{proof}
  Obviously it suffices to prove the first assertion modulo $p$. Since $\Ocal_{\Mfr^b_{\k,y}}\otimes \Ocal_{\bar\Ig^b,\ztilde}$ is dense in $\Ocal_{\Xfr^\wedge_x}$ , it suffices to show the first assertion for $\Mfr^b_{\k}$ and $\bar\Ig^b$. Since $\Mfr^b$ is locally noetherian, the $\mfr_{\Mfr^b_{\k},y}$ is already finitely generated and the statement trivial. Now denote by $z$ the image of $\ztilde$ in $\Cbar^b$. For $\Ocal_{\bar\Ig^b,\ztilde}$, we claim that the ideal $I \coloneqq \mfr_{\Cbar^b,z}$ satisfies the assumptions. First note that it is finitely generated since $\Cbar^b$ is noetherian. In order to see that $\sqrt{I} = \mfr_{\bar\Ig^b,\ztilde}$, we decompose the canonical morphism $\Ocal_{\Cbar^b,z} \to \Ocal_{\bar\Ig^b,\ztilde}$ as
  \begin{equation} \label{eq local Igusa}
   \Ocal_{\Cbar^b,z} \to \Ocal_{\Cbar^b,z}^\perf \to \Ocal_{\bar\Ig^b,\ztilde}.
  \end{equation}
  Now
  \[
   \mfr_{\bar\Ig^b,\ztilde} \supseteq \sqrt{I} \supseteq \sqrt{\mfr_{\Cbar^b,z}\cdot\Ocal_{\Cbar^b,z}^\perf} \cdot \Ocal_{\bar\Ig^b,\ztilde} = \mfr_{\Cbar^\perf,z} \cdot \Ocal_{\bar\Ig^b,\ztilde} = \mfr_{\bar\Ig^b,\ztilde},
  \]
  where the last equality follows from the fact that the second morphism of (\ref{eq local Igusa}) is ind-\'etale. This finishes the proof of the first assertion.
  
  For the second assertion we may assume that $n=1$ by the local flatness criterion (see for example \cite[Thm.~22.3]{matsumura:crt}). Since the completion of a coherent ring with respect to a finitely generated ideal is flat by \cite[Thm.~3.58]{Hassan:CoherentRings}, it suffices to show that $\Ocal_{\Xfr^b_{\k},x}$ is coherent. Moreover, as
  $
   \Ocal_{\Xfr^b_{\k},x}$ and $ \Ocal_{\bar\Ig^b,\ztilde} \hat\otimes \Ocal_{\Mfr^b_{\k},y}
  $
  define the same completion
  and the completion of coherent ring with respect to a finitely generated ideal is still coherent (e.g. \cite[Thm.~2.4.1]{Glaz:CoherentRings} combined with \cite[Prop.~3.59]{Hassan:CoherentRings}), it suffices to show that $\Ocal_{\bar\Ig^b,\ztilde} \hat\otimes \Ocal_{\Mfr^b_{\k},y}$ is coherent.
  
   First note that by (\ref{eq local Igusa}) the ring $\Ocal_{\bar\Ig,\ztilde}$ can be written as direct limit $\varinjlim R_i$ of regular local rings where the transition morphisms are either \'etale or the absolute Frobenius; in particular they are flat. Thus
  \[
   \Ocal_{\bar\Ig^b,\ztilde} \otimes \Ocal_{\Mfr^b_{\k},y} \cong \varinjlim R_i \otimes  \Ocal_{\Mfr^b_{\k},y}
  \]
  can be written as a direct limit of noetherian rings with flat transition morphisms and is thus coherent (see for example\cite[Th.~2.3.3]{Glaz:CoherentRings}).
 \end{proof}

 Now fix a point $x = (y,\ztilde) \in \Xfr^b(\k)$ and let $w' \coloneqq \pi_\infty'(x) = \Theta'_{\ztilde}(y)$ and $w \coloneqq \Theta_{\ztilde}(y)$. Then $\Theta_{\ztilde}'$ and $\Theta_{\ztilde}$ induce compatible isomorphisms of formal neighbourhoods $\Mfr_y^{b',\wedge} \isom \Sscr_{w'}'^\wedge$ and $\Mfr_y^{b,\wedge} \isom \Sscr_{w}^\wedge$. Writing $\ztilde = (z,j_0)$, we moreover obtain
 \begin{align*}
  \Ig^{b',\wedge}_{\ztilde}(S,\mfr_S) &= \left\{(A,\lambda,\eta;j) \left\vert
  \begin{array}{ll}
   (A,\lambda,\eta) \in \Sscr_{z}'^\wedge(S), \\
   j\colon (\XX,\lambda_\XX) \isom (A[p^\infty],\lambda) \textnormal{ with } \restr{j}{\{\mfr_S\}} = j_0
  \end{array} \right. \right\} \\
  &\isom \left\{(X,\lambda_X,\rho;j) \left\vert 
 \begin{array}{ll}
   (X,\lambda_X,\rho) \in \Mfr^{b'}(S) \textnormal{ such that } \restr{\rho}{\mfr_S} = \id_{\XX} \\ 
   j\colon (\XX,\lambda_\XX) \isom (X,\lambda_X) \textnormal{ such that } \restr{j}{\mfr_S} = \restr{\rho}{\mfr_S}
  \end{array} \right. \right\} \\ 
  &\isom \QIsog_{G'}(\XX)^\circ_{\k}(S),                                        
 \end{align*}
 where the first isomorphism is induced by $\Theta_{\ztilde}'$ and the second given by $(X,\lambda_X,\rho;j) \mapsto \rho \circ j^{-1}$. Using the uniqueness of flat lifts we conclude $\bar\Ig^{b',\wedge}_{\ztilde} \cong \QIsog_{G'}(\XX)_{\k}^\circ$ and using the uniqueness of flat lifts this extends to $\Igfr^{b',\wedge}_{\ztilde} \cong \QIsog_{G'}(\XX)$. By \cite[Cor.~5.1.3]{Kim:ProductStructure} this isomorphism restricts to $\Igfr^{b,\wedge}_{\ztilde} \cong \QIsog_{G}(\XX)$. Tracing through the identifications above, we see that the restriction to formal neighbourhoods $\pi_{\infty,x}'^\wedge\colon \Xfr_x^{b',\wedge} \to \Sscr_{w'}'^\wedge$ becomes the canonical group action $\QIsog_{G'}(\XX)^\circ \times \Mfr_y^{b',\wedge} \to \Mfr_y^{b',\wedge}$. By \cite[Thm.~4.3.1]{Kim:ProductStructure} this restricts to $\QIsog_{G}(\XX)^\circ \times \Mfr_y^{b,\wedge} \to \Mfr_y^{b,\wedge}$, i.e. $\pi_{\infty,x}'^\wedge$ restricts to $\pi_{\infty,x}^\wedge\colon \Xfr^{b,\wedge}_x \to \Sscr_{w}^\wedge$. We may reformulate our observation as follows.

 \begin{lemsub} \label{lem lift on formal fibres}
  Let $x = (y,\ztilde) \in \Xfr^b(\k)$. Then the restriction $\restr{\pi_\infty'}{\Xfr^{b,\wedge}_x}$ has a unique lift $\pi_{\infty,x}^\wedge\colon \Xfr_x^{b,\wedge} \to \Sscr$. Moreover, $\pi_{\infty,x}^\wedge(x) = \Theta_{\ztilde}(y)$ and the isomorphism in Proposition~\ref{prop normalisation on formal neighbourhoods} identifies the induced map $\Xfr^{b,\wedge}_x \to \Sscr^\wedge_{\pi_{\infty,x}^\wedge(x)}$ with the $\QIsog_G(\XX)^\circ$ action on $\Def_G(\Ascr_{\Gsf,\pi_{\infty,x}^\wedge(x)}[p^\infty])$ constructed in \cite[\S~4.3]{Kim:ProductStructure}.
 \end{lemsub}
 \begin{proof}
  The existence of $\pi_{\infty,x}^\wedge$ and its properties were proven above. Uniqueness follows from Proposition~\ref{prop normalisation on formal neighbourhoods} since it states that the formal fibre of $\iota$ corresponds to the decomposition of $\Sscr_{\pi_{\infty,x}^\wedge}^\wedge$ into irreducible components and $\pi_\infty'$ maps $\Xfr^{b,\wedge}_x$ surjectively onto one of them.
 \end{proof} 
 
 \begin{thmsub} \label{thm product structure}
  The restriction $\restr{\pi_\infty'}{\Xfr^b}$ has a unique lift $\pi_\infty: \Xfr^b \to \Sscr$. Moreover, $\pi_\infty$ has the following properties
  \begin{subenv}
   \item The perfection of the underlying reduced subscheme $\pi_{\infty,\k}^{\perf}\colon \Xfr^{b, {\rm \red}, \perf} \to \Ssb^\perf$ represents the moduli problem
   \begin{align*}
    &({\rm PerfAlg})_{\k} \to \Ssb \\
    & R \mapsto \{ (P,\psi)\mid P \in \Sscr(R), \psi\colon (\XX,(t_{\XX,\alpha}))_R \to (\Ascr_{\Gsf,P}[p^\infty],(t_{\alpha,P})) \textnormal{ quasi-isogeny} \}.
   \end{align*}
   In particular, $\pi_\infty$ factorises through the completion $\widehat\Sscr^b$ of $\bar\Sscr^b$ in $\Sscr$.
   \item $\pi_\infty: \Xfr^b \to \widehat\Sscr^b$ is a quasi torsor under  the $\QIsog_G(\XX)$-action, that is the morphism
   \[
    \QIsog_G(\XX) \times \Xfr^b \to \Xfr^b \times_\Sscr \Xfr^b, (g,x) \mapsto (x,g.x)
   \]
   is an isomorphism.
   
   \item We equip $\Xfr^b$ with a Weil descent datum $\alpha_E$ by restricting  the Weil descent datum on $\Xfr^{b'}$ given by $(\Ascr,\lambda,\eta,\psi) \mapsto (\Ascr,\lambda,\eta,\psi \circ F_q^{-1})$. Then the isomorphism $\Xfr^b \cong \Igfr^b \times \Mfr^b$ is $W_E$-equivariant.
   \item The Hecke action of $\Gsf(\AA_f^p)$ on $\{\Sscr_{\Ksf^p\Ksf_p}\}_{\Ksf^p \subset \Gsf(\AA_f^p)}$ lifts canonically to $\{\Xfr^b_{\Ksf^p}\}_{\Ksf^p \subset \Gsf(\AA_f^p)}$. More precisely, let $g \in \Gsf(\AA_f^p)$ and $\Ksf_1^p \subset \Gsf(\AA_f^p)$ small enough compact open such that $g^{-1}\Ksf_1^p g \subset \Ksf^p$. Then the diagram
   \begin{center}
    \begin{tikzcd}
     \Xfr^b_{\Ksf_1^p} \arrow{r}{g} \arrow{d} & \Xfr^b_{\Ksf^p} \arrow{d} \\
     \Sscr_{\Ksf_1^p\Ksf_p} \arrow{r}{g} & \Sscr_{\Ksf^p\Ksf_p}
    \end{tikzcd}
   \end{center}
   is cartesian. Moreover, the morphism in the first line corresponds to $g \times \id$ with respect to the isomorphism $\Xfr^b\cong \Igfr^b \times \Mfr^b$.
  \end{subenv}
 \end{thmsub}
 \begin{proof}
 
  The uniqueness of the lift readily follows from the uniqueness of the lift on formal fibres by the lemma above. Now consider the cartesian diagram
  \begin{center}
   \begin{tikzcd}
    \Xfr^{b,\diamond} \coloneqq \Xfr^b \times_{\Sscr'} \Sscr  \arrow{r}{\tilde\pi'_\infty} \arrow{d}{\tilde\iota} & \Sscr \arrow{d}{\iota} \\
    \Xfr^b \arrow{r}{\pi_\infty'}& \Sscr'
   \end{tikzcd}
  \end{center}
  Then a lift of $\pi_\infty'$ corresponds to a section $s$ of $\tilde\iota$. Note that $s^{\rm set}(y,\ztilde) \coloneqq (y,\ztilde;\Theta_{\ztilde}(y))$ gives the set-theoretical section on $\k$-valued points corresponding to the lift $(y,\ztilde) \mapsto \Theta_{\ztilde}(y)$. We claim that the image of $s^{\rm set}$ in $\Xfr^{b,\diamond}(\k)$ is open, thus corresponds to an open formal subscheme $U \subset \Xfr^{b,\diamond}$ (since $\Xfr^b$ is Jacobson), and we claim moreover that $\restr{\tilde\iota}{U}\colon U \to \Xfr^{b,\diamond}$ is an isomorphism. Taking the inverse of $\restr{\tilde\iota}{U}$ would give us the wanted section $s$.
  
  For $x \in \Xfr^b(\k)$ consider the morphism ${\tilde\iota}_x\colon \Ocal_{\Xfr^b,x} \to \Ocal_{\Xfr^{b,\diamond}s^{\rm set}(x)}$. By our considerations of formal neighbourhoods above, it is an isomorphism after extending scalars to $\hat\Ocal_{\Xfr^b,x}$. Thus $\tilde\iota_x$ is an isomorphism by faithfully flat descent and Lemma~\ref{lem formal nbhd of X^b}. Since $\tilde\iota$ is of finite presentation, it follows that there exist open neighbourhoods $U_x \subset \Xfr^{b, \diamond}$ of $s(x)$ and $V_x \subset \Xfr^b$ of $x$ such that $\restr{\tilde\iota}{U_x}\colon U_x \to V_x$ is an isomorphism. Now let $U \coloneqq \bigcup_{x \in \image s^{\rm set}} U_x$. Since by Lemma~\ref{lem lift on formal fibres} the morphism $\tilde\iota$ can only be a local isomorphism at $\k$-valued points in $\image s^{\rm top}$, we have indeed $U(\k) = \image s^{\rm top}$. Moreover, the restriction $\restr{\tilde\iota}{U}$ is a local isomorphism of Jacobson formal schemes which is bijective on closed points and thus an isomorphism.

  To prove (1), consider the morphism $s\colon \Xfr^b \mono \Xfr^{b,\diamond}$ constructed above. Since $\Xfr^{b,\diamond}$ solves the moduli problem
 \[
  R \mapsto \{ (P,\psi) \mid P \in \Sscr(R), \psi\colon (\XX,\lambda) \otimes R/p \to (\Ascr_\Gsf[p^\infty],\lambda_\Gsf) \otimes R/p \textnormal{ quasi-isogeny} \},
 \]
  we also have $\Xfr_0^b \subset \Xfr^{b,\diamond,\perf}_{\k}$, where $\Xfr_0^b$ denotes the functor defined in (1).  Note that $\Xfr_0^b$ is represented by a closed union of connected components of the perfection of $\Xfr^{b,\diamond}_{k_F}$ by the same argument as in Lemma~\ref{lemma clopen}. Since the perfection of $\Xfr^{b,\diamond}_{\k}$ is Jacobson, it suffices to check equality of the reduced subschemes $\Xfr_0^b$ and $\Xfr^{b,\perf}_{\k}$ on $\k$-valued points. 
 Since $\pi_\infty(y,\ztilde) = \Theta_{\ztilde}(y)$ for any $(y,\ztilde) \in \Xfr^b(\k)$ by construction, we have $\Xfr^b(\k) \subset \Xfr^b_0(\k)$. On the other hand, if $(P,\psi) \in \Xfr^b_0(\k)$ let $z = \Theta_{(P,\psi)}(\id)$ and let $j$ denote the composition of
  \[
   \XX \stackrel{\psi}{\lto} \Ascr_{\Gsf,P}[p^\infty] \to \Ascr_{\Gsf,P}[p^\infty]/\ker(\psi^{-1}) = \Ascr_{\Gsf,\Theta_{(P,\psi)}(\id)}[p^\infty].
  \]
  Then $(P,\psi)$ is the image of $(z,j;\psi)$, thus $(P,\psi) \in \Xfr^b(\k)$. 

  For $G = G'$ the second claim follows from the moduli description (cf.\ Remark before \cite[Prop.~4.3.13]{CaraianiScholze:ShVar}). Since $\pi_\infty$ is $\QIsog_G(\XX)$-equivariant by uniqueness of the lift, we obtain a commutative diagram
  \begin{center}
   \begin{tikzcd}
    \QIsog_{G'}(\XX) \times \Xfr^{b'} \arrow{r}{\sim}[swap]{pr_2 \times act} & \Xfr^{b'} \times_{\Sscr'} \Xfr^{b'} \\
    \QIsog_{G}(\XX) \times \Xfr^{b} \arrow{r}[swap]{pr_2 \times act} \arrow[hook]{u} & \Xfr^{b}  \times_{\Sscr} \Xfr^{b} \arrow[hook]{u},
   \end{tikzcd}
  \end{center}
  where the vertical arrows, and thus also the bottom line arrow, are closed immersions. In particular it suffices to show that the bottom morphism is surjective and induces an isomorphism on formal neighbourhoods. The surjectivity follows from (1). By homogeneity, we may assume that we are considering the identity component of $\QIsog_G(\XX)$ when checking formal neighbourhoods, i.e.\ we have to show that
  \[
     \QIsog_{G}(\XX)^\circ \times \Xfr^{b \wedge}_z \stackrel{pr_2 \times act}{\lto} \Xfr^{b\wedge}_z  \times_{\Sscr} \Xfr^{b\wedge}_z
  \]
  is an isomorphism for any $z \in \Xfr^b(\k)$. Now decompose $\Xfr_z^{b\wedge} \cong \Igfr_x^{b\wedge} \times \Mfr_y^{b\wedge}$ and identify $\Igfr_x^{b\wedge}$ with $\QIsog_G(\XX)^\circ$ as in the construction before Lemma~\ref{lem lift on formal fibres}. Then the above morphism becomes
  \begin{align*}
   \QIsog_G(\XX)^\circ \times \QIsog_G(\XX)^\circ \times \Mfr_y^{b\wedge} &\to (\QIsog_G(\XX)^\circ \times \Mfr_y^{b\wedge}) \times_{\Mfr_y^{b\wedge}} (\QIsog_G(\XX)^\circ \times \Mfr_y^{b\wedge}) \\
   (g_1,g_2,\varphi) &\mapsto ((g_2,\varphi);(g_2\circ g_1^{-1}, g_1.\varphi))
  \end{align*}
  One easily deduces from the cocycle condition of the group action that this is an isomorphism.

  It suffices to show (3) for $\Gsf=\Gsf'$ and the Frobenius $\sigma \in W_{\QQ_p}$, since the inertia group acts by definition trivially. Fix a point $(\Ascr',\lambda',\eta';\psi) \in \Xfr^{b'}(S)$, which corresponds to a point $((\Ascr,\lambda,\eta;j),(X,\varphi)) \in \Igfr^{b'}(S) \times \Mfr^{b'}(S)$.  Consider the commutative diagram
  \begin{center}
   \begin{tikzcd}
    (\sigma^\ast\XX)_S \arrow{r}{\sim}[swap]{j^{(q)}} \arrow{d}{F^{-1}} & \sigma^\ast\Ascr[p^\infty] \arrow{d}{F^{-1}} \\
    \XX_S \arrow{r}{\sim}[swap]{j} \arrow{d}{\varphi} \arrow{rd}[swap]{\psi}& \Ascr[p^\infty] \arrow{d}{j_\ast\varphi} \\
    X \arrow{r}{\sim} & \Ascr'[p^{\infty}]
   \end{tikzcd}
  \end{center}
  Since the top row corresponds to $\alpha_\F(\Ascr,\lambda,\eta;j)$ and the left column to $\alpha_\F(X,\varphi)$, the image of is the image of $\alpha_\F((\Ascr,\lambda,\eta;j),(X,\varphi))$ is indeed $(\Ascr,\lambda,\eta,\psi \circ F^{-1})$.

  One easily checks that (4) holds in the case $\Gsf=\Gsf'$, using the moduli description. Since $g\times\id$ maps $\Igfr_{\Ksf_1^p}^b \times \Mfr^b$ onto $\Igfr_{\Ksf^p}^b \times \Mfr^b$ by Lemma~\ref{lemma Hecke on Igusa tower}, the restriction of $g$ maps $\Xfr^b_{\Ksf_1^p}$ onto $\Xfr^b_{\Ksf^p}$ with the wanted properties.
 \end{proof}

 \begin{corsub} \label{cor product structure}
  The perfection $\pi_{\infty,\k}^{\perf}$ of $\pi_{\infty,\k}$ is a $J_b(\QQ_p)$-torsor over $\Ssb^{b, \perf}$ for the pro-\'etale topology.
 \end{corsub}
 \begin{proof}
  Part (1) of above theorem gives a moduli description of $\pi_{\infty,\k}^\perf$, which is a $J_b(\QQ_p)$-torsor by \cite[Prop.~4.3.13]{CaraianiScholze:ShVar} (or rather its proof).
 \end{proof}

 \begin{corsub} \label{cor product structure kuenneth}
   We consider the  schemes $\Mfr^{b, {\rm red}},\bar\Ig^b$ and $\Ssb^b$ with the $W_E$-action defined in Lemma~\ref{lem:RZWeilDescentDatum}, after Corollary~\ref{cor:IgusaGroupAction} and by the composition of the canonical projection $W_E \to \Gal(\FFbar_p/\kappa_E)$ and the Galois action on $\Ssb$, respectively and the constant sheaves $\ZZ/l^r\ZZ$ on them. Denote by $\Hcal_r \coloneqq C_c^\infty(J_b(\QQ_p),\ZZ/l^r\ZZ)$. With the definition of $R\Gamma_c$ from Appendix~\ref{appendix Kuenneth}, we get an isomorphism in the category of complexes of $W_E$-representations
  \[
    \Rrm\Gamma_c(\Ssb^b, \ZZ/l^r\ZZ) = \Rrm\Gamma_c(\bar\Ig^b,\ZZ/l^r\ZZ) \otimes^L_{\Hcal_r} \Rrm\Gamma_c(\Mfr^{b,\red},\ZZ/l^r\ZZ),
  \]
  where $W_E$ acts trivially on $R\Gamma_c(\bar\Ig^b,\ZZ/l^r)$.
 \end{corsub}
 \begin{proof}
  By Corollary~\ref{cor product structure} the prerequisites of Corollary~\ref{cor-kuenneth} are met; thus it is only left to show that $W_E$ acts trivially on $R\Gamma_c(\bar\Ig^b,\ZZ/l^r)$. This follows since on the underlying topological space the $W_E$-action on $\bar\Ig^b$ is (universally) trivial.
 \end{proof}

\section{Cohomology of automorphic sheaves} \label{sect cohomology}
In this section, we prove our main result of this paper: Mantovan's formula (Corollary~\ref{cor:MantShVar}).

 \subsection{Auxiliary integral models with level structure}\label{ssect Drinfeld}
 In the unramified PEL case, Mantovan \cite[\S6]{Mantovan:Foliation} constructed auxiliary integral models of Shimura varieties and Rapoport-Zink spaces with deeper level structure at $p$ via Drinfeld level structure. These integral models do not possess nice properties (for example, the Hecke correspondences do not extend to these integral models in general), but nonetheless one gets just enough to obtain the action of Hecke correspondence on the nearby cycle cohomology.

 In order to define Drinfeld level structure, denote by $\Ksf_p(m)$ the kernel of the canonical projection $\Gscr_x(\ZZ_p) \epi \Gscr_x(\ZZ/p^m)$. 

If the prime-to-$p$ level $\Ksf^p\subset\Gsf(\AA_f^p)$ is not relevant, then we write
\[
\Sh_m\coloneqq \Sh_{\Ksf_{p}(m)\Ksf^p}(\Gsf,\Xsf) \text{ and } \Sh\coloneqq \Sh_{\Ksf_{p}\Ksf^p}(\Gsf,\Xsf)
\]
for simplicity. We can understand $\Sh_m$ as the \'etale cover classifying  $\Ksf_p(m)$-level structures over $\Sh$.
More precisely, we have
\begin{multline*}
\Sh_m(R) = \big\{
(P,\eta_m)|\, P\in\Sh(R), \\ \eta_m\colon\Lambda/p^m \xrightarrow\sim \Asf_{\Gsf,P}[p^m] = T_p(\Asf_{\Gsf,P})/p^m \text{ sending } (s_\alpha)\text{ to }(t_{\alpha,\et})
\big\},
\end{multline*}
for any $\Esf$-algebra $R$.
Here, we call the tensor-preserving isomorphism $\eta_m\colon\Lambda/p^m \xrightarrow\sim \Asf_{\Gsf,P}[p^m]$ a \emph{$\Ksf_p(m)$-level structure}.

Denote by $\Mcal^b_m\coloneqq \Mcal^b_{\Ksf_{p}(m)}$ the finite \'etale cover of $\Mcal^b = (\Mfr^b)^\ad_{\breve E}$ defined by adding the $\Ksf_p(m)$-level structure. 
To explain, let $\Xscr_G$ denote the universal {\BT} over $\Mfr^b$, and we view its integral Tate module $T_p(\Xscr_G)$ as a $\ZZ_p$-local system over $\Mcal^b$. Then, via the Rapoport-Zink uniformisation map $\Theta_{\ztilde}:\Mfr^b\to\Sscr$ for $\ztilde\in\bar\Ig^b(\k)$, the \'etale tensors $(t_{\alpha,\et})$ on $\Sh$ induce the tensors $(t_{\alpha,\et})$ in $\Gamma(\Mcal^b,T_p(\Xscr_G)^\otimes)$. Now, the ``universal $\Ksf_p(m)$-level structure'' is the trivialisation of $(T_p(\Xscr_G)/p^m, (t_{\alpha,\et}))$ by $(\Lambda/p^m,(s_\alpha))$. (Note that $T_p(\Xscr_G)/p^m \cong \Xscr_G[p^m]^\ad_{\breve E}$ as lisse \'etale sheaves on $\Mcal^b$.)

We make analogous definitions for $\Gsf'$. Since $\Ksf_p(m) = \Ksf_p \cap \Ksf'_p(m)$ by definition, we have closed immersions 
$\Mcal^b_m \hra (\Mcal^{\prime b'}_m)_{\breve E}$ and $\Sh_m \hookrightarrow \Sh'_m$ (with the obvious notation).
\begin{defnsub}
We define the following submonoid of $G'(\Qp)$:
\[
G'(\Qp)^+ \coloneqq\{ g\in G'(\Qp) \text{ such that } g^{-1}\Lambda \subset\Lambda\},
\]
where $\Lambda = \Zp^{2n}$ is the standard representation of $G'=\GSp_{2n}$.
Similarly, we define the submonoid $G(\Qp)^+\coloneqq G'(\Qp)^+ \cap G(\Qp)$ of $G(\Qp)$.
Since $p^\ZZ$ and $G'(\Qp)^+$ generate $G'(\Qp)$ and $p^\ZZ$ is contained in $G(\Qp)$, it follows that $p^\ZZ$ and $G(\Qp)^+$ generate $G(\Qp)$.

For $g\in G'(\Qp)^+$, we define $e(g)$ to be the minimal non-negative integer such that $g\Lambda \subset p^{-e(g)}\Lambda$. By definition, we have $\Ksf'_p(m-e(g)) \supset \Ksf'_p\cap (g^{-1}\Ksf'_p(m)g)$ for any $m\geqslant e(g)$. If $g\in G(\Qp)^+$, then we have $\Ksf_p(m-e(g)) \supset \Ksf_p\cap (g^{-1}\Ksf_p(m)g)$ for any $m\geqslant e(g)$.
\end{defnsub}

Below, we deduce the following construction of integral models from the Siegel case of \cite[\S6]{Mantovan:Foliation}:
\begin{propsub}\label{prop:DrinfeldModels}
For any $g\in G(\Qp)^+$ and $m\geqslant e(g)$, we have a normal integral model $\Sscr_{m,g}$ of $\Sh_{m}$ and a normal formal model $\Mfr^b_{m,g}$ of $\Mcal^b_m$
equipped with the following morphisms:
\begin{subenv}
\item\label{prop:DrinfeldModels:StrMorph} for any $m'>m\geqslant e(g)$, a proper morphism $\Sscr_{m',g} \ra  \Sscr_{m,g}$ (respectively, $\Mfr^b_{m',g} \ra \Mfr^b_{m,g}$) that extends the natural projection on the generic fiber;
\item\label{prop:DrinfeldModels:Proj} a proper map $ \Sscr_{m,g} \ra \Sscr_{m}\coloneqq\Sscr_{m,\id_G}$ (respectively, $\Mfr^b_{m,g}\ra \Mfr^b_m=:\Mfr^b_{m,\id_G}$) inducing the identity map on the generic fiber;
\item\label{prop:DrinfeldModels:Hecke} a proper map $[g]:\Sscr_{m,g} \ra \Sscr_{m-e(g)}$ (respectively, $[g]:\Mfr^b_{m,g}\ra \Mfr^b_{m-e(g)}$) that extends the natural Hecke correspondence on the generic fibers
\begin{align*}
&{\xymatrix@1{
\Sh_{m} \ar[r]^-{[g]} & \Sh_{g^{-1}\Ksf_p(m)g} \ar@{->>}[r] & \Sh_{m-e(g)}
}}\\
\text{(respectively, }
&
{\xymatrix@1{
\Mcal^b_{m} \ar[r]^-{[g]} & \Mcal^b_{g^{-1}\Ksf_p(m)g} \ar@{->>}[r] & \Mcal^b_{m-e(g)}
}}\text{),}
\end{align*}
which are compatible under compositions in the obvious sense. (For example, we have a well-defined structure morphism $\Sscr_{m,g} \ra \Sscr$, and all the morphisms above are morphisms over $\Sscr$. And for $g,g'\in G(\Qp)^+$, we have $[g]\circ [g'] = [gg']$ whenever this expression makes sense.)
\end{subenv}
Furthermore, if we let $\Xfr^b_{m,g} \coloneqq \Xfr^b \times_{\pi_\infty,\Sscr} \Sscr_{m,g}$, then we have a canonical isomorphism
\begin{equation}\label{eqn:DrinfeldModels:ProdStr}
\Xfr^b_{m,g} \cong \Igfr^b\times \Mfr^b_{m,g}
\end{equation}
over $\Xfr^b\cong \Igfr^b\times \Mfr^b$, commuting with the Hecke actions of $\Gsf(\AA_f^p)$. And if we define  the Hecke  correspondences for $g\in G(\QQ_p)^+$ on $\{\Xfr^b_{m,g}\}$ by pulling back the maps (\ref{prop:DrinfeldModels:Proj})--(\ref{prop:DrinfeldModels:Hecke}) over $\{\Sscr_{m,g}\}$, then the isomorphism (\ref{eqn:DrinfeldModels:ProdStr}) commutes with the Hecke correspondences for any $g\in G(\QQ_p)^+$.
\end{propsub}
Note that by taking fibre of the isomorphism (\ref{eqn:DrinfeldModels:ProdStr}) at $\ztilde\in\bar\Ig^b(\k)$ gives an isomorphism $\Mfr^b_{m,g} \cong \Mfr^b\times_{\Theta_{\ztilde},\Sscr}\Sscr_{m,g}$, commuting with the Hecke action of $\Gsf(\AA_f^p)$ and $G(\QQ_p)^+$.
\begin{proof}
If $\Gsf=\Gsf'$, then Mantovan \cite[\S6]{Mantovan:Foliation} constructed  (not necessarily normal) integral models $\Sscr'^\Dr_{m,g'}$ and $\Mfr^{b',\Dr}_{m,g'}$ for any $g'\in G'(\Qp)^+$ and $m\geqslant e(g')$ as adding Drinfeld level structure (at $p$) to the moduli problems. These integral models come equipped with the various natural projections and the Hecke correspondence for $g'\in G'(\Qp)^+$ in the sense of (\ref{prop:DrinfeldModels:StrMorph})--(\ref{prop:DrinfeldModels:Hecke}), which are constructed via the moduli description. We show that these maps are proper by directly verifying the valuative criterion using the moduli description of $\Sscr'^\Dr_{m,g'}$ and $\Mfr^{b',\Dr}_{m,g'}$. Lastly, since the Drinfeld level structure on an abelian variety only depends on the associated {\BT}, it follows that $\Mfr^{b',\Dr}_{m,g'} \cong \Mfr^{b'}\times_{\Theta_{\ztilde'},\Sscr'}\Sscr'^{\Dr}_{m,g'}$, where $\Theta_{\ztilde'}$ is the map induced by the image $\tilde z'\in\bar\Ig^{b'}(\k)$ of $\tilde z$.

For any $m$ and $g\in G(\Qp)^+$, we define  $\Sscr_{m,g}$ to be the normalization of the Zariski closure of $\Sh_{m}$ in $\Sscr_{m,g}'^\Dr$. Since all the maps as in  (\ref{prop:DrinfeldModels:StrMorph})--(\ref{prop:DrinfeldModels:Hecke}) for $\Sscr_{m,g}'^\Dr$ restrict to dominant maps of the Zariski closures of $\Sh_{m}$ in $\Sscr_{m,g}'^\Dr$, they naturally extend to the normalizations, which clearly satisfy the properties stated in (\ref{prop:DrinfeldModels:StrMorph})--(\ref{prop:DrinfeldModels:Hecke}) for $\Sscr_{m,g}$'s.

To define $\Mfr^b_{m,g}$, let us recall the notion of normalization of formal schemes $\Mfr$ locally formally of finite type over $O_{\breve E}$ (or more generally, for \emph{excellent} formal schemes). We first note that any formally finitely generated algebra $R$ over a complete discrete valuation ring of (generic) characteristic~$0$ is excellent (\emph{cf.} \cite{Valabrega:FewThms}). Therefore, the normalization $\widetilde R$ of $R$ is finite over $R$ (hence, already complete with respect to the natural adic topology of $R$), and the normalization of $R\langle 1/f \rangle$ coincides with $\widetilde R\langle 1/f\rangle$ for any $f\in R$, where $R\langle 1/f \rangle$  and $\widetilde R\langle 1/f\rangle$ denotes the completion of the localization. Therefore, by choosing a Zariski open covering $\{\Spf R_\alpha \}_\alpha$ of $\Mfr$, we can construct the normalization $\widetilde \Mfr$ by glueing $\{\Spf \widetilde R_\alpha \}_\alpha$.

We claim that if $\Mcal\coloneqq\Mfr^\ad_{\breve E}$ is normal (for example, smooth), then $(\wt \Mfr)^\ad_{\breve E}$ is naturally isomorphic to $\Mcal$.
Indeed, we may assume $\Mfr$ is affine, say $\Mfr = \Spf R$, and it suffices to show that the normality of $\Mcal$ implies the normality of $R[\frac{1}{p}]$. Since $R[\frac{1}{p}]$ is excellent, the normality can be checked after the completion at maximal ideals. Now the claim follows since for any maximal ideal $x$ of $R[\frac{1}{p}]$, which we also view as a classical point of $\Mcal$, the completed local ring $\wh\Ocal_{\Mcal,x}$ coincides with the $x$-adic completion of $R[\frac{1}{p}]$; \emph{cf.}~\cite[Lemma~7.1.9]{dejong:crysdieubyformalrigid}.

We set $\Mfr^{b,\diamond}_{m,g}$ to be the normalization of $\Mfr^{b',\Dr}_{m,g}\times_{\Mfr^{b'}}\Mfr^b$. By the discussion above, the following finite \'etale cover of $\Mcal^b$
\[
(\Mfr^{b,\diamond}_{m,g})^\ad_{\breve E} = \Mcal^{b'}_m \times_{\Mcal^{b'}}\Mcal^b,
\]
contains $\Mcal^b_m$ as an open and closed subspace (as $\Mcal^b_m$ is a finite \'etale cover of $\Mcal^b$).
By \cite[Theorem~7.4.1]{dejong:crysdieubyformalrigid}, there exists an open and closed formal subscheme $\Mfr^b_{m,g}$ of $\Mfr^{b,\diamond}_{m,g}$ whose generic fiber is $\Mcal^b_m$.
Now, by repeating the proof of (\ref{prop:DrinfeldModels:StrMorph})--(\ref{prop:DrinfeldModels:Hecke}) for  $\Sscr_{m,g}$, we can obtain maps between $\Mfr^b_{m,g}$ that  satisfy the properties stated in (\ref{prop:DrinfeldModels:StrMorph})--(\ref{prop:DrinfeldModels:Hecke}).

It remains to show the existence of the isomorphism (\ref{eqn:DrinfeldModels:ProdStr}) and its compatibility with the Hecke action. We begin with the analogue of (\ref{eqn:DrinfeldModels:ProdStr}) for the integral models with Drinfeld level structure in the case of $\Gsf=\Gsf'$; namely,  one can  lift the natural isomorphism (\ref{eq moduli structure}) to obtain the following canonical isomorphism
\[
\Xfr^{b',\Dr}_{m,g}\coloneqq \Xfr^{b'} \times_{\pi_\infty,\Sscr'} \Sscr_{m,g}'^\Dr
\cong \Igfr^{b'}\times \Mfr^{b',\Dr}_{m,g}.
\]
where $\Sscr_{m,g}'^\Dr$ and $\Mfr^{b'}_{m,g}$ are as before. This isomorphism holds as the Drinfeld level structure defining  $\Sscr_{m,g}'^\Dr$ only depends on the associated {\BT} of the universal abelian scheme.

We let $\Sscr_{m}^\Dr$ denote the closure of $\Sh_{m}$ in $\Sscr_{m}'^\Dr\times_{\Sscr'}\Sscr$. By construction, $\Sscr_{m}$ is the normalization of $\Sscr_{m}^\Dr$ and we have a finite morphism $\Xfr^b_{m,g}\to\Xfr^{b',\Dr}_{m,g}$ induced by $\Sscr_{m,g}\to\Sscr'^\Dr_{m,g}$. Thus we obtain morphisms $\Xfr^b_{m,g}\to \Igfr^{b'}$ and $\Xfr^b_{m,g}\to\Mfr^{b',\Dr}_{m,g}$   by the natural projections.
Note that the first projection $\Xfr^b_{m,g}\to \Igfr^{b'}$ factorises as
\[
\Xfr^b_{m,g} \to \Xfr ^b \cong  \Igfr^b\times \Mfr^b \to \Igfr^b \hookrightarrow \Igfr^{b'},
\]
and the resulting map $\Xfr^b_{m,g}\to \Igfr^b$ is clearly  invariant under the Hecke action of $G(\QQ_p)^+$.

Let us now inspect the second projection $\Xfr^b_{m,g}\to\Mfr^{b',\Dr}_{m,g}$.  Let us choose  a point $\tilde z\in \Igfr^b(\k)$ and let $z\in\Ssb^b(\k)$ denote its image. As $\Igfr^b$ is an $O_{\breve E}$-lift of a perfect scheme, $\tilde z$ admits a unique lift in $\Igfr^b(O_{\breve E})$, which we also denote as $\tilde z$. Then we have
\[
\Xfr^b_{m,g} \times_{\Igfr^b,\tilde z}\Spf O_{\breve E} \cong \Mfr^b\times_{\Theta_{\ztilde},\Sscr}\Sscr_{m,g} 
\]
which follows from the identity $\Theta_{\ztilde} = \pi_\infty |_{\{\tilde z\} \times \Mfr^b}$.  We now claim that the right hand side of this isomorphism is naturally isomorphic to $\Mfr^b_{m,g}$, so the first projection induces $\Xfr^b_{m,g}\to\Mfr^b_{m,g}$. For this we first show that it is normal. Since it is an excellent formal scheme, it suffices to check the normality on the formal neighbourhood of closed points. Since $\Theta_{\ztilde}$ induces isomorphism of formal neighbourhoods of closed points so does the natural projection $\Mfr^b\times_{\Theta_{\ztilde},\Sscr}\Sscr_{m,g}\to\Sscr_{m,g}$; thus the normality follows from the normality of $\Sscr_{m,g}$. Therefore $\Mfr^b_{m,g}$ and $\Mfr^b\times_{\Theta_{\ztilde},\Sscr}\Sscr_{m,g}$ are normal formal models of $\Mcal^b_m$ that are finite over $\Mfr^b\times_{\Theta_{\ztilde},\Sscr}\Sscr^{-}_{m,g}$. Thus we obtain a natural isomorphism
\[\Mfr^b_{m,g} \riso \Mfr^b\times_{\Theta_{\ztilde},\Sscr}\Sscr_{m,g}\]
 of formal models of $\Mcal^b_m$, and the universal property of normalisation shows that this isomorphism respects with the maps (\ref{prop:DrinfeldModels:StrMorph})--(\ref{prop:DrinfeldModels:Hecke}) on both sides; i.e, this isomorphism commutes with the Hecke action of $G(\QQ_p)^+$.

Now, we have a natural map
\[
\Xfr^b_{m,g} \to  \Igfr^b\times \Mfr^b_{m,g}
\]
over the isomorphism $\Xfr^b\riso  \Igfr^b\times  \Mfr^b$. We claim that this map gives the desired isomorphism (\ref{eqn:DrinfeldModels:ProdStr}). 
Indeed, it is a finite morphism since both the source and the target are finite over $\Xfr^{b',\Dr}_{m,g}\cong \Igfr^{b'}\times\Mfr^{b',\Dr}_{m,g}$, and it is a closed immersion by Nakayama lemma since the fibre at each $\tilde z\in\Igfr^b(\k)$  is an isomorphism (\emph{cf.} \cite[Thm.~4.8]{matsumura:crt}). We conclude that the above natural map is an isomorphism, since the fibre at each $\tilde z\in\Igfr^b(\k)$ shows that the map $\Xfr^b_{m,g} \to  \Igfr^b\times \Mfr^b_{m,g}$ cannot factor through any proper closed formal subscheme of the target.
The equivariance with respect to the Hecke $\Gsf(\AA_f^p)$-action is clear from the same property for the isomorphism $\Xfr^b\riso  \Igfr^b\times \Mfr^b$, and we have already verified the equivariance with respect to the Hecke action of $G(\QQ_p)^+$. 
This  completes the proof.
\end{proof}

\subsection{Application of K\"unneth formula} \label{sect kuenneth}
We fix a complete algebraically closed extension $C$ of $\breve E$. We also fix an algebraic representation $\xi\colon \Gsf_{\QQbar_l} \to \mathsf{GL}(V)$ and denote by $\Lscr_\xi$ the corresponding automorphic \'etale sheaf on the Shimura variety (for a non-specified level). Recall that the $l$-adic cohomology with $\Lscr_\xi$-coefficients on the ``infinite-level'' has a natural $\Gsf(\AA_f)$-action by cohomological correspondence. To spell it out, for any $g\in\Gsf(\AA_f)$ and an open compact subgroup $\Ksf\subset \Gsf(\AA_f)$, we have a natural map $[g]:\Sh_{g\Ksf g^{-1}}(\Gsf,\Xsf)\riso\Sh_\Ksf(\Gsf,\Xsf)$ defining Hecke correspondences, and we have a natural isomorphism
\[
[g]^*\Lscr_\xi \cong \Lscr_\xi
\] 
of $l$-adic \'etale sheaves on $\Sh_{g\Ksf g^{-1}}(\Gsf,\Xsf)$, which satisfy the obvious cocycle relations with respect to the group structure of $\Gsf(\AA_f)$. Using this, we will identify $\Lscr_\xi$ and $[g]^*\Lscr_\xi$.

\begin{defnsub}
We set
\[
\Rrm\Gamma_\Ccal(\Mcal^b_{m,C},\QQbar_l)\coloneqq\Rrm\Gamma_{\pr_{m,g}\iv(\Ccal_{\Mfr^b_{m,g}})}(\Mcal^b_{m,C},\QQbar_l) 
\]
 for some $g\in G(\Qp)^+$ such that $m\geqslant e(g)$.
By Lemma~\ref{lem:Funct}, it follows that $\Rrm\Gamma_\Ccal(\Mcal^b_m,\QQbar_l)$ is independent of the choice of integral model $\Mfr^b_{m,g}$ (i.e., the choice of $g\in G(\Qp)^+$). 

We write $\Scal^b_{m}\coloneqq (\wh\Sscr^b_{m,g})^\ad_{\breve E}$, which is a finite \'etale covering of $\Scal^b\coloneqq (\wh\Sscr^b)^\ad_{\breve E}$ independent of the choice of $g\in G(\Qp)^+$. For any automorphic \'etale sheaf $\Lscr_\xi$, we set
\[
\Rrm\Gamma_\Ccal(\Scal^b_{m,C},\Lscr_\xi)\coloneqq\Rrm\Gamma_{\pr_{m,g}\iv(\Ccal_{\wh\Sscr^b_{m,g}})}(\Scal^b_{m,C},\Lscr_\xi),
\]
for some $g\in G(\Qp)^+$ such that $m\geqslant e(g)$. Similarly, $\Rrm\Gamma_\Ccal(\Scal^b_{m,C},\Lscr_\xi)$ does not depend on the choice of integral model or $g\in G(\Qp)^+$.
\end{defnsub}

The following lemma is now a straightforward consequence of Proposition~\ref{prop:VanCycSpecSeq} and the assumption above.
\begin{lemsub}\label{lem:HeckeRZ}
We have the following natural isomorphisms
\begin{eqnarray*}
\Rrm\Gamma_\Ccal(\Mcal^b_C,\pr_{m,\ast}\QQbar_l) &\cong \Rrm\Gamma_\Ccal(\Mcal^b_{m,C},\QQbar_l) &\cong \Rrm\Gamma_c((\Mfr^b_{m,g})^\red,\Rrm\Psi\QQbar_l); \\
\Rrm\Gamma_\Ccal(\Scal^b_{C},\pr_{m,\ast}\Lscr_\xi) &\cong \Rrm\Gamma_\Ccal(\Scal^b_{m,C},\Lscr_\xi) &\cong \Rrm\Gamma_c(\Ssb^b_{m,g},\Rrm\Psi\Lscr_\xi);\\
\Rrm\Gamma_\Ccal(\Scal_{C},\pr_{m,\ast}\Lscr_\xi) &\cong \Rrm\Gamma_\Ccal(\Scal_{m,C},\Lscr_\xi) &\cong \Rrm\Gamma_c(\Ssb_{m,g},\Rrm\Psi\Lscr_\xi),
\end{eqnarray*}
for any $b$, $m$ and $g$. Here, $\Scal$ and $\Scal_{m}$ are the analytic generic fibre of the $p$-adic completion of $\Sscr$ and $\Sscr_{m,g}$, and $\Ssb_{m,g} = (\Sscr_{m,g})_{\k}$.

Furthermore, if we set $m'\coloneqq m+e(g)$ for $g\in G(\Qp)^+$, then $[g]:\Mcal^b_{m'}\riso \Mcal^b_{g\iv\Ksf_p(m') g}\thra \Mcal^b_{m}$ induces a natural map
\[
[g]^\ast :\Rrm\Gamma_\Ccal(\Mcal^b_{m,C},\QQbar_l) \ra \Rrm\Gamma_\Ccal(\Mcal^b_{m',C},\QQbar_l).
\]	
Similarly, we also obtain
\[
[g]^\ast :\Rrm\Gamma_\Ccal(\Scal^b_{m,C},\Lscr_\xi) \ra \Rrm\Gamma_\Ccal(\Scal^b_{m',C},\Lscr_\xi).
\]	
\end{lemsub}
\begin{proof}
Recall that we have natural isomorphisms
\[
\Rrm\Psi(\pr_{m,\ast}\Lscr)\cong \Rrm\pr^\red_{m,g,\ast}(\Rrm\Psi\Lscr)\quad \text{(}\textit{cf. }\text{\cite[Corollary~2.3(ii)]{Berkovich:VanishingCyclesFormal2})},
\]
over $(\Mfr^b_{m,g})^{\red}$ and  $\Ssb_{m,g}$, where $\Lscr = \QQbar_l$ or $\Lscr_\xi$. (Note that the generic fibre of $\pr_{m,g}$ is 
the natural projection $\pr_m:\Mcal^b_m\to\Mcal^b$ or $\Scal_m\to \Scal$, which is a finite morphism. So we have $\Rrm\pr_{m,\ast}\Lscr_\xi = \pr_{m,\ast}\Lscr_\xi$.) Now, the lemma follows directly from  Proposition~\ref{prop:VanCycSpecSeq}.
\end{proof}

Recall that we have a ``surjective $\QIsog_G(\XX)$-pretorsor'' $\pi_\infty:\Xfr^b =\Igfr^b\times \Mfr^b \ra \wh\Sscr^b$. We may pull back this morphism to $\pi_\infty:\Xfr^b_{m,g}\cong\Igfr^b\times \Mfr^b_{m,g} \ra \wh\Sscr^b_{m,g}$ for any $g\in G(\Qp)^+$ and $m\geqslant e(g)$.
\begin{propsub} \label{prop:ProductDecompConstantSheaf}
For any automorphic \'etale sheaf $\Lscr_\xi$, we have a natural isomorphism
\[
\pi_\infty^\ast \Rrm\Psi_{\wh\Sscr^b_{m,g}}\Lscr_\xi \cong \Lscr_\xi \boxtimes^\Lrm \Rrm\Psi_{\Mfr^b_{m,g}}\QQbar_l
\]
This isomorphism respects the natural $J_b(\Qp)$-actions, Weil descent data, and the Hecke action of $G(\Qp)^+$ on both sides.
\end{propsub}
\begin{proof}
We first deduce the proposition for  $\Lscr_\xi = \QQbar_l$ and $\Gsf'=\mathsf{GSp}_n$ with hyperspecial level at $p$, which essentially follows from \cite[\S7,~Proposition~21]{Mantovan:Foliation}. To explain, let us introduce some notation. Recall that $\Mfr^{b'} $ is the union of certain quasi-compact closed formal subschemes $ \Mfr^{b';r,d}$,  called ``truncated Rapoport-Zink spaces'' (\emph{cf.} \cite[pp.~590-591]{Mantovan:Foliation}), and for any $g\in G'(\Qp)^+$ and $m\geqslant e(g)$ we let $\Mfr^{b';r,d}_{m,g}\subset \Mfr^{b'}_{m,g}$ denote the pull back of $\Mfr^{b';r,d}$ by the natural projection. Note also that
$\ol\Ig^{b'}$ is the perfection of $\varprojlim_{r'}J_{r'}$ where $J_{r'}$ are certain finite \'etale covering of the central leaf. Then $\pi_\infty$ restricted to $\ol\Ig^{b'}\times(\Mfr^{b';r,d}_{m,g})^\red$ factors as
\[\xymatrix@1{
\ol\Ig^{b'}\times(\Mfr^{b';r,d}_{m,g})^\red \ar@{->>}[r]& J_{r'}^\perf\times(\Mfr^{b';r,d}_{m,g})^\red \ar[r]^-{\pi}& \ol \Sscr^{b'}_{m,g}}
\]
for some large enough $r'$. Now, \cite[\S7,~Proposition~21]{Mantovan:Foliation} shows the following isomorphism in the derived category of \'etale sheaves over $J_{r'}^\perf\times(\Mfr^{b';r,d}_{m,g})^\red$:
\[
\pi^\ast  \Rrm\Psi_{\wh\Sscr^b_{m,g}}\QQbar_l \cong \QQbar_l \boxtimes^\Lrm \Rrm\Psi_{\Mfr^{b;r,d}_{m,g}}\QQbar_l
\]
for $r'\gg1$.\footnote{Although our integral models are the normalisations of the integral models in \cite{Mantovan:Foliation}, we can pull back the statement of \cite[\S7,~Proposition~21]{Mantovan:Foliation} via the normalisation maps. Indeed, the formation of the formal nearby cycle sheaves commute with proper base change of formal models; \emph{cf.} \cite[Corollary~2.3(ii)]{Berkovich:VanishingCyclesFormal2}.} This implies the desired decomposition of $\pi_\infty^\ast  \Rrm\Psi_{\wh\Sscr^b_{m,g}}\QQbar_l$ when restricted to $\ol\Ig^{b'}\times(\Mfr^{b';r,d}_{m,g})^\red$ (by pulling back via $\ol\Ig^{b'} \ra J_{r'}^\perf$), so we obtain the desired decomposition by increasing $\Mfr^{b';r,d}_{m,g}$.

Now, let us  handle the case with $\Lscr_\xi = \QQbar_l$ allowing any $\Gsf$. Let $g\in G(\Qp)^+$ (not just in $G'(\Qp)^+$). 
Let $\Xfr^{b,\diamond}_{m,g}$  denote the pull back of $\Xfr^{b'}_{m,g}$ by $\Sscr_{m,g}\ra\Sscr'_{m,g}$. Then $\Xfr^b_{m,g}$ is a closed union of connected components of $\Xfr^{b,\diamond}_{m,g}$, and $\pi_\infty$ on $\Xfr^b_{m,g}$ coincides with the restriction of $\pi_\infty$ on $\Xfr^{b'}_{m,g}$. The same statement also holds for $\Igfr^b$ and $\Mfr^b_{m,g}$. Since the formation of formal nearby cycle sheaves commute with proper base change of formal schemes (such as $\wh\Sscr^b_{m,g}\ra \wh\Sscr'^{b'}_{m,g}$), we obtain the following isomorphism in the derived category of \'etale sheaves over $\Xfr^{b,\diamond}_{m,g}$
\[
\pi_\infty^{\diamond,\ast} \Rrm\Psi_{\wh\Sscr^b_{m,g}}\QQbar_l \cong \QQbar_l \boxtimes^\Lrm \Rrm\Psi_{\Mfr^{b,\diamond}_{m,g}}\QQbar_l
\]
where $\pi_\infty^\diamond:\Xfr^{b,\diamond}_{m,g}\cong \Igfr^{b,\diamond}\times\Mfr^{b,\diamond}_{m,g} \ra \wh\Sscr^b_{m,g}$ restricting $\pi_\infty:\Xfr^{b'}_{m,g}\ra \wh\Sscr'^{b'}_{m,g}$. Now, restricting this isomorphism over the formal subscheme $\Xfr^b_{m,g} \cong \Igfr^b\times\Mfr^b_{m,g}$, we obtain the desired isomorphism.

We deduce the general case (for any automorphic \'etale sheaf $\Lscr_\xi$) from the the case for $\Lscr_\xi = \QQbar_l$. For this, we briefly recall the definition of $\Lscr_\xi$. Fix an $l$-adic field $K$ such that $\xi$ descends to an algebraic representation $\xi_0\colon \Gsf_K \to \mathsf{GL} (V_0)$. If $\Ksf^p = \Ksf^{l,p}\Ksf_l$ with  $\Ksf^{l,p}\subset\Gsf(\AA_f^{l,p})$ and $\Ksf_l \subset \Gsf(\QQ_l)$, then we consider the pro-Galois cover
\[
 \Sscr_{m,g,\Ksf^{l,p}} \coloneqq \varprojlim_{\Ksf_l' \subset \Ksf_l} \Sscr_{m,g,\Ksf_l'\Ksf^{l,p}}
\]
of $\Sscr_{m,g}$ with Galois group $\Ksf_l$, and define
\[
 \Lscr_0 \coloneqq \underline{V_0} \times^{\Ksf_l} \Sscr_{m,g,\Ksf^{l,p}}.
\]
Then $\Lscr_0$ is a lisse pro-\'etale sheaf over $\Sscr_{m,g}$ such that $\Lscr_\xi = \Lscr_0 \otimes_K \QQbar_l$ by definition. Since $\pi$ commutes with the Hecke action, the pull-back of $\pi_\infty^\ast\Lscr_\xi$ to $\Xfr^b_{m,g,\Ksf^{l,p}} \coloneqq \varinjlim_{\Ksf_l' \subset \Ksf_l} \Xfr^b_{m,g,\Ksf_l'\Ksf^{l,p}}$ is isomorphic to the constant sheaf $V$ such that the descent datum is given by $\restr{\xi}{\Ksf_l}\colon \Ksf_l \to \GL(V)$. By the isomorphism (\ref{eqn:DrinfeldModels:ProdStr}) this coincides with the descent datum for $\restr{\Lscr_\xi}{\Igfr^b} \boxtimes \restr{\QQbar_l}{\Mfr^b_{m,g}}$, thus
\[
 \pi_\infty^\ast\Lscr_\xi \cong \restr{\Lscr_\xi}{\Igfr^b} \boxtimes \restr{\QQbar_l}{\Mfr^b_{m,g}}.
\]
Since $\Lscr_\xi$ is lisse over $\wh\Sscr^b_{m,g}$, we have $\Rrm\Psi_{\wh\Sscr^b_{m,g}}\Lscr_\xi \cong \restr{\Lscr_\xi}{\Ssb} \otimes^\Lrm \Rrm\Psi_{\wh\Sscr^b_{m,g}}\QQbar_l$. Altogether we obtain
  \begin{align*}
   \pi_\infty^\ast \Rrm\Psi_{\wh\Sscr^b_{m,g}}\Lscr_\xi &\cong \pi_\infty^\ast\Lscr_\xi \otimes^\Lrm \pi_\infty^\ast R\Psi_{\wh\Sscr^b_{m,g}}\QQbar_l \\
   & \cong (\restr{\Lscr_\xi}{\bar\Ig^b} \boxtimes^\Lrm \restr{\QQbar_l}{\Mfr^b_{m,g}}) \otimes^\Lrm (\restr{\QQbar_l}{\bar\Ig^b} \boxtimes^\Lrm R\Psi_{\Mfr^b_{m,g}} \QQbar_l) \\
   & \cong \restr{\Lscr_\xi}{\bar\Ig^b} \boxtimes^\Lrm R\Psi_{\Mfr^b_{m,g}} \QQbar_l.
  \end{align*}
\end{proof}

\begin{propsub}\label{prop:Ext}
For each $m$, we have the following natural $W_E$-equivariant isomorphism
\begin{align*}
\Rrm\Gamma_\Ccal(\Scal^b_{m}, \Lscr_\xi) &\cong \Rrm\Gamma_c(\bar\Ig^b,\Lscr_\xi)\otimes^\Lrm_{\Hcal_{\QQbar_l}(J_b(\Qp))}\Rrm\Gamma_\Ccal (\Mcal^b_m,\QQbar_l)  \\
&\cong \RHom_{J_b(\Qp)}(\Rrm\Gamma_c(\Mcal^b_m,\QQbar_l), \Rrm\Gamma_c(\bar\Ig^b,\Lscr_\xi))(-d)[-2d],
\end{align*}
where $d$ is the dimension of $\Mcal^b$.
\end{propsub}
\begin{proof}
The first isomorphism is straightforward from the K\"unneth formula (\emph{cf.} Corollary~\ref{cor product structure kuenneth}). To see the second isomorphism, we apply the duality theorem of Mieda's (\emph{cf.} Theorem~\ref{thm:Poincare}):
\[
\Rrm\Gamma_\Ccal (\Mcal^b_m,\QQbar_l) \cong \Rrm\Dbf_{J_b(\QQ_p)}(\Rrm\Gamma_c(\Mcal^b_m,\QQbar_l))(-d)[-2d].
\]
Since each $\coh i_c(\Mcal^b_m,\QQbar_l)$ is finitely generated as a $J_b(\Qp)$-representation (\emph{cf.} Remark~\ref{rmk:Poincare}), the desired isomorphism reduces to Lemma~\ref{lem:Ext} using the Cartan-Eilenberg spectral sequence (\emph{cf.} \cite[Proposition~5.7.6]{Weibel:IntroHA} and its cohomological analogue).

\end{proof}

\begin{defnsub}

For $\rho\in\Rep_{J_b(\Qp)}(\QQbar_l)$, we define the following element of the Grothendieck group of smooth $G(\QQ_p)\times W_E$-representations over $\QQbar_l$
\begin{align*}
\Mant_{b,\mu}(\rho)&\coloneqq \sum_{i=0}^{2d} (-1)^i\varinjlim_m \Ext^{-2d+i}_{J_b(\Qp)} (\Rrm\Gamma_c(\Mcal^b_m,\QQbar_l),\rho)(-d)\\
&= \sum_{i,j=0}^{2d} (-1)^{i+j}\varinjlim_m \Ext^{i}_{J_b(\Qp)} (\coh{j}_c(\Mcal^b_m,\QQbar_l),\rho)(-d)
,
\end{align*}
where $d$ is the dimension of $\Mcal^b$.
(Here, the $W_E$-action is induced from the Weil descent datum on each $\Mcal^b_m$, and the smooth $G(\Qp)$-action is given by the Hecke correspondence on the tower $\{\Mcal^b_m \}_m$.)
\end{defnsub}

Let us write
\begin{align*}
 \Sh_{\infty} &\coloneqq \varprojlim_{\Ksf \subset \Gsf(\AA_f)} \Sh_{\Ksf} \\
 \bar\Ig_{\infty}^b &\coloneqq \varprojlim_{\Ksf^p \subset \Gsf(\AA_f^p)} \bar\Ig_{\Ksf^p}^b.
\end{align*}
Note that contrary to $\Sh_m$ the variety $\Sh_\infty$ considers also deeper level away from $p$. Their cohomology is given by
\begin{align*}
\coh i_c(\Sh_{\infty,C}, \Lscr_\xi)&= \varinjlim_{\Ksf\subset \Gsf(\AA_f)}\coh i_c(\Sh_{\Ksf, C}, \Lscr_\xi)\\
\coh i_c(\bar\Ig^b_\infty, \Lscr_\xi)&=\varinjlim_{\Ksf^p\subset \Gsf(\AA_f^p)}\coh i_c(\bar\Ig^b_{\Ksf^p}, \Lscr_\xi),
\end{align*}
where the former has a smooth action of $\Gsf(\AA_f)\times W_E$, and the latter has a smooth action of $\Gsf(\AA_f^p)\times J_b(\QQ_p)$. Note that this definition is equivalent to the one given in Definition~\ref{definition:lowershriek} by Remark~\ref{rmk proper pushforward}(3).

\begin{corsub}\label{cor:MantNewtStr}
The   action of $G(\Qp)^+$ on $\coh i_\Ccal(\Scal^b_{\infty,C},\Lscr_\xi)$ naturally extends to a smooth action  of $G(\Qp)$, and with respect to this action, we obtain, from  Proposition~\ref{prop:Ext}, an equality in the Grothendieck group of smooth $\Gsf(\AA_f)\times W_E$-representations over $\QQbar_l$
\[\sum_{i=0}^{2d}(-1)^i\coh i_\Ccal(\Scal^b_{\infty,C},\Lscr_\xi)= \sum_{j=0}^{2c}(-1)^j\Mant_{b,\mu}(\coh j_c(\bar\Ig^b_{\infty},\Lscr_\xi)),\]
where $d$ and $c$ are the dimensions of the Shimura variety $\Scal$ and the central leaf $\bar C^b$, respectively. (Recall that $\bar\Ig^b_{\infty}$ is the perfection of some pro-\'etale cover of $\bar C^b$.)
\end{corsub}

\begin{proof}

By taking the direct limit with respect to $m$ of the isomorphism in Proposition~\ref{prop:Ext} (with fixed prime-to-$p$ level $\Ksf^p$), we obtain
\begin{equation}\label{eqn:MantNewtStr}
\varinjlim_m\coh i_\Ccal(\Scal^b_{m,C},\Lscr_\xi) = \varinjlim_m \Ext^{-2d+i}_{J_b(\Qp)} (\Rrm\Gamma_c(\Mcal^b_m,\QQbar_l),\Rrm\Gamma_c(\bar\Ig^b_{\Ksf^p},\Lscr_\xi))(-d),
\end{equation}
which is clearly $G(\Qp)^+\times W_E$-equivariant. But since the action of $G(\Qp)^+$ on the left hand side extends to $G(\Qp)$-action, it follows that $G(\Qp)^+$-action on the right hand side is invertible. Therefore, the action of $G(\QQ_p)^+$ on the right hand side naturally extends to $G(\QQ_p)$-action in such a way that the isomorphism (\ref{eqn:MantNewtStr}) is $G(\QQ_p)$-equivariant. (Indeed, note that $G(\Qp)$ is generated by $G(\Qp)^+$ and its inverse, as $G(\Qp)$ is generated by $G(\Qp)^+$ and $p^\ZZ$.)

Note that the right hand side of (\ref{eqn:MantNewtStr}) is equal to $\sum_{j}(-1)^j\Mant_{b,\mu}(\coh j_c(\bar\Ig_{\Ksf^p},\Lscr_\xi))$ in the Grothendieck group of smooth representations of $J_b(\QQ_p)\times W_E$. Now, the corollary follows from taking the direct limit of (\ref{eqn:MantNewtStr}) with respect to $\Ksf^p\subset \Gsf(\AA_f^p)$.
\end{proof}

We need the following additional axiom for the integral models $\{\Sscr_m\}$  for our main result (Corollary~\ref{cor:MantShVar}):
\begin{axiom}\label{axiom:LanStroh}
For each $m$, we assume that the following natural $W_E$-equivariant map
\begin{equation}\label{eqn:LanStroh}
\coh i_\Ccal(\Scal_{m},\Lscr_\xi) \cong \coh i_c(\Ssb_{m},\Rrm\Psi\Lscr_\xi) \to \coh i_c(\Sh_{m},\Lscr_\xi)
\end{equation}
is an \emph{isomorphism}. In particular, we have a natural isomorphism
\[
\coh i_\Ccal(\Scal_\infty,\Lscr_\xi) \riso \coh i_c(\Sh_\infty,\Lscr_\xi)
\]
respecting the natural  action of $W_E\times\Gsf(\AA_f^{p})\times G(\QQ_p)^+$.
\end{axiom}

\begin{rmksub}
The verification of Axiom~\ref{axiom:LanStroh} essentially boils down to the existence of ``nice compactifications'' for $\Sscr_m$ for any $m$. (See \cite[Proposition~2.2]{LanStroh:NearbyCycles} and \cite[Proposition~2.4]{LanStroh:NearbyCycles2} for the properties of compactifications sufficient to ensure the natural map (\ref{eqn:LanStroh}) to be an isomorphism.)

In more practical terms Axiom~\ref{axiom:LanStroh} holds for $\{\Sscr_m\}$ if the Shimura datum $(\Gsf,\Xsf)$ satisfies one of the following:
\begin{enumerate}
\item The associated Shimura variety $\Sh_\Ksf(\Gsf,\Xsf)$ is proper for any $\Ksf\subset \Gsf(\AA_f)$.
\item There exists a PEL datum giving rise to $(\Gsf,\Xsf)$. Here, we do not assume that $\Gsf\otimes\QQ_p$ is unramified, and we allow the associated Shimura varieties to be non-compact.
\end{enumerate}

To deduce Axiom~\ref{axiom:LanStroh} in the first case, recall that if $\Sh$ is proper then  $\Sscr_m$ is also proper for any $m$; indeed, the main result of \cite{MadapusiPera:Compactification} implies that $\Sscr$ is proper, and  $\Sscr_m$ is finite over $\Sscr$. Therefore, the map \ref{eqn:LanStroh}) is an isomorphism by the nearby cycles spectral sequence for proper schemes over a complete discrete valuation ring.

If $(\Gsf,\Xsf)$ is of PEL type, then $\Sscr_m$ satisfies the assumption (Nm) in  \cite[Assumption~2.1]{LanStroh:NearbyCycles}. Therefore it was proved in \cite[Corollary~5.20]{LanStroh:NearbyCycles} that the map (\ref{eqn:LanStroh}) is an isomorphism so Axiom~\ref{axiom:LanStroh} holds in the second case.

From the recent progress in arithmetic compactification of Shimura varieties (\emph{cf.} \cite{MadapusiPera:Compactification}), it seems possible to be able to show that the normal flat integral models $\Sscr_m$ have ``nice compactifications'' for any $m$ so that  Axiom~\ref{axiom:LanStroh} holds (at least for the ``tamely ramified Hodge-type case'').
\end{rmksub}

\begin{corsub}\label{cor:MantShVar}
Assume that in addition to Axiom~\ref{axiom RZ}, Axiom~\ref{axiom:LanStroh} also holds true. 
Then the following equality holds in the Grothendieck group of smooth $(G(\AA_f)\times W_E)$-representations over $\QQbar_l$:
\[
\sum_i(-1)^i\coh i_c(\Sh_{\infty,C}, \Lscr_\xi) =\sum_{b\in B(G,\{\mu\})}\sum_j(-1)^j\Mant_{b,\mu}(\coh j_c(\bar\Ig^b_\infty,\Lscr_\xi)),
\]
where $B(G,\{\mu\})\subset B(G)$ is the set of neutral acceptable $\sigma$-conjugacy class with respect to the geometric conjugacy class of $\mu:\GG_m\ra G_{\Qpbar}$ associated to the Shimura datum, and the $\Gsf(\AA_f^p)$-action on the righ hand side is given by the Hecke $\Gsf(\AA_f^p)$-action on the Igusa towers $\{\bar\Ig^b_{\Ksf^p}\}$.
\end{corsub}
\begin{proof}
By excision sequence of compact support cohomology on $\Ssb_{m}$ and Proposition~\ref{prop:VanCycSpecSeq}, we have the following equality in the Grothendieck group of $W_E$-representations over $\QQbar_l$:
\begin{multline*}
\sum_i(-1)^i\coh i_\Ccal(\Scal_{m},\Lscr_\xi) = \sum_i(-1)^i\coh i_c(\Ssb_{m},\Rrm\Psi\Lscr_\xi)\\
= \sum_{b\in B(G,\{\mu\})}\sum_i(-1)^i\coh i_c(\Ssb^b_{m},\Rrm\Psi\Lscr_\xi) = \sum_{b\in B(G,\{\mu\})}\sum_i(-1)^i\coh i_\Ccal(\Scal^b_{m},\Lscr_\xi).
\end{multline*}

Now, by taking the direct limit with respect to $m$ and $\Ksf^p\subset\Gsf(\AA_f^p)$, we obtain the following equalities in the Grothendieck group of smooth $(G(\AA_f)\times W_E)$-representations over $\QQbar_l$:
\begin{multline*}
\sum_i(-1)^i\coh i_c(\Sh_{\infty},\Lscr_\xi) = \sum_i(-1)^i\coh i_\Ccal(\Scal_{\infty},\Lscr_\xi) \\= \sum_{b\in B(G,\{\mu\})}\sum_i(-1)^i\coh i_\Ccal(\Scal^b_{\infty},\Lscr_\xi)=\sum_{b\in B(G,\{\mu\})}\Mant_{b,\mu}(\Rrm\Gamma_c(\bar\Ig^b_\infty,\Lscr_\xi)),
\end{multline*}
where the first equality follows from Axiom~\ref{axiom:LanStroh} and the last equality follows from Corollary~\ref{cor:MantNewtStr}.
\end{proof}

Let us write
\begin{equation}
\coh j_c(\Mcal^b_{\infty,C}, \QQbar_l)\coloneqq \varinjlim_m\coh j_c(\Mcal^b_{m, C}, \QQbar_l),
\end{equation}
which is a smooth representation of $G(\Qp)\times J_b(\Qp)\times W_E$ over $\QQbar_l$.

We end this section by rewriting the functors $\Mant_{b,\mu}(\rho)$ directly in terms of the cohomology of  Rapoport-Zink spaces with infinite level.  The immediate obstacle is that the natural $G(\Qp)$-action on $\Ext^i_{J_b(\Qp)}(\coh j_c(\Mcal^b_{\infty,C}, \QQbar_l),\rho)$ is not smooth in general. We fix this issue by working with ``smoothened extensions'' $\Ecal^i(\bullet,\bullet)$ introduced in Definition~\ref{def:Ecal} with $G=G(\Qp)$ and $J=J_b(\Qp)$.

By Proposition~\ref{prop:Infty2Fin}, $\Ecal^{j}(\coh{j'}_c(\Mcal^b_{\infty,C},\QQbar_l),\rho)$ is the $\QQbar_l$-submodule of $G(\Qp)$-smooth elements in $\Ext^{j}_{J_b(\Qp)}(\coh{j'}_c(\Mcal^b_{\infty,C},\QQbar_l),\rho)$.
\begin{corsub} \label{cor Mantovan's formula}
For any $\rho\in\Rep_{\QQbar_l}(J)$, we have the following equality in the Grothendieck groups of smooth $G(\Qp)\times W_E$-representations over $\QQbar_l$:
\[\Mant_{b, \mu}(\rho) = \sum_{j,j'}(-1)^{j+j'} \Ecal^{j}(\coh{j'}_c(\Mcal^b_{\infty,C},\QQbar_l),\rho)(-d).\]
Finally, for the Shimura varieties considered in Corollary~\ref{cor:MantShVar}, we have the following equality in the Grothendieck groups of smooth $\Gsf(\AA_f)\times W_E$-representations over $\QQbar_l$:
\begin{multline*}
\sum_i(-1)^i\coh i_c(\Sh_{\infty,C}, \Lscr_\xi) \\=\sum_{\substack{b\in B(G,\{\mu\})\\j,j',j''}}  (-1)^{j+j'+j''}\Ecal^{j}(\coh{j'}_c(\Mcal^b_{\infty,C},\QQbar_l),\coh{j''}_c(\bar\Ig^b_{\infty},\Lscr_\xi))(-d).
\end{multline*}
\end{corsub}

%
%
%

 \appendix

 \section{Homomorphisms of constant $F$-Crystals}\label{appendix RR}

 In order to show that the Igusa variety associated to a Barsotti-Tate group $(\XX,(t_\alpha))$ is a flat closed subvariety of the Igusa variety associated to $\XX$, we will need the proposition below. To improve readability, we use the following notation. For any profinite set $S$ and topological ring $R$, we denote
  \[
   R^S \coloneqq \Map_{\rm cont}(S,R).
  \]
  Unless stated explicitly otherwise, we will assume that $R$ is equipped with the discrete topology.

 \begin{prop} \label{prop homomorphisms of F-crystals}
  Let $M_1, M_2$ be $F$-crystals over an algebraically closed field $k$. Denote $H = \Hom(M_1,M_2)$ with the $p$-adic topology. We denote by the same symbol the associated $k$-ring scheme $\Spec C^\infty(H,k)$, which may also be descibed via the functor of points
  $
   H(S) = H^{\pi_0(S)}.
  $
  Then $H$ represents the functor
  \begin{align*}
   ({\rm PerfSch}_k) &\to ({\rm Set}) \\
   S & \mapsto \Hom(M_{1,S},M_{2,S})
  \end{align*}
 \end{prop}
 \begin{proof}
  The proof is identical to the proof of \cite[Lemma~3.9]{RapoportRichartz:Gisoc}, except for the last sentence. Here Rapoport and Richartz use $W(R)_\QQ^\sigma = \QQ_p$, i.e.\ require $\Spec R$ to be connected. The same argument still works for us if we replace it by $W(R)_\QQ^\sigma = \QQ_p^{\pi_0(\Spec R)}$ (w.r.t.\ the $p$-adic topology), which is proven in Lemma~\ref{lemma Frobenius fixed elements} and Lemma~\ref{lemma product of Witt vectors} below.
 \end{proof}

 For any $\FF_q$-algebra $R$ we have a canonical embedding $\FF_q^{\pi_0(\Spec R)} \mono R$ given as follows. For any $f \in \FF_q^{\pi_0(\Spec R)}$ consider the clopen sets $U_{f,x} \subset \Spec R$ given by the preimage of $x \in \FF_q$ w.r.t.\ the composition $\Spec R \to \pi_0(\Spec R) \stackrel{f}{\to} \FF_q$. Then there exists a unique idempotent $e_{f,x} \in R$ such that $U_{f,x} = D(e_{f,x})$. We define the image of $f$ in $R$ to be $\sum_{x \in \FF_q} x \cdot e_{f,x}$.

 \begin{lem} \label{lemma Frobenius fixed elements}
  For any $\FF_q$-algebra $R$ one has $\{r \in R \mid r^q = r\} = \FF_q^{\pi_0(\Spec R)}$.
 \end{lem}
 \begin{proof}
  The inclusion ``$\supseteq$'' is obvious. Thus let $r \in R$ with $r^q - r = 0$. Then $\Spec R = \coprod_{x \in \FF_q} V(r-x)$, thus $r$ defines an locally constant function
  \begin{align*}
   \varphi: \Spec R &\to \FF_q \\
   \pfr &\mapsto r \mod \pfr.
  \end{align*}
  Let $f \in \FF_q^{\pi_0(\Spec R)}$ be the corresponding element. Then $r-f$ is nilpotent, say $(r-f)^{q^n} = 0$. Now
  \[
   r-f = r^{q^n} - f^{q^n} = (r-f)^{q^n} = 0.
  \]
 \end{proof}

 \begin{lem} \label{lemma product of Witt vectors}
  Let $R$ be a perfect ring of characteristic $p$ and $S$ be a profinite set. We have
  \[
   W(R^S) = W(R)^S
  \]
  where we consider the $p$-adic topology on $W(R)$.
 \end{lem}
 \begin{proof}
 Recall that $W(R)$ is uniquely characterised by the property that it is torsion-free, $p$-adically complete such that $W(R)/p = R$. As a direct consequence we obtain that for any directed system $(R_\lambda)$ of perfect rings $W(\varinjlim R_\lambda)$ is the $p$-adic completion of $\varinjlim W(R_\lambda)$. Now let $S = \varprojlim S_\lambda$ with $S_\lambda$ finite. We obtain
 \begin{align*}
  W(R^S) &= W(\varinjlim\, R^{S_\lambda}) \\
         &= (\varinjlim\,  W(R^{S_\lambda}))^{\wedge\, p-ad} \\
         &= (\varinjlim\, W(R)^{S_\lambda})^{\wedge\, p-ad} \\
         &= \Map_{\rm loc.~const.}(S,W(R))^{\wedge\,p-ad} \\
         &= W(R)^S.
 \end{align*}
 \end{proof}

 \section{The K\"unneth formula for pro\'etale torsors}\label{appendix Kuenneth}

 The cohomology of compact support is normally only defined for schemes of finite type. We define it for a slightly larger class of schemes, as we also want to calculate it for pro-\'etale torsors over varieties.

 \begin{deflem} \label{definition:lowershriek}
  Let $f\colon P \rightarrow S$ be a morphism of schemes where $S$ is noetherian and $f$ factorises as $P \xrightarrow{g} X \xrightarrow{h} S$ where $h$ is locally of finite type and $g$ is integral. We define the derived pushforward with compact support for \'etale sheaves on $P$  by
  \[
   \Rrm f_! \coloneqq \Rrm h_! \circ \Rrm g_\ast
  \]
  This definition does not depend on the choice of factorisation above.
 \end{deflem}
 \begin{proof}
  Let $P \xrightarrow{g} X \xrightarrow{h} S$ and $P \xrightarrow{g'} X' \xrightarrow{h'} S$ be two factorisations as above. If $g'$ factorises as $P \xrightarrow{g} X \xrightarrow{g''} X'$, then $g''$ is finite and thus
  \[
   \Rrm h'_! \circ \Rrm g'_\ast = \Rrm h'_! \circ \Rrm g''_\ast \circ \Rrm g_\ast = \Rrm h'_! \circ \Rrm g''_! \circ \Rrm g_\ast = \Rrm h_! \circ R g_\ast
  \]
  In general, since
  \[
   Rh_! = \varinjlim_{U \subset X \textnormal{ qc open}}\, Rh_{U !}
  \]
  where $h_U$ is the composition of the open embedding $U \mono X$ with $h$, we may assume that $h$ and $h'$ are of finite type. Writing $P = \varprojlim X_\lambda$, where $X_\lambda$ is finite over $X$, we see that $g'$ factors over some $X_\lambda$ since $X'$ is of finite presentation over $S$. Replacing $X$ by $X_\lambda$, we are in the situation discussed above.
 \end{proof}

 In the case where $S = \Spec k$, we may simply write $\Rrm \Gamma_c(X,-)$ instead of $\Rrm f_!$ and $H_c^i(X,-)$ instead of $\Rrm^i f_!$.

 \begin{rmk} \label{rmk proper pushforward}
  For special cases of $f$, the definition of $\Rrm f_!$ simplifies.
  \begin{subenv}
  \item If $f:X\to S=\Spec k$ is locally of finite type, then this definition of $\Rrm\Gamma_c(X,-)$ coincides with Definition~\ref{def:CompSuppCoh}.
  \item If $X$ is proper, $\Rrm f_! = \Rrm f_\ast$.
  \item Writing $P = \varprojlim X_\lambda$ with $X_\lambda \xrightarrow{g_\lambda} X$ finite, assume that the sheaf $\Lscr$ over the \'etale site of $P$ is defined as the pullback of a sheaf $\Lscr_\lambda$ over the \'etale site of $X_\lambda$. By \cite[09YQ]{AlgStackProj} we have $\Rrm f_! \Lscr = \varinjlim \Rrm f_{\lambda, !} \Lscr_\lambda$ where $f_\lambda = h \circ g_\lambda$. Thus the usual adhoc definition of cohomology with compact support for infinite level Shimura varieties coincides with the definition above. 
  \end{subenv}
 \end{rmk}

 We  fix a locally profinite group $J$ with a compact open pro-$p$-subgroup $K_0$ and denote by $dh$ the Haar measure such that $dh(K_0) = 1$.  From now on we assume that $f: P \to X$ is a pro\'etale $J$-torsor and that $\Lscr$ is a smooth $J$-equivariant abelian sheaf on $P$ such that $p$ is invertible in $\Lscr$. The $J$-action defines an effective descent datum on $\Lscr$; let $\Lscr_X$ denote the corresponding sheaf on $X$. One obtains the following explicit description of the stalks of $\Rrm^0 f_!\Lscr$.

 \begin{lem} \label{lem-csdi-stalks}
  Let $\xbar$ be a geometric point of $X$.
  \begin{subenv}
   \item At any geometric point $\xbar$ of $X$, the stalk of $\Rrm^0 f_! \Lscr$ at $\xbar$ is canonically isomorphic to $C_c^\infty(f^{-1}(\xbar),\Lscr_{X,\xbar})$. Moreover, we have a canonical $J$-action on $f_!\Lscr$ which is given on stalks by the natural $J$-action on $C_c^\infty(f^{-1}(\xbar),\Lscr_{X,\xbar})$.
   \item There exists a (noncanonical) morphism $\int_f: \Rrm^0 f_!\Lscr \to \Lscr_X$, which is given on stalks by
  \[
   C_c^\infty(f^{-1}(\xbar),\Lscr_{X,\xbar}) \to \Lscr_{X,\xbar},  \varphi \mapsto \int_{f^{-1}(\xbar)} \varphi dh.
  \]
  In particular, $\int_f$ is $J$-equivariant and the definition only depends on the choice of $dh$.
  \end{subenv}
 \end{lem}
 \begin{proof}
  For a compact open subgroup $K$ denote $X_K \coloneqq P/K$ and let $P \xrightarrow{g_K} X_K \xrightarrow{f_K} X$ be the canonical factorisation of $f$ and $\Lscr_K \coloneqq f_K^\ast \Lscr_X$ the descent of $\Lscr$ to $X_K$. 
  
  
  For $K \subset K_0$ let $g_{K,K_0}: X_{K} \to X_{K_0}$ be the canonical projection. Then
  \[
   g_{K_0\,\ast} \Lscr = \varinjlim_{K}\, g_{K,K_0\, \ast} \Lscr_{K} = \varinjlim_{K}\, g_{K,K_0\,\ast}g_{K,K_0}^\ast \Lscr_{K_0}.
  \]
   Since $K_0/K$ is finite, the morphism $g_{K,K_0}$ finite \'etale. Thus $g_{K,K_0\, \ast} = g_{K,K_0\,!}$ is the left adjoint of $g_{K,K_0}^\ast$. Hence the stalk at a geometric point $\overline{y} \in X_{K_0}$ equals (see e.g.\ \cite[Tag 0710]{AlgStackProj})
  \[
   (g_{K_0\,\ast} \Lscr)_{\ybar} = \varinjlim_{K} (g_{K,K_0\,\ast}g_{K,K_0}^\ast \Lscr_{K_0})_{\ybar} = \varinjlim_K \bigoplus_{g_{K,K_0}^{-1}(\ybar)} \Lscr_{K_0,\ybar} = C_c^\infty(g_{K_0}^{-1}(\ybar),\Lscr_{K_0,\ybar}).
  \]
  Since $J/K_0$ is discrete, $f_{K_0}$ is \'etale; thus by the same argument as above
  \[
   (f_{K_0\, !} g_{K_0\,\ast}\Lscr)_{\xbar}
   = \bigoplus_{\ybar \in f_{K_0}^{-1}(\xbar)} (g_{K_0\,\ast}\Lscr)_{\ybar}
   = C_c^\infty(f^{-1}(\xbar),\Lscr_{X,\xbar}).
  \]
  Since $f_!$ is functorial, the $J$-action on $\Lscr$ induces an $J$-action on $f_!\Lscr$. Tracing through the definitions, one easily checks that the action of $J$ on $(f_{!} \Lscr)_{\xbar} =  C_c^\infty(f^{-1}(\xbar),\Lscr_{X,\xbar})$ is indeed the canonical action induced by the $J$-action of $f^{-1}(\xbar)$.

  The morphism $\int_f$ is defined as follows. Let $\tau_{K/K_0} \coloneqq \frac{1}{[K_0:K]} \cdot \trace\colon g_{K,K_0\,\ast} \Lscr_K \to \Lscr_{K_0}$. On stalks, this morphism is given by
  \[
   \bigoplus_{\zbar \in g_{K,K_0}^{-1}(\ybar)} \Lscr_{K_0,\ybar} \rightarrow \Lscr_{K_0,\ybar}, (f_{\zbar}) \mapsto \frac{1}{[K_0:K]} \cdot \sum_{\zbar\in g_{K,K_0}^{-1}(\ybar)} f_{\zbar},
  \]
  i.e.\ taking the average. Hence $\tau \coloneqq \varinjlim\, \tau_{K,K_0}\colon g_{K_0\,\ast} \Lscr \to \Lscr_{K_0}$ is given on stalks by taking the integral $\int_{K_0} dh$. We define $\int_f$ as the concatination \[\Rrm^0f_!\Lscr \xrightarrow{f_{K_0\, !}(\tau)} f_{K_0\, !} \Lscr_{K_0} \xrightarrow{\trace} \Lscr_X.\] By the above calculation it indeed coincides with the integral over $f^{-1}(\xbar)$ on stalks.
 \end{proof}

 \begin{cor} \label{cor-cInd-stalks}
  $\Rrm^0 f_!\Lscr$ satisfies property $(\Pcal)$ of \cite[Def.~5.3]{Mantovan:UnitaryShVar}, i.e. $(\Rrm^0 f_!\Lscr)_{\xbar} \cong \cInd_1^J \Lscr_{X,\xbar}$.
 \end{cor}

 In this article, we will consider sheaves which have alongside the $J = J_b(\QQ_p)$-action an action of the Weil group $W = W_E$ of the $p$-adic completion of the Shimura field. We remark that the above construction still works if we consider $W$-equivariant sheaves for some profinite group $W$ acting on the underlying schemes. For simplicity we assume that $\Lscr$ is a ($W$-equivariant) sheaf of $\ZZ/l^r\ZZ$-modules for a prime $l \not= p$. Let $\Hcal_r = C_c^\infty(J_b(\QQ_p),\ZZ/l^r\ZZ)$ and denote by $\Lambda$ the module $\ZZ/l^r\ZZ$ as trivial $\Hcal_r$-representation.

 The following statements are a formal consequence of Corollary~\ref{cor-cInd-stalks}, and properties of $R\Gamma_c$ (cf.\  \cite[\S~7.2]{neupert:thesis}\cite[\S~5.3]{Mantovan:UnitaryShVar})

 \begin{prop}[{cf. \cite[Prop.~7.2.2]{neupert:thesis}, \cite[Thm.~5.11]{Mantovan:UnitaryShVar}}]
   $ \Rrm\Gamma_c(X,f_!\Lscr_X) = \Lambda \otimes^L_{\Hcal_r} \Rrm\Gamma_c(P,\Lscr)$
 \end{prop}

 \begin{thm}[{cf.\ \cite[Thm.~7.2.3]{neupert:thesis}, \cite[Thm.~5.13]{Mantovan:UnitaryShVar}}] \label{thm-kuenneth}
  Assume that we may write $P = P_1 \times P_2$ where $P_1$ and $P_2$ are $J$-equivariant and locally of finite presentation up to Galois cover and assume that
  \[
   \Lscr = \Escr_1 \boxtimes \Ecal_2
  \]
  where $\Escr_i$ is a smooth $J$-equivariant $\ZZ/l^r\ZZ$-sheaf on $P_i$ with continuous $W$-action. Then
  \[
   \Rrm\Gamma_c(X, \Lcal_X) = \Rrm\Gamma_c(P_1,\Escr_1) \otimes^L_{\Hcal_r} \Rrm\Gamma_c(P_2,\Escr_2)
  \]
 \end{thm}

 In particular, in the case where all the sheaves are constant \'etale sheaves associated to $\ZZ/l^r\ZZ$, we obtain the following result.

 \begin{cor}[{cf.\ \cite[Cor.~7.2.5]{neupert:thesis}, \cite[Thm.~5.14]{Mantovan:UnitaryShVar}}] \label{cor-kuenneth}
  We have a spectral sequence
  \[
   E_2^{p,q} = \bigoplus_{s+t = q} \Tor_{\Hcal_r}^p(H^s(P_1,\ZZ/l^r\ZZ), H^t(P_2,\ZZ/l^r\ZZ)) \Rightarrow H^{p+q}(X,\ZZ/l^r\ZZ).
  \]
 \end{cor}

 \bibliographystyle{amsalpha}
 \bibliography{bib}

\end{document}